\documentclass[reqno,10pt,a4paper]{amsart}
\usepackage[utf8]{inputenc}

\usepackage{amssymb,eucal,mathrsfs,array,setspace,geometry,enumitem,cite,tensor,amsmath,wrapfig,amscd,mathptmx,graphicx,bm,bbm,multirow,multicol,adjustbox}
\usepackage[dvipsnames]{xcolor}
\usepackage{tikz}
\usepackage[centertableaux]{ytableau}
\usepackage{txfonts}            
\usepackage{newtxtext}            
\usepackage{mathtools} 
\usepackage{stmaryrd} 

\usepackage{caption}
\usepackage{subcaption}

\usetikzlibrary{positioning}

\usepackage[colorlinks=true,citecolor=red,linkcolor=blue]{hyperref} 
\usepackage[capitalise,noabbrev]{cleveref}

\usepackage{tikz-cd}
\usetikzlibrary{decorations.pathmorphing}

\usetikzlibrary{knots}

\DeclareSymbolFont{largesymbols}{OMX}{zplm}{m}{n} 

\setitemize{leftmargin=*}   
\setenumerate{leftmargin=*, 
              label=\textup{(\arabic*)}} 

\let\originalleft\left     
\let\originalright\right
\renewcommand{\left}{\mathopen{}\mathclose\bgroup\originalleft}
\renewcommand{\right}{\aftergroup\egroup\originalright}

\newcolumntype{C}{>{$}c<{$}} 

\numberwithin{equation}{section}
\allowdisplaybreaks

\renewcommand{\Re}{\operatorname{Re}}

\renewcommand{\ge}{\geq}
\renewcommand{\le}{\leq}

\DeclarePairedDelimiter{\brac}{\lparen}{\rparen} 
\DeclarePairedDelimiter{\sqbrac}{\lbrack}{\rbrack} 
\DeclarePairedDelimiter{\set}{\lbrace}{\rbrace}
\newcommand{\st}{\mspace{5mu} : \mspace{5mu}} 
\DeclarePairedDelimiter{\abs}{\lvert}{\rvert}
\DeclarePairedDelimiter{\norm}{\lVert}{\rVert}
\DeclarePairedDelimiter{\ang}{\langle}{\rangle}
\DeclarePairedDelimiter{\powser}{\llbracket}{\rrbracket} 

\DeclarePairedDelimiterX{\comm}[2]{\lbrack}{\rbrack}{#1 , #2}  
\DeclarePairedDelimiterX{\acomm}[2]{\lbrace}{\rbrace}{#1 , #2} 
\DeclarePairedDelimiterX{\super}[2]{\lparen}{\rparen}{#1 \delimsize\vert \mathopen{} #2} 

\newcommand{\pair}[2]{\ang*{#1,#2}} 

\DeclareMathOperator{\id}{id}

\newcommand{\dd}{\mathrm{d}}   
\newcommand{\ii}{\mathfrak{i}} 
\newcommand{\ee}{\mathsf{e}}   

\newcommand{\wun}{\mathbf{1}}  

\DeclareMathOperator{\cspn}{span}
\newcommand{\spn}[1]{\cspn_{\CC}\set*{#1}}                    

\newcommand{\rspn}[1]{\cspn{_{\RR}\set*{#1}}}

\newcommand{\ra}{\rightarrow}

\newcommand{\isomto}{\stackrel{\cong}{\smash{\longrightarrow}\rule{0pt}{0.4ex}}}

\DeclareMathOperator{\ind}{Ind}
\newcommand{\Ind}[3]{\ind^{#1}_{#2} #3}

\DeclareMathOperator{\Hom}{Hom}
\newcommand{\Homgrp}[3]{\Hom_{#1}\brac*{#2,#3}}
\newcommand{\iHom}[2]{\underline{\Hom}\brac*{#1,#2}}

\DeclareMathOperator{\End}{End}
\newcommand{\Endgrp}[1]{\End{#1}}

\newcommand{\fld}[1]{\mathbb{#1}}    
\newcommand{\Mod}[1]{\mathcal{#1}}   
\newcommand{\VOA}[1]{\mathsf{#1}}    
\newcommand{\categ}[1]{\mathscr{#1}} 

\newcommand{\ZZ}{\fld{Z}}
\newcommand{\NN}{\fld{N}}

\newcommand{\RR}{\fld{R}}
\newcommand{\CC}{\fld{C}}

\newcommand{\UEA}[1]{\mathsf{U}\brac*{#1}}

\newcommand{\lhmd}[1]{\mathsf{HM}\brac*{#1}}

\newcommand{\hva}{\VOA{V}}
\newcommand{\hvoa}[1]{\hva\brac*{#1}}
\newcommand{\lvoa}[2]{\hva\brac*{#1,#2}}
\newcommand{\lvmd}[1]{\mathsf{VM}\brac*{#1}}

\newcommand{\Fock}[1]{\Mod{F}_{#1}}          
\newcommand{\LFock}[1]{\mathbb{F}_{#1}}      
\newcommand{\iFock}[1]{\mathbb{F}\sqbrac*{#1}} 

\DeclarePairedDelimiter{\ket}{\lvert}{\rangle}
\DeclarePairedDelimiterX{\braket}[2]{\langle}{\rangle}{#1 \delimsize\vert \mathopen{} #2}
\DeclarePairedDelimiterX{\bracket}[3]{\langle}{\rangle}{#1 \delimsize\vert \mathopen{} #2 \delimsize\vert \mathopen{} #3}

\DeclareMathOperator{\tr}{tr}

\newcommand{\fuse}{\mathbin{\boxtimes}}                                            
\DeclareMathOperator{\compspace}{COMP}
\newcommand{\comp}[2]{\compspace\brac*{#1,#2}}

\newcommand{\iop}[1]{\mathcal{Y}_{#1}}
\newcommand{\ifld}[3]{\iop{#1}\brac*{#2,#3}}
\DeclareMathOperator{\opp}{opp} 
\newcommand{\ityp}[3]{\binom{#3}{#1,#2}}
\newcommand{\ispc}[3]{\operatorname{I}\binom{#3}{#1,#2}}
\newcommand{\grispc}[3]{\operatorname{Gr}\binom{#3}{#1,#2}}
\newcommand{\tcycsymb}{\varepsilon}
\newcommand{\tcyc}[2]{\tcycsymb\brac*{#1,#2}}

\newcommand{\cft}{conformal field theory}

\newcommand{\cfts}{conformal field theories}

\newcommand{\va}{vertex algebra}

\newcommand{\voa}{vertex operator algebra}
\newcommand{\Voa}{Vertex operator algebra}
\newcommand{\ope}{operator product expansion}

\newcommand{\lhs}{left-hand side}
\newcommand{\rhs}{right-hand side}

\newcommand{\cfin}[1]{\(C_{#1}\)-cofinite}

\theoremstyle{plain}
\newtheorem{thm}{Theorem}[section]
\newtheorem{prop}[thm]{Proposition}
\newtheorem{lem}[thm]{Lemma}
\newtheorem{cor}[thm]{Corollary}
\newtheorem*{thm*}{Theorem}

\theoremstyle{definition} 
\newtheorem*{ex}{Example}
\newtheorem{defn}[thm]{Definition}
\newtheorem*{rmk}{Remark}
\newtheorem{prob}[thm]{Problem}

\Crefname{thm}{Theorem}{Theorems}
\Crefname{prop}{Proposition}{Propositions}
\Crefname{lem}{Lemma}{Lemmas}
\Crefname{cor}{Corollary}{Corollaries}
\Crefname{defn}{Definition}{Definitions}
\Crefname{tab}{Table}{Tables}

\newcommand{\module}[1]{#1} 
\newcommand{\Amod}{\module{A}}
\newcommand{\Bmod}{\module{B}}
\newcommand{\Cmod}{\module{C}}
\newcommand{\Mmod}{\module{M}}
\newcommand{\Nmod}{\module{N}}

\newcommand{\Pmod}{\module{P}}
\newcommand{\Qmod}{\module{Q}}

\newlength\squareheight
\newcommand\squareslash{\tikz{\draw (0,0) rectangle (\squareheight,\squareheight);\draw(0,0) -- (\squareheight,\squareheight)}}
\DeclareMathOperator\hlz{\squareslash}

\newcommand{\Vect}{\mathsf{Vect}}

\newcommand{\modules}{\text{-}\mathrm{Mod}}

\newcommand{\lat}{\Lambda}      
\newcommand{\dlat}{\lat^{\ast}} 
\newcommand{\qlat}{\dlat/\lat}  
\newcommand{\plat}{\lat^\perp}  
\newcommand{\nlat}{\lat^\circ}  
\newcommand{\ldat}{\Psi}      

\newcommand{\hlie}{\widehat{\mathfrak{h}}}
\newcommand{\hvec}{\mathfrak{h}}

\newcommand{\btf}{G^T}
\newcommand{\btfi}[2]{\btf_{#1}\brac*{#2}}
\newcommand{\tstr}{\varphi}

\newcommand{\Dcat}{\categ{D}}
\newcommand{\Ccat}{\categ{C}}

\newcommand{\cst}[1]{{#1}}
\newcommand{\sct}[1]{s\brac*{\cst{#1}}}

\newcommand{\ffv}{\xi}
\newcommand{\rffv}{\widetilde{\ffv}}
\newcommand{\gv}{Grothendieck-Verdier}
\newcommand{\rk}{\mathrm{rk}}
\newcommand{\rgv}{ribbon Grothendieck-Verdier}
\newcommand{\data}{bosonic lattice data}

\newcommand{\halg}{\Amod_{\mathsf{H}}}
\newcommand{\valg}{\Amod_{\mathsf{V}}}

\begin{document}

\title[Duality structures for representation categories of vertex operator algebras and the Feigin-Fuchs boson]
{Duality structures for representation categories of vertex operator algebras \\ and the Feigin-Fuchs boson}

\author[R Allen]{Robert Allen}

\address[Robert Allen]{
School of Mathematics \\
Cardiff University \\
Cardiff, United Kingdom, CF24 4AG.
}

\email{allenr7@cardiff.ac.uk}

\author[S Lentner]{Simon Lentner}
\address[Simon Lentner]{
  Algebra and Number Theory\\
  University Hamburg\\
  Bundesstraße 55\\
  D-20146 Hamburg
}
\email{simon.lentner@uni-hamburg.de}

\author[C Schweigert]{Christoph Schweigert}
\address[Christoph Schweigert]{
  Algebra and Number Theory\\
  University Hamburg\\
  Bundesstraße 55\\
  D-20146 Hamburg
}
\email{christoph.schweigert@uni-hamburg.de}

\author[S Wood]{Simon Wood}

\address[Simon Wood]{
  School of Mathematics \\
  Cardiff University \\
  Cardiff, United Kingdom, CF24 4AG.
}

\email{woodsi@cardiff.ac.uk}


\begin{abstract}
Huang, Lepowsky and Zhang developed a representation theory for \voa{s}
that endows suitably chosen module categories with the structures of
braided monoidal categories. Included in the theory is a functor which
assigns to discretely strongly graded modules a
contragredient module, obtained as a gradewise dual. In this paper, we show that this
gradewise dual endows the representation category with the
structure of a \rgv{} category. This duality structure is more general
than that of a rigid monoidal category; in contrast to rigidity, it 
naturally accommodates
the fact that a \voa{} and its gradewise dual need not be
isomorphic as modules and that the tensor product of modules over \voa{s}
need not be exact.

We develop criteria which allow the detection of \rgv{} 
equivalences and use
them to explore  ribbon \gv{} structures in the example of
the rank \(n\) Heisenberg \voa{} or chiral free boson on a not necessarily full rank
even lattice with arbitrary choice of conformal vector. We show that these
categories are equivalent, as \rgv{} categories, to certain categories of graded vector 
spaces and categories of modules over a certain Hopf algebra.
\end{abstract}

\maketitle

\onehalfspacing

\section{Introduction}

\Voa{s} are algebraic structures with numerous applications in mathematical physics, representation theory,
geometry  and combinatorics. For any algebraic structure, it is important to first select a ``sensible'' category of representations and then to
understand the structure this representation category naturally
carries. In the case of \voa{s}, the consensus expectation is that a
sensible category of representations should admit
the structure of a braided tensor category. For a large class of choices of \voa{} 
module categories a good tensor product theory 
that includes a braiding has been found in
 by Huang, Lepowsky and Zhang  in the long series of papers \cite{HuaLog}.  In \cref{thm:sufHLZ} we record a list of sufficient
conditions, collated from  \cite{HuaLog}, for a module category to admit these
structures and specialise these in \cref{thm:simpsufconds} for easier
application to the categories considered in \cref{sec:voas}. Categories
satisfying these conditions include \(\NN\) gradable modules over \cfin{2}
\voa{s} (this covers all rational theories and also all logarithmic \cfin{2}
theories such as the \(\mathcal{W}_{p,q}\)
triplet models \cite{FeiWpq06}) as well as certain module categories with infinitely many
inequivalent simple modules such as Heisenberg or bosonic ghost module
categories to name but a few \cite{All20,AugBp19,CreuVir20,CreuSing20,CreuGL20}.

Contragredient representations appear for many algebraic structures. They
lead, in many cases where the representation category is a monoidal category, 
to the notion
of {\em rigidity}. Recall that an object 
is called rigid, if it has 
both left and right duals, each of which comes with evaluation and
coevaluation morphisms that satisfy the usual zig-zag relations.
A category
is called rigid, if every object is rigid. This is a property: any
left dual or right dual is unique up to unique isomorphism.

Rigidity as categorical formalisation of duality is widely used: it applies to the
category of finite-dimensional
vector spaces and to categories of finite-dimensional modules over
  finite-dimensional (weak quasi) Hopf algebras. Duals
are also widely used in quantum topology, since they lead to a powerful
graphical 
calculus which allows, for suitable tensor categories, for the construction of 
invariants of knots, links and manifolds.

On the other hand, the notion of rigidity has severe limitations. They already
become apparent when one considers Hopf algebroids: 
Hopf algebroids are interesting
algebraic structures with the desirable property that any 
finite tensor category can be realised as the
representation category of a finite-dimensional Hopf algebroid \cite[Theorem 7.6]{brlv}.
However, the natural duality structure for Hopf algebroids is {\em not} rigidity, see
\cite{daSt3}, where it was shown that the natural duality structure is
that of a \gv{} category.  

A \emph{\gv{} category} is a monoidal category $\Ccat$, with a distinguished
object $K$, called the dualising object, such that for any pair of objects
$X,Y\in\Ccat$, there are natural isomorphisms
\begin{equation}
  \Homgrp{}{-\otimes Y}{K} \isomto \Homgrp{}{-}{DY},
\end{equation}
where $D$ is a contravariant equivalence of categories.  In the context
  of \voa{} module categories
the dualising object $K$ should be thought of  as the gradewise dual of the \voa{}, seen as a module over
itself, and  $D$ as the functor which assigns to any object its gradewise dual
and to any morphism its transpose.
This seemingly simple definition of a \gv{} category 
has important consequences, for example, 
it guarantees the existence
of internal Homs for all objects $X,Z\in \Ccat$, by providing the explicit formula
\begin{equation}
  \iHom{X}{Z}\cong D(X\otimes D^{-1} Z),
  \label{eq:ihomfor}
\end{equation}
and implies that the tensor product 
of $\Ccat$ is right exact, if the category is abelian. 
Intriguingly, every \gv{} category is also endowed with a second
tensor product	$X\bullet Y= D^{-1} (DY\otimes DX)$ which turns out to
be left exact \cite{Boy12,ManGV17}, again, if the category is abelian. The two tensor products $\bullet$ and $\otimes$ should be
considered on an equal footing. It remains to be discovered what the full
implication of these two tensor products is for \voa{s} and \cfts{}. 
Rigid categories are examples of \gv{} categories, where the tensor unit is a
dualising object, though the tensor unit being a dualising object does not
imply that the category is rigid in general.

The notion of a \gv{} category (no rigidity assumed)
is nicely compatible with additional structure on the category $\Ccat$,
for example, 
a braiding, and it is possible to introduce notions of a balancing and a
twist. It is thus not surprising that this structure has surfaced in numerous
disparate places:
\gv\ categories are also known as  $*$-autonomous categories \cite{Bar79}, however,
in this paper, we use the more recent terminology 
of \gv{} categories \cite{Boy12,ManGV17,Mul20}.
The main insight of this paper is that the notion of a \rgv{} category
is the natural duality structure 
on tensor categories of modules of \voa{s} to which the HLZ-theory of tensor
products applies.
This is a very welcome insight. For example, the tensor product of a rigid abelian 
tensor category is necessarily exact, however, this can, in general, not be
expected to be true for representation categories of \voa{s}.
Indeed, the \(\mathcal{W}_{2,3}\) triplet model provides
just such a counter example \cite{GabW2309}. 

It should be emphasised that the structure of a \gv{} category naturally
appears in many fields of mathematics. \gv{} structures are linked to the
appearance of dualising sheaves {see, for example, \cite{neeman10} for a recent
discussion of dualising sheaves; further cyclic algebras over
the framed little disc operad with values in the bicategory of finite linear
categories are \rgv{} categories \cite{Mul20};
\gv{} structures are referred to
as the (categorical semantics) of the multiplicative fragment of linear logic
(MLL) \cite{Mel09}; and
\gv{} categories have also been used to describe categorical structures on categories
of topological vector spaces \cite[Appendix]{Bar79}. 

\smallskip

The purpose of this paper is three-fold: first, we show that categories of
representations of \voa{s} to which the HLZ theory of tensor products
applies are naturally \rgv{} categories. Second, we provide tools to compare
categories of different algebraic origin as \gv\ categories. Finally, we provide
first simple, yet instructive applications of these general results by considering
\voa{s} based on Feigin-Fuchs bosons. These are also interesting
building blocks for the description of more general classes of \voa{s}
of recent interest, for example, of ghost systems \cite{AdaBG19,All20} or the
triplet models \cite{FeiWpq06} and their higher rank generalisations \cite{FeiLog10}.

Let us comment on the importance of these results: in the HLZ theory of tensor
products, an 
important role is played by the contragredient dual, that is, a grade-wise dual.
It is known that this dual does not, in general, provide the structure of
rigidity on the representation category. Moreover, examples show that 
it is {\em not} natural to require the contragredient dual of the \voa{},
that is, of the monoidal unit of the tensor category, to be isomorphic to the
monoidal unit. Indeed, in the \gv\ structure, the contragredient dual
of the monoidal unit has an important independent conceptual
role as a dualising object.  To the best of the authors' knowledge, this is
the first paper explicitly observing that \voa{} module categories admit \gv{} structures,
however, consequences of \gv{} duality structures for \voa{}
  module categories have been observed in the past.
For example, in
  \cite[Display (3.19) and Theorem 3.10]{GabW2309} it was noted that internal homs for the \(c=0\) triplet
  algebra satisfy the formula \eqref{eq:ihomfor} above. Further, in
  \cite[second paragraph above Main Theorem 1 and end of
  Section 5.1]{CreuGlu19} it
  was noted that, if the \voa{} is self-contragredient as a module over itself
  (hence the \voa{} is a dualising object),
  then internal homs exist. Consequences of the \gv{} structure arising when
    the \voa{} is self-contragredient were also crucial to recent results \cite{McrC221}
    relating \cfin{2}ness and rigidity.
It is most gratifying to see  \gv{} structure explain and generalise such 
  phenomena and that the deep and general HLZ theory of tensor product finds
  its natural categorical counterpart in general \rgv{} categories.

We expect that our results will enable much future research. \Voa{s}
are notoriously intricate algebraic structures. Thus for many constructions, 
in particular, the construction of full local conformal field theories from chiral conformal field
theories, it is therefore desirable to work, as far as possible, in terms of the
appropriate categorical structures. The structure of a  \rgv{} category is
rich enough to give us confidence that such a theory can be developed.

We now summarise the main results and how
this paper is structured. In \cref{sec:cattools} we give an
  overview of the categorical notions required for this paper including \rgv{}
  structures and HLZ tensor product theory. The two main general
results are:
\begin{itemize}
	\item In \cref{thm:gvstr}, we state precise conditions which ensure
	that a representation category $\Ccat$ of a \voa{} $V$ is a 
        \rgv{} category
	with dualising object $V'$ the contragredient of the \va\ as a module over itself
	and the functor of taking contragredients as the dualising functor.
	\item
	In \cref{thm:btcvoa} and the subsequent \cref{thm:tstrcor},
	we establish explicit ways to set up equivalences of \rgv{} categories.
\end{itemize}

In \cref{sec:freebosons} we then turn to chiral free bosons which are partially compactified, on a
lattice (which can have non-maximal rank) and a non-degenerate bilinear form of indefinite
signature. For the conformal structure, we admit the possibility of Feigin-Fuchs bosons, that is, we consider conformal vectors of the form
\begin{equation}
  \omega_{\gamma} = \frac{1}{2} \sum_{i} \alpha^i_{-1} \alpha^{i*}_{-1}\ket{0} +
  \gamma_{-2}\ket{0},\quad \gamma\in \hvec_\CC,
\end{equation}
where the first summand is the standard Sugawara formula for a conformal
vector and the second is a deformation by a derivative of (a linear
combination of) any of the conformal weight 1 generating free boson fields.
This leads to the following set of input data which we collect in the form of 
\emph{\data{}},
see \cref{def:boslatt}: a finite dimensional real vector space
$\hvec$ with a non-degenerate symmetric real-valued bilinear form 
$\pair{-}{-}$ , an even lattice $\lat \subset \hvec$ and a
distinguished element \(\ffv \in \qlat\) (where \(\dlat\) is the subgroup of
\(\hvec\) that pairs integrally with \(\lat\)) which describes the Feigin-Fuchs
structure of the boson. From these data, we construct three
different algebraic structures:
\begin{itemize}
\item A \rgv{} category of graded vector spaces, extending a classical
  construction of Eilenberg-MacLane 	\cite{Mac52} and Joyal-Street 
  \cite{Joy93} for braided categories,
  see \cref{thm:vecgv}.
\item A lattice \voa{} built from Heisenberg \va{s} and a category
  of lattice \voa{} modules to which the HLZ tensor product theory applies
  so that is a ribbon \gv{} category.
\item A Hopf algebra that is possibly infinite dimensional, together with 
  a representation category that is a \rgv{} category.
\end{itemize}
	
Then the two \cref{thm:voaequiv,thm:hpfgvstr} assert that,
given a set of bosonic lattice data, these three categories are all equivalent as
\rgv{} categories.  The associativity and braiding structures of these
categories all do not depend on the distinguished element \(\ffv \in
\qlat\). The role of this element is to determine the dualising object and
twist of the \rgv{} structure, and in the case of the \voa{} module category
the conformal structure.

In \cref{sec:chars}, the final section, we explore the characters of simple
lattice \voa{} modules and explore their transformation
under the modular group $\mathrm{SL}(2,\ZZ)$, and compute the
generalisation of the Verlinde formula conjectured by the standard module
formalism
\cite{RidSL208,CreMod12,CreW1p12,CreWZW13,CreLogInt13,RidVer14,CreWpq17,RidBos14}.
We observe that this conjectured Verlinde formula correctly
predicts the multiplicities of simple modules in the tensor product, which
do not depend on the distinguished element \(\ffv\). In contradistinction,
the formulae of the modular $S$ and \(T\) matrices show shifts 
that take into account the non-trivial \gv{} structure of the Feigin-Fuchs boson.
It remains an interesting problem to explain these shifts in terms of a systematic theory
of traces for pivotal \gv{} categories.

\section{The categorical framework}
\label{sec:cattools}

In this section we review the notion of a \rgv{} category. This notion includes rigid ribbon categories, but is more general.
We show that this notion is the natural duality structure on tensor categories
of modules of \voa{s}.  We assume a basic familiarity with tensor
categories and 
with \voa{s}, referring readers unfamiliar with these to \cite{EtiTen15} or
\cite{DoLGVA93}, respectively. The notion of a \rgv{}
category and the relevant aspects of the HLZ-theory of tensor products
will be reviewed.

\subsection{Grothendieck-Verdier categories}
\label{sec:GVcats}

Duality, in particular, in the form a rigidity, plays an important role in quantum
topology, a subject intimately linked to vertex algebras and their representation
categories.  The tensor product of a rigid abelian tensor category
is necessarily exact. This can, in general, not be
expected to be true for representation categories of vertex algebras, which are
monoidal categories. Indeed, the \(\mathcal{W}_{2,3}\) triplet model provides
just such a counter example \cite{GabW2309}. Rigidity is a property;
it is actually a special case of a more general duality {\em  structure}. Categories with
such a structure are called $*$-autonomous categories \cite{Bar79} or, more recently,
Grothendieck-Verdier categories\cite{Boy12,ManGV17,Mul20}. In this paper, we will 
see that this  duality structure is 
indeed very naturally realised on \voa{} module categories 
to which the HLZ-tensor product theory applies. In fact, since these categories
are naturally braided and have a canonical identification of the bidual with the
original module, they admit a pivotal structure. This pivotal structure is
equivalent to the existence of a ribbon structure (which has a prominent
manifestation in the context of \voa{s}). So we are naturally lead to study these module categories as \rgv{} categories.

\begin{defn}
  A \emph{\gv{} category} is a monoidal category $\Ccat$, together with
  a distinguished object $K \in \Ccat$, \emph{called the dualising object}
  satisfying the following conditions.
  \begin{enumerate}
  \item For any object \(Y\in \Ccat\), the contravariant functor
    \(\Homgrp{}{-\otimes Y}{K}\) is representable, that is, there exists an
    object \(DY\in \Ccat\) such that there is a natural isomorphism
    \begin{equation}
      \Homgrp{}{-\otimes Y}{K} \isomto \Homgrp{}{-}{DY}.
      \label{eq:GVprop}
    \end{equation}
    By Yoneda's Lemma there therefore exists a unique (up to natural
    isomorphism) contravariant functor \(D\), called \emph{the dualising functor},
    which assigns to every \(Y\in\Ccat\) the representing object \(DY\), that
    is \(D(Y)=DY\).
  \item The contravariant functor \(D\) above is an anti-equivalence.
  \end{enumerate}
  If in addition the category \(\Ccat\) is braided, then it is called a
  \emph{braided \gv{} category}.
\end{defn}

\begin{rmk}
  An immediate consequence of the above definition of \gv{} categories is the
  existence of a
  natural isomorphism in two variables
    \begin{equation}
      \Homgrp{}{-_1\otimes -_2}{K} \isomto \Homgrp{}{-_1}{D(-_2)},
    \end{equation}
    where the subscripts indicate that the ordering of the variables is preserved,
     of the contravariant functors \(\Homgrp{}{-\otimes -}{K}\) and
     \(\Homgrp{}{-}{D(-)}\).
\end{rmk}

Note that the choice of a dualising object $K$ is structure, as there can be many
inequivalent choices.

\begin{prop}[Boyarchenko-Drinfeld {\cite[Proposition 1.3]{Boy12}}]
  Let $\Ccat$ be a \gv{} category with dualising object \(K\) and
  corresponding dualising functor \(D\).
  \begin{enumerate}
  \item For any invertible object \(U\), the objects \(D(U)\cong K\otimes
    U^{-1}\) and \(D^{-1}U\cong U^{-1}\otimes K\) are again dualising objects
    in \(\Ccat\).
  \item The functors $U\mapsto D(U)=K\otimes U^{-1}$ and $U\mapsto
    D^{-1}U\cong U^{-1}\otimes K$ are anti-equivalences
    between the full subcategory of invertible objects $U\in\Ccat$ and the full
    subcategory of dualising objects.
    \label{itm:invtwist}
  \item If $U\in\Ccat$ is invertible then so is $D^2 U$ and one has a canonical isomorphism
    $K\otimes U^{-1}\cong (D^2U)^{-1}\otimes K$.
  \end{enumerate}
  \label{thm:invdual}
\end{prop}

\begin{prop}
  Let $\Ccat$ be a \gv{} category with dualising object \(K_\Ccat\) and let $\Dcat$
  be a monoidal category. For any monoidal equivalence $F:
  \Ccat \to \Dcat$, the object
  $F(K_\Ccat)$ is a dualising object for $\Dcat$. Thus $\Dcat$ admits a \gv{} structure.
  In particular, if $\Dcat$ has already been endowed with a dualising object
  $K_\Dcat$, then $F(K_\Ccat)$ and $K_\Dcat$ differ by tensoring with 
  an invertible object.
  \label{prop:GVequiv}
\end{prop}

\begin{proof}
  Let \(F^{-1}\) be a quasi-inverse of the equivalence \(F\). For $X,Y\in \Dcat$, consider
  \begin{align}
    \Hom_\Dcat(X\otimes Y, FK_\Ccat) &\isomto\Hom_\Dcat(FF^{-1}X\otimes FF^{-1}Y,FK_\Ccat)
                                         \isomto \Hom_\Dcat(F(F^{-1}X\otimes F^{-1}Y),FK_\Ccat) \nonumber\\
                                       &\isomto \Hom_\Ccat(F^{-1}X\otimes F^{-1}Y, K_\Ccat) \isomto
                                         \Hom_\Ccat(F^{-1}X, D_\Ccat F^{-1}Y) 
                                       \isomto \Hom_\Dcat (X,FD_\Ccat F^{-1} Y),
  \end{align}
  where the second bijection follows from the monoidal structure on \(F\) and
  the forth uses the definition property of the dualising object \(K_\Ccat\).
  This implies that \(FK_\Ccat\) is a dualising object in \(\Dcat\) with
  corresponding dualising functor  $F \circ D_\Ccat\circ F^{-1}$. Finally, if
  \(\Dcat\) was already endowed with a dualising object \(K_\Dcat\), then
  \(FK_\Ccat\) and \(K_\Dcat\) differing by tensoring with an invertible object
  is an immediate consequence of \cref{thm:invdual}.\ref{itm:invtwist}.
\end{proof}

\cref{prop:GVequiv}
shows that monoidal equivalences transport dualising
objects, thus allowing their comparison. In particular, in the notation of
this proposition the pair of \gv{}
categories \(\Ccat,\Dcat\) have equivalent \gv{} structures if and only if
\(FK_\Ccat\cong K_\Dcat\).

Braided \gv{} categories can admit ribbon structures that are compatible with the
\gv{} structure. To make this precise we recall the definition of a twist on a
braided monoidal category. Let \(\Ccat\) be a braided monoidal category with
braiding \(c\), then 
the identity functor with monoidal structure given by the
double braiding
\(J_\Ccat=(\id_{\Ccat},\id_\wun,c \circ c)\) is a braided
monoidal auto-equivalence called the \emph{Joyal-Street equivalence}. A twist on a
braided monoidal category is a monoidal isomorphism \(\theta:\id_\Ccat\to
J_\Ccat\). Explicitly, this means that the twist \(\theta\) obeys
\begin{equation}
  \theta_\wun=\id_\wun,\quad\text{and}\quad
  \theta_{X\otimes Y}=c_{Y,X}\circ c_{X,Y}\circ\brac*{\theta_X\otimes
    \theta_Y}.
  \label{eq:twistbalancing}
\end{equation}
If $\theta$ is a twist on a Grothendieck-Verdier category $\Ccat$, then
$$ \theta^D_X= D^{-1}(\theta_{DX})      $$
is also a twist on $\Ccat$. By \cite[Proposition 7.3]{Boy12}, this is an involution on the 
set of twists. The fixed points under this involution are relevant for 
representation categories of conformal \voa{s}.

\begin{defn}
  A \emph{\rgv{} category} is a braided category \(\Ccat\) with a twist \(\theta\) such that
  \begin{equation}
    D(\theta_X)=\theta_{DX},\quad \forall X\in\Ccat.
    \label{eq:dualisedtwist}
  \end{equation}
  \label{def:ribbonGV}
\end{defn}

Combining all of the notions above, we are lead to the following natural definition of
an equivalence of \rgv{} categories.

\begin{defn}
  Let \(\Ccat\) and \(\Dcat\) be \rgv{} categories with respective dualising
  objects \(K_\Ccat\) and \(K_\Dcat\), 
  and twists \(\theta_\Ccat\) and
  \(\theta_\Dcat\). A \emph{\rgv{} equivalence} is a monoidal equivalence
  satisfying
  \begin{itemize}
  \item equivalence of dualising objects: \(FK_\Ccat\cong K_\Dcat\),
  \item equivalence of twists: \(F(\theta_\Ccat)=\theta_\Dcat\).
  \end{itemize}
  \label{def:rgveq}
\end{defn}

\subsection{Huang-Lepowsky-Zhang tensor categories}

The complete reference for tensor structures arising from \voa{s} and
intertwining operators is the series of papers
\cite{HuaLog} by Huang, Lepowsky and
Zhang. Due to the series admirably operating in great generality, while also
providing many technical details, it can be perceived as intimidatingly
long. There are therefore a number of articles the literature, which include helpful reviews 
highlighting different aspects of the series relevant to different types of
\voa{s} and choices of module category \cite{Cre17,KanRid18,CreuDLim20}. Here
we give our own overview with an emphasis on the results necessary for later
sections. There will be three types of grading appearing
  below, whose relative importance might not be immediately clear for readers
  unfamiliar with the theory. There is the conformal grading by generalised
  eigenvalues of the Virasoro \(L_0\) operator and an additional grading by two abelian
  groups \(A\le B\), with \(A\) grading the \voa{} and \(B\) its
  modules. The latter two gradings have an analogy in the
  setting of a
  simple finite dimensional Lie algebra, where \(A\) is the root lattice
  (which grades the Lie algebra) and
  \(B\) is the dual of the Cartan subalgebra (which grades general weight modules).

\begin{defn}
  Let \(V\) be a \voa{} and \(B\) an abelian group with subgroup \(A\).
  \begin{itemize}
  \item The \voa{} \(V\) is said to be \(A\)-graded, if \(V\) admits a decomposition
    into homogenous spaces \(V^{(\gamma)},\ \gamma\in A\) such that
    \begin{enumerate}
    \item \(V=\bigoplus_{\gamma\in A} V^{(\gamma)}\).
    \item For any \(\alpha,\beta\in A\) and any \(v\in V^{(\alpha)}\)
      \begin{equation}
        Y(v,z)V^{(\beta)}\subset V^{(\alpha+\beta)}\powser*{z,z^{-1}}.
      \end{equation}
    \end{enumerate}
  \item A \emph{weak \(V\)-module} is a vector space
    \(M\) together with a
    field map
    \begin{align*}
      Y_M:V&\to \brac*{\Endgrp{M}}\powser*{z,z^{-1}}\nonumber\\
      v&\mapsto Y_M(v,z)=\sum_{n\in\ZZ}v_n z^{-n-1}
    \end{align*}
    satisfying
    \begin{enumerate}
    \item \emph{Lower truncation}: For all \(v\in V\) and \(m\in M\), \(v_n
      m=0\) for sufficiently large \(n\in \ZZ\).
    \item \emph{Vacuum property}: \(Y_M(\wun,z)=\id_M\), where \(\wun\in V\)
      is the vacuum vector.
    \item \emph{\(L_{-1}\) derivation}: For any \(v\in V\)
      \begin{equation}
        Y_M(L_{-1}v,z)=\frac{\dd}{\dd z}Y_M(v,z) .
      \end{equation}
    \item \emph{Jacobi identity}: For any \(u,v\in V\),
      \begin{equation}
        z_0^{-1}\delta\brac*{\frac{z_1-z_2}{z_0}}Y_M(u,z_1)Y_M(v,z_2)=z_0^{-1}\delta\brac*{\frac{-z_2+z_1}{z_0}}Y_M(v,z_2)Y_M(u,z_1)
        +z_2^{-1}\delta\brac*{\frac{z_1-z_0}{z_2}}Y_M(Y(u,z_0)v,z_2),
      \end{equation}
      where \(\delta\) denotes the algebraic
      delta distribution, that is the formal power series
    \begin{equation}
      \delta
      \brac*{\frac{z_2-z_1}{z_0}} = \sum_{\substack{r \in \ZZ \\s \ge 0}}
      \binom{r}{s} (-1)^s z_1^s z_2^{r-s} z_0^{-r}.
    \end{equation}
  \end{enumerate}
  If in addition there is a \(B\)-grading on the weak module
  \(M=\bigoplus_{\beta\in B}M^{(\beta)}\), then \(M\) is a \emph{\(B\)-graded weak 
    \(V\)-module}, if the following condition is satisfied.
  \begin{enumerate}[resume]
\item \emph{Grading compatibility}: For all \(\alpha\in A\), \(v\in
      V^{(\alpha)}\), \(\beta\in B\),
      \begin{equation}
        Y_M(v,z) M^{(\beta)}\subset M^{(\alpha+\beta)}\powser*{z,z^{-1}} .
      \end{equation}
    \end{enumerate}
    \item A \emph{\(B\)-graded generalised \(V\)-module} \(M\) is a 
      \(B\)-graded weak \(V\)-module that is graded by generalised \(L_0\)
      eigenvalues, that is, \(M=\bigoplus_{h\in \CC, \beta\in B}M_h^{(\beta)}\), where
      \begin{equation}
        M_h^{(\beta)}=\set*{m\in M^{(\beta)} \st \exists n\in \NN, \ (L_0-h)^n m=0}.
      \end{equation}
      The elements of \(M_h^{(\beta)}\) are called \emph{doubly homogeneous vectors}.
      Note that \emph{\(B\)-graded generalised \(V\)-modules} together with
      module homomorphisms form an abelian
      category.
    \item  A \emph{\(B\)-graded generalised \(V\)-module} \(M\) is called \emph{lower bounded} 
      if for each \(\beta\in B\), \(M_{h}^{(\beta)}=0\) for
      \(\Re h\) sufficiently negative.
    \item A \emph{strongly \(B\)-graded generalised \(V\)-module \(M\)} is a
      \(B\)-graded generalised \(V\)-module whose simultaneous homogeneous
      spaces \(M^{(\beta)}_h\) are all finite dimensional and for fixed \(h\in
      \CC\) and \(\beta\in B\), \(M^{(\beta)}_{h+k}=0\) for sufficiently negative \(k\in\ZZ\). Such a module is
      called \emph{discretely strongly graded} if all non-zero homogeneous spaces
      have real conformal weight and for any \(h\in \RR\), \(\beta\in B\), the
      space
      \(\bigoplus_{\tilde h\in \RR, \tilde h\le h} M^{(\beta)}_{h}\) is finite
      dimensional.
    \item A strongly \(B\)-graded generalised \(V\)-module \(M\) is called
      \emph{graded \cfin{1}} if for any \(\beta\in B\) the space
      \begin{equation}
        C_1\brac*{M}^{(\beta)}=\spn{v_{-1}m\in M^{(\beta)}\st v\in V_h, h>0,
          m\in M}
      \end{equation}
      has finite codimension in \(M^{(\beta)}\).
    \end{itemize}
    \label{def:module}
  \end{defn}

\begin{rmk}
  We abbreviate \(B\)-graded generalised \(V\)-module as \(B\)-graded
  \(V\)-module, or when the abelian group \(B\) is obvious from context as
  \(V\)-module. For the specific \voa{s} to be considered below, we will
  mainly be interested in discretely strongly graded modules which are in
  addition  graded \cfin{1} with respect to a suitable choice of
  vertex operator subalgebra.
\end{rmk}

\begin{prop}[Huang-Lepowsky-Zhang {\cite[Part I, Theorem 2.34]{HuaLog}}]
  Let \(A\leq B\) be abelian groups, \(V\) an \(A\)-graded \voa{}, let
  \(M\) be a \(B\)-graded weak \(V\)-module and define the vector spaces
  \begin{equation}
    M^\prime = \bigoplus_{b\in B,h\in\CC} \brac*{M_h^{(\beta)}}^\ast,\qquad \brac*{M_h^{(\beta)}}^\ast=\Homgrp{\CC}{M_h^{(\beta)}}{\CC}.
  \end{equation}
  If \(M\) is strongly \(B\)-graded, then the canonical linear isomorphisms
  identifying a finite dimensional vector space with its double dual extends to a canonical linear isomorphism
  \(M\cong M''\) of bigraded vector spaces.
  If, in addition, \(M\) is discretely strongly \(B\)-graded, then \(M'\)
  is also a discretely strongly \(B\)-graded with field map \(Y_{M'}\)
  uniquely characterised by
  \begin{equation}
    \pair{Y_{M'}(v,z)\phi}{m}=\pair{\phi}{Y^{\opp}_{M}(v,z)m},\qquad v\in V,\
    \phi\in M',m\in M,
  \end{equation}
  where \(Y^{\opp}_{M}\) is the \emph{opposed field map}
  \begin{equation}
    Y^{\opp}_{M}(v,z)=Y_M\brac*{\ee^{z L_1}\brac*{-z^{-2}}^{L_0}v,z^{-1}}.
    \label{eq:opfieldmap}
  \end{equation}
  The module \(M'\) is called the \emph{contragredient} of \(M\). Opposing the
  field map is involutive, that is, \(Y^{\opp\, \opp}_{M} = Y_{M}\), hence
  the canonical linear isomorphism \(M\cong M''\) above is an isomorphism of
  \(V\)-modules.
  \label{thm:contragredients}
\end{prop}

Note that by (\ref{eq:opfieldmap}) the opposed field map depends on the
conformal (or at least the M\"obius)
structure on the \voa{}, that is, the actions of the Virasoro \(L_0\) and
\(L_1\) operators enter explicitly.
Note further that the opposed field map can be used to define an action of \(V\) on
\(M'\) (or even the full vector space dual \(M^\ast\)) for any weak module \(M\), however, in general the lower truncation
axiom need not hold and thus the terms in the Jacobi identity need not
converge. There are numerous boundedness conditions on conformal weights
which are sufficient for module structures on \(M'\). Here we shall only
consider discrete strong gradation, as this is also a sufficient condition
for tensor product structures in module categories to be considered below.

\begin{defn}
  Let \(A\leq B\) be abelian groups, \(V\) an \(A\)-graded \voa{} and let
  \(M_1, M_2, M_3\) be \(B\)-graded weak \(V\)-modules. Denote by
  \(M_3\set{z}\sqbrac*{\log z}\) the space of formal power 
  series in \(z\) and \(\log z\) with coefficients in \(M_3\), where the
  exponents of \(z\) can be arbitrary complex numbers and with only finitely
  many \(\log z\) terms for any fixed exponent of \(z\).
  A \emph{logarithmic intertwining operator of type
  \(\binom{M_3}{M_1,M_2}\)} is a linear map
  \begin{align}
    \mathcal{Y}:M_1\otimes M_2&\to M_3\set{z}\sqbrac*{\log z},\nonumber\\
    m_1\otimes m_2&\mapsto
                    \mathcal{Y}\brac*{m_1,z}m_2=\sum_{\substack{s\ge0\\t\in
    \CC}}\brac*{m_1}_{t,s} m_2z^{-t-1}\brac*{\log z}^s
  \end{align}
  where \(\brac*{m_1}_{t,s}\in \Homgrp{\CC}{M_2}{M_3}\),
  satisfying the following properties.
  \begin{enumerate}
  \item \emph{Lower truncation}: For any \(m_i\in M_i\), \(i=1,2\), and \(s\ge0\)
	\begin{equation}
	\brac*{m_1}_{t+k,s}m_2 = 0
	\end{equation}   
	for sufficiently large $k \in \ZZ$. 
  \item \emph{\(L_{-1}\) derivation}: For any \(m_i\in M_i\),\ \(i=1,2\),
    \begin{equation}
      \mathcal{Y}(L_{-1}m_1,z)m_2=\frac{\dd}{\dd z}\mathcal{Y}(m_1,z)m_2.
    \end{equation}
  \item \emph{Jacobi identity}: For any \(v\in V\), \(m_i\in M_i\),\ \(i=1,2\),
    \begin{align}
      z_0^{-1} \delta \brac*{\frac{z_1-z_2}{z_0}} Y_{M_3}(v,z_1) \mathcal{Y}\brac*{m_1,z_2}
      m_2 = & z_0^{-1} \delta \brac*{\frac{-z_2 +z_1}{z_0}} \mathcal{Y}\brac*{m_1,z_2}
              Y_{M_2}(v,z_1) m_2 
              + z_2^{-1} \delta \brac*{\frac{z_1 - z_0}{z_2}} \mathcal{Y}\brac*{Y_{M_1}(v,z_0)
              m_1,z_2} m_2 .
                    \label{eq:intjac}
    \end{align}
  \end{enumerate}
  The intertwining operator \(\mathcal{Y}\) is called \emph{grading
    compatible} if addition to the conditions above it also satisfies the
  following condition.
  \begin{enumerate}[resume]
\item \emph{Grading compatibility}: For any \(\beta_1,\beta_2\in B\), \(m_1\in M_1^{(\beta_1)}\)
  \begin{equation}
    \mathcal{Y}(m_1,z) M_2^{(\beta_2)} \subset
    M_3^{(\beta_1+\beta_2)}\set{z}\sqbrac*{\log z}.
  \end{equation}  
\end{enumerate}
The conditions above are all linear and so we denote by
\begin{equation}
  \ispc{M_1}{M_2}{M_3},\qquad
  \grispc{M_1}{M_2}{M_3},
\end{equation}
respectively, the vector space of all logarithmic intertwining operators of type
  \(\binom{M_3}{M_1,M_2}\) and the subspace of all grading compatible ones.
\end{defn}

Note that if, as will be the case in \cref{sec:voas},
the \(B\)-grading corresponds to
eigenvalues of zero modes of certain vectors in \(V\) of conformal weight 1,
then the Jacobi identity implies that all logarithmic intertwining operators
are grading compatible. Intertwining operators admit a dualisation analogous
to the opposed field map in \cref{thm:contragredients} which will prove
crucial to showing the existence of \gv{} structures on \voa{} module categories.

\begin{thm}[Huang-Lepowsky-Zhang {\cite[Part II Proposition 3.46]{HuaLog}}]
  Let \(M_1,M_2,M_3\) be strongly graded generalised modules over some \voa{}
  \(V\). Then there exists a natural linear isomorphism \(A:\ispc{M_1}{M_2}{M_3}\to
  \ispc{M_1}{M_3^\prime}{M_2^\prime}\), which on any intertwining operator
  \(\mathcal{Y}\in \ispc{M_1}{M_2}{M_3}\)
  evaluates as
  \begin{equation}
    \pair{A(\mathcal{Y})\brac*{m_1,x}m_3^\prime}{m_2}_{M_2}
    =\pair{m_3^\prime}{\mathcal{Y}\brac*{\ee^{xL_1}\ee^{\ii\pi
          L_0}\brac*{x^{-L_0}}^2m_1,x^{-1}}m_2}_{M_3},\qquad m_1\in M_1,\
    m_2\in M_2,\ m_3^\prime\in M_3^\prime,
  \end{equation}
  where the subscript indicates which module the pairings are evaluated in.
  The isomorphism \(A\) preserves grading and hence restricts to a natural
  isomorphism \(\grispc{M_1}{M_2}{M_3}\to
  \grispc{M_1}{M_3^\prime}{M_2^\prime}\).
  \label{thm:intopduals}
\end{thm}

The Jacobi identity for intertwining operators implies that intertwining operators are essentially maps from a pair \(M_1,M_2\) of modules
to a third module \(M_3\), which are bilinear in the action of the \voa{} \(V\). It
therefore makes sense to ask if there exists some universal tensor product
module through which all intertwining operators factor. That is, given some
choice of \(V\) module category \(\Ccat\) and two modules \(M_1,M_2\in \Ccat\),
does there exist a module \(M_1\boxtimes M_2\in\Ccat\) together with an intertwining operator
\(\iop{M_1,M_2}\in V\binom{M_1\boxtimes M_2}{M_1,M_2}\) such that for any
\(R\in \Ccat\) and intertwining operator \(\mathcal{Y}\in
V\binom{R}{M_1,M_2}\) there exists a unique module map \(f\in
\Homgrp{\Ccat}{M_1\boxtimes M_2}{R}\) such that \(\mathcal{Y}=f\circ
\iop{M_1,M_2}\)? That is, such that the diagram
\begin{equation}
  \begin{tikzcd}
    M_1\otimes M_2 \arrow[rr,"\iop{M_1,M_2}"] \arrow[rrd,"\mathcal{Y}"] && M_1\boxtimes M_2\set{z}\sqbrac*{\log z}\arrow[d,dashed,"\exists!f"]\\
    &&R\set{z}\sqbrac*{\log z}
  \end{tikzcd}
  \label{eq:fusuprop}
\end{equation}
commutes. Assuming that \(M_1\boxtimes M_2\in\Ccat\) exists for all pairs of
modules in \(\Ccat\), \(-\fuse -\) becomes a bifunctor after defining the
following evaluation on pairs of morphisms. For \(M_1,N_1,M_2,N_2\in\Ccat\) and morphisms \(f_1:M_1\to N_1\), \(f_2:M_2\to
N_2\), the tensor product morphism \(f_1\fuse f_2: M_1\fuse M_2 \to N_1\fuse
N_2\) is the unique morphism, characterised by the universal property
\eqref{eq:fusuprop}, such that \(\iop{N_1,N_2}\circ(f_1\otimes f_2)=(f_1\fuse
f_2)\circ \iop{M_1,M_2}\).
This characterisation of tensor products (also
called fusion products) of \voa{} modules via a universal property is
conceptually very clear and allows us
to construct maps out of a tensor products in terms of vertex operators; this
will be used frequently in the sequel. However,
it does not provide an actual construction, nor does it guarantee existence.
We sketch some of the ideas of the
construction of \(M_1\fuse M_2\) here but refer to the original source
\cite[Part IV]{HuaLog} and the 
review \cite{KanRid18} for details. The fusion product of two modules
\(M_1,M_2\) can be constructed inside \((M_1\otimes M_2)^\ast\), the full
vector space dual of the complex tensor product. While \((M_1\otimes
M_2)^\ast\) is not a \(V\)-module, it is possible to move the action of fields
in \(V\) from either of the tensor factors \(M_1,M_2\) to \((M_1\otimes
M_2)^\ast\) using a generalisation of the opposed field map. This leads to the consideration of the subspace
\(\comp{\Mmod_1}{\Mmod_2}\subset (M_1\otimes M_2)^\ast\) consisting of all
linear functionals on which the evaluation of a field has only finitely many
singular terms and on which the transported action of fields from either
tensor factor agrees, that is, linear functionals compatible with the lower
truncation property of \cref{def:module} and the \voa{} version of bilinearity. The name COMP refers to the compatibility of the
actions \(V\) on the two tensor factors transported to \((M_1\otimes M_2)^\ast\). It was shown in \cite[Part IV, Theorem 5.48]{HuaLog}
that \(\comp{\Mmod_1}{\Mmod_2}\) is a weak \(V\) module and should morally be
thought of as the contragredient of the fusion product \(M_1\boxtimes M_2\). The subspace
\(\comp{\Mmod_1}{\Mmod_2}\) is however usually too large to be contained in
the category \(\Ccat\) one is considering. For example, it generally
contains vectors which are not finite sums of homogeneous vectors. Under
suitable conditions on \(\Ccat\) (such as those in \cref{thm:sufHLZ}) one can
construct the subspace \(M_1\hlz M_2\subset \comp{\Mmod_1}{\Mmod_2}\)
consisting of the sum of all images of module maps from objects in \(\Ccat\)
into \(\comp{\Mmod_1}{\Mmod_2}\). The contragredient \((M_1\hlz M_2)'\) is
then the fusion product module satisfying the universal property \eqref{eq:fusuprop}.

\begin{prop}[Huang-Lepowsky-Zhang {\cite[Part VIII, Section 12]{HuaLog}}]
  Let \(A\leq B\) be abelian groups, let $V$ be an \(A\)-graded \voa{} and
  $\Ccat$ a choice of category of \(V\)-modules (that is a subcategory of the
  category of all \(B\)-graded \(V\)-modules) containing \(V\) as an object such that the following
  conditions hold.
  \begin{enumerate}
  \item For any \(M_1,M_2\in \Ccat\) there exist \(M_1\boxtimes M_2\in \Ccat\) and
    \(\iop{M_1,M_2}\in \grispc{M_1}{M_2}{M_1\boxtimes M_2}\) such that the
    universal property \eqref{eq:fusuprop} holds.
  \item For any \(M_1,M_2,M_3\in \Ccat\), there is a family of isomorphisms \(A_{M_1,M_2,M_3}^{x_1,x_2}: \brac*{M_1\boxtimes
      M_2}\boxtimes M_3\to M_1\boxtimes \brac*{M_2\boxtimes M_3}\)  depending
    on complex variables \(x_1,x_2\)
    that is functorial in \(M_1,M_2,M_3\).
    Further, for \(m_i\in M_i\), \(x_1,x_2\in \CC\), \(|x_1|>|x_2|>|x_1-x_2|>0\), the
    expressions
    \begin{equation}
      \ifld{M_1,M_2\fuse M_3}{m_1}{x_1}\ifld{M_2,M_3}{m_2}{x_2}m_3,\qquad
      \ifld{M_1\fuse M_2,M_3}{\ifld{M_1,M_2}{m_1}{x_1-x_2}m_2}{x_2}m_3,
    \end{equation}
    converge absolutely for any choice of branch of logarithm for
      \(x_1,x_2\)
    in the algebraic completions of \(M_1\fuse\brac*{M_2\fuse
      M_3}\) and \(\brac*{M_1\fuse M_2}\fuse M_3\), respectively. Finally,
    \begin{equation}
      \overline{A}_{M_1,M_2,M_3}^{x_1,x_2}\brac*{\ifld{M_1,M_2\fuse M_3}{m_1}{x_1}\ifld{M_2,M_3}{m_2}{x_2}m_3}=
      \ifld{M_1\fuse M_2,M_3}{\ifld{M_1,M_2}{m_1}{x_1-x_2}m_2}{x_2}m_3,
    \end{equation}
    where \(\overline{A}_{M_1,M_2,M_3}^{x_1,x_2}\) is the natural extension of
    \(A_{M_1,M_2,M_3}^{x_1,x_2}\) to algebraic completions.
    \label{itm:assoccond}
  \end{enumerate}
  Then \(\Ccat\) is a braided monoidal category with the \voa{} \(V\) as the
  monoidal unit, whose structure isomorphisms
  are uniquely characterised by the following.
  \begin{itemize}
  \item For \(M\in \Ccat\) the unit morphisms are uniquely characterised by
    \begin{align}
      \ell_M\brac*{\ifld{V,M}{v}{z}m}&=Y_M(v,z)m\nonumber\\
      r_M\brac*{\ifld{M,V}{m}{z}v}& =\ee^{zL_{-1}}Y_M(v,-z)m,
      \label{eq:unitmor}
    \end{align}
    where \(Y_M\) is the field map of \(V\) acting on the module
    \(M\), \(v\in V\) and \(m\in M\).
  \item For \(M_1,M_2\in \Ccat\) the braiding isomorphism \(c_{M_1,M_2}:M_1\fuse M_2\to
    M_2\fuse M_1\) is uniquely characterised by
    \begin{equation}
      c_{M_1,M_2}\brac*{\ifld{M_1,M_2}{m_1}{z}m_2}=\ee^{zL_{-1}}\ifld{M_2,M_1}{m_2}{\ee^{\ii\pi}z}m_1,
    \end{equation}
    where \(m_1\in M_1\) and \(m_2\in M_2\).
  \item There is a twist morphism \(\theta_{M_1}=\ee^{2\pi\ii
      L_0}|_{M_1}\), \(M_1\in\Ccat\), which satisfies \(\theta_V=\id_V\) and for any \(M_2\in \Ccat\)
    also satisfies the balancing equation
    \begin{equation}
      \theta_{M_1\fuse M_2}=c_{M_2,M_1}\circ
      c_{M_1,M_2}\circ\brac*{\theta_{M_1}\fuse \theta_{M_2}}.
      \label{eq:balance}
    \end{equation}
    \label{itm:VOAtwist}
  \item For \(i=1,2,3\), \(M_i\in\Ccat\), \(m_i\in M_i\), \(x_1,x_2\in \RR\), \(x_1>x_2>x_1-x_2>0\), the
    associativity isomorphism
    \(A_{M_1,M_2,M_3}:M_1\fuse\brac*{M_2\fuse M_3}\to \brac*{M_1\fuse
      M_2}\fuse M_3\) is
    given by the isomorphism of Part \ref{itm:assoccond} above in the limit
    \(x_2\to1,\ x_1-x_2\to1\), that is
    \begin{equation}
      A_{M_1,M_2,M_3}=\lim_{x_2\to1}\lim_{x_1-x_2\to1}A_{M_1,M_2,M_3}^{x_1,x_2},
    \end{equation}
    where the limits are taken along the positive real line with choice of
    branch of logarithm such that the arguments for \(x_1,x_2,x_1-x_2\) are all \(0\).
  \end{itemize}
  \label{thm:tenstr}
\end{prop}

Proving that a choice of \voa{} module category admits the braided tensor structure of
\cref{thm:tenstr} is a highly non-trivial task. Fortunately a number of
sufficient conditions were identified in \cite{HuaLog}, which we quote and
summarise in the following theorem.
\begin{thm}[Huang-Lepowsky-Zhang]
  Let \(A\leq B\) be abelian groups and let $V$ be an \(A\)-graded
  \voa{}. Then the following conditions on a choice of \(B\)-graded module
  category \(\Ccat\) are sufficient for \(\Ccat\) to have the 
  braided monoidal structure induced from intertwining operators described in
  \cref{thm:tenstr}.
  \begin{enumerate}
  \item The \voa{} \(V\) is an object in \(\Ccat\) and all objects of
    \(\Ccat\) are strongly \(B\)-graded. For any \(M_1, M_2\in \Ccat\) the logarithmic
    intertwining operator \(\iop{M_1,M_2}\) satisfying the universal property
    in the definition of the tensor product of \(M_1\) and \(M_2\) is grading
    compatible (hence all logarithmic
    intertwining operators are grading compatible)
    \cite[Part III, Assumption 4.1]{HuaLog}.\label{itm:grading}
  \item $\categ{C}$ is a full subcategory of the category of 
    strongly \(B\)-graded modules and is closed under the contragredient functor and
    under taking finite direct sums 
    \cite[Part IV, Assumption 5.30]{HuaLog}.
    \label{itm:sums}
  \item For any object in $\categ{C}$ all conformal weights are real and the
    non-semisimple part of $L_0$ acts nilpotently, that is, there is a uniform
    bound on the size of Jordan blocks for any given module though there need
    not be global bound for the entire category
    \cite[Part V, Assumption 7.11]{HuaLog}.\label{itm:real}
  \item  \(\categ{C}\) is closed under images {of module homomorphisms}
    \cite[Part VI, Assumption 10.1.7]{HuaLog}.\label{itm:closure} 
  \item The convergence and extension properties for either
    products or iterates holds \cite[Part VII, Theorem 11.4]{HuaLog}. \label{itm:prodsits}
  \item For any objects $M_1, M_2 \in \Ccat$, let $M_{v}$ be
    the $V$-module generated 
    by a \(B\)-homogeneous generalised \(L_0\) eigenvector
    $v \in \comp{\Mmod_1}{\Mmod_2}$.
    If $M_{v}$ is lower bounded then $M_{v}$ is strongly graded
    and an object in 
    $\Ccat$ \cite[Theorem 3.1]{HuaApp17}.
    \label{itm:cycmods}
  \end{enumerate}
  \label{thm:sufHLZ}
\end{thm}

\begin{rmk}
  The above sufficient conditions are in a sense the weakest known conditions
  for a \voa{} module category admitting a braided monoidal
  structure. However, they can in practice be difficult to verify (especially
  Conditions \ref{itm:prodsits} and \ref{itm:cycmods}). Other more restrictive
  and hence more tractable sets of conditions are therefore also commonly considered
  in the literature. 
  The most famous set arguably being:
  \begin{itemize}
  \item The \voa{} \(V\) is \cfin{2}, all \(L_0\) eigenspaces are
    finite dimensional, the only non-zero eigenspaces have non-negative
    integral \(L_0\) eigenvalue, \(\dim V_0 =1\) and \(V\cong V'\).
  \item The category \(\Ccat\) of all admissible (also known as \(\NN\)
    gradable) modules is semisimple.
  \end{itemize}
  If the above conditions hold, then \(\Ccat\) is a modular tensor category and
  the much celebrated Verlinde formula holds \cite{HuaVer08}. A weaker set of
  sufficient conditions only requires the \voa{} to be \cfin{2} and the
  category to be the category of admissible modules without any assumptions on
  semisimplicity (\cfin{2}ness, however, still guarantees that the category is
  finite). Comparatively few general results are known when the
  \cfin{2}ness condition is not satisfied, that is, there is currently no
    known general choice of module category for a general \voa{} that admits 
    the braided monoidal structure of \ref{thm:tenstr}. However, a recent flurry of new
  insights  appears to be changing this at last
  \cite{CreuVir20,CreuDLim20,CreuGL20,CreuSing20,All20} for certain families
  of non-\cfin{2} \voa{s}.
\end{rmk}

With the braided monoidal properties described in \cref{thm:tenstr} and the
sufficient conditions of \cref{thm:sufHLZ} in hand, we can now connect these
structures with the \gv{} structures introduced in \cref{sec:GVcats}.

\begin{thm}
  Let $V$ be a \voa{} and
  $\Ccat$ a choice of category of \(V\)-modules which  contains \(V\) as an
  object, is closed under taking contragredients and which satisfies the two conditions (and hence also the conclusions) of
  \cref{thm:tenstr}. Then \(\Ccat\) is a \rgv{} category with dualising object
  \(V'\) (the contragredient of the \voa{} as a module over
  itself), dualising functor given by the taking of contragredients, and with
  twist \(\theta=\ee^{2\pi\ii  L_0}\).
  \label{thm:gvstr}
\end{thm}
\begin{proof}
  Let \(X,Y,Z\in \Ccat\) and
  recall the linear isomorphism \(A:\ispc{X}{Y}{Z}\to
  \ispc{X}{Z^\prime}{Y^\prime}\) of \cref{thm:intopduals}.
  When working with a strong grading consider
  instead the restriction \(A:\grispc{X}{Y}{Z}\to \grispc{X}{Z^\prime}{Y^\prime}\).
  Since category \(\Ccat\) is closed under contragredients we therefore also
  have a natural isomorphism
  $\Homgrp{}{X \fuse Y}{Z} \isomto \Homgrp{}{X \fuse Z^\prime}{Y^\prime}$.
  Setting $Z = V^\prime$, we have
  \begin{equation}
    \Homgrp{}{X \fuse Y}{V^\prime} \isomto \Homgrp{}{X\fuse V^{\prime\prime}}{Y^\prime} \isomto
    \Homgrp{}{X\fuse V}{Y^\prime} \isomto \Homgrp{}{X}{Y^\prime}
  \end{equation}
  where we have made use of the canonical isomorphism \(V^{\prime\prime}\simeq V\) of
  \cref{thm:contragredients} and the left
  unit isomorphism $V \fuse Y \simeq Y$.
  This proves that \(\Ccat\) is a braided \gv{} category with
  dualising object $V^\prime$. Next we show that the twist \(\theta_M=\ee^{2\pi \ii L_0}|_M\) and
  the contragredient functor satisfy the compatibility condition \eqref{eq:dualisedtwist}.
  From formula \eqref{eq:opfieldmap} for the opposed field map one sees that
  \(L_0^{\opp}=L_0\) and hence for any module \(M\in \Ccat\),
  \(\brac*{\theta_M}^\prime=\theta_{M^\prime}\).
  Thus \(\Ccat\) is ribbon \gv{}.
\end{proof}

\begin{rmk}
  Note that in \cref{thm:gvstr} we do not require $V^\prime$ and $V$ to be isomorphic as $V$-modules.
  Indeed, $V^\prime$ plays the important structural role of a dualising
  object. Note also that \cref{thm:gvstr} and the remark preceeding
  it together imply that categories of admissible modules over \cfin{2}
  \voa{s} are a source of finite \rgv{} categories.
\end{rmk}

The convergence and extension property of \cref{thm:sufHLZ}.\ref{itm:prodsits}
is a technical condition on the analytic properties of intertwining operators,
whose details we shall not state here. Instead we give sufficient conditions
for the convergence and extension property to hold.

\begin{thm}[Allen-Wood {\cite[Theorem 5.7]{All20}}]
    \label{thm:gradconext}
  Let \(A\le B\) be abelian groups, let \(V\) be an
  \(A\)-graded \voa{} and let \(\overline{V}\) be a vertex subalgebra
  of \(V^{(0)}\). Further, let \(M_i,\ i=0,1,2,3,4\) be 
  \(B\)-graded \(V\)-modules. Finally let \(\mathcal{Y}_1\),
  \(\mathcal{Y}_2\), \(\mathcal{Y}_3\) and \(\mathcal{Y}_4\) be logarithmic
  grading compatible intertwining operators
  of types \(\binom{M_0}{M_1,\ M_4}\), \(\binom{M_4}{M_2,\ M_3}\),
  \(\binom{M_0}{M_4,\ M_3}\) and \(\binom{M_4}{M_1,\ M_2}\)
  respectively.
  If the modules  \(M_i,\ i=0,1,2,3\) (note \(i=4\) is excluded) are
  discretely strongly graded, and graded \cfin{1} as
  \(\overline{V}\)-modules,
  then \(\mathcal{Y}_1\), \(\mathcal{Y}_2\) satisfy the convergence and
  extension property for products and \(\mathcal{Y}_3\), \(\mathcal{Y}_4\)
  satisfy the convergence and extension property for iterates.
  \label{thm:sufcon}
\end{thm}

Choosing the module category to be abelian and combining
\cref{thm:sufHLZ,thm:sufcon} we obtain the following simplified sufficient
conditions.

\begin{cor}
  Let \(A\leq B\) be abelian groups, let $V$ be an \(A\)-graded
  \voa{} and let \(\Ccat\) be an abelian full subcategory of the
  category of all \(B\)-graded \(V\)-modules, which contains \(V\).
  Then the following conditions are sufficient for \(\Ccat\) to have the 
  \rgv{} category structures induced from intertwining operators described in \cref{thm:tenstr,thm:gvstr}.
  \begin{enumerate}
  \item All objects in \(\Ccat\) are discretely strongly graded and \(\Ccat\) is closed
    under taking contragredients.
    \label{itm:graddual}
  \item The non-semisimple part of $L_0$ acts nilpotently on any
    object in \(\Ccat\).
  \item There exists a vertex operator subalgebra \(\overline{V}\subset
    V^{(0)}\) such that all objects in \(\Ccat\) are graded \cfin{1} as
    \(\overline{V}\) modules.
    \label{itm:c1}
  \item For any objects $M_1, M_2 \in \Ccat$, every lower bounded submodule of
    \(\comp{\Mmod_1}{\Mmod_2}\) that is finitely generated by doubly
    homogeneous vectors is an object in \(\Ccat\).
    \label{itm:fus}
  \end{enumerate}
  \label{thm:simpsufconds}
\end{cor}

We have ordered the conditions in \cref{thm:simpsufconds} by how difficult
they are to verify in practice. Note in particular that Conditions
\ref{itm:graddual} -- \ref{itm:c1} are merely properties of the types of
modules one wishes to consider and make no reference to tensor products.

\subsection{Functors involving \voa{} module categories}

The monoidal structures of \voa{} module categories are a consequence of the
properties of intertwining operators. The following lemma illustrates how monoidal
functors from linear braided monoidal categories to \voa{} module categories
interact with intertwining operators.

\begin{lem} 
  Let $V$ be a vertex algebra with choice of module category $(\Ccat,\fuse,l,r,\alpha,c)$,
  admitting the 
  braided monoidal structure induced from intertwining operators described in
  \cref{thm:tenstr}.
  Let $(\Dcat,\otimes,l,r,\alpha,c)$ be a linear 
  braided monoidal category, $G: \Dcat \ra \Ccat$ a
  $\CC$-linear abelian functor and 
  \(\tstr_0\) a choice of morphism $\tstr_0: V \ra G(1_\Dcat)$.
  Then the following are equivalent.
    \begin{enumerate}
    \item There exists a natural transformation \(\tstr_2: G(-)\fuse G(-)\to
      G(-\otimes-)\) such that \((G,\tstr_0,\tstr_2)\) is a braided
     lax monoidal functor (lax here means that \(\tstr_0\) and \(\tstr_2\) are
     not required to be isomorphisms).
      \label{itm:ntransf}
    \item There exists a family of linear maps
      \begin{align}
        \btf : \mathrm{Hom}_{\Dcat}(\Mmod \otimes \Nmod, \Pmod) &\ra \ispc{G(\Mmod)}{G(\Nmod)}{G(\Pmod)}, \nonumber \\
        f &\mapsto \btfi{f}{z},
            \label{eq:Gmap}
      \end{align}
      for all \(\Mmod,\Nmod,\Pmod\in\Dcat\),
      satisfying the following conditions.
      \begin{itemize}
      \item Functoriality: For any $\Mmod,\Mmod^\prime, \Nmod,\Nmod^\prime, \Pmod,\Pmod^\prime\in\Dcat$ and any
      $g : \Mmod^\prime \ra \Mmod$, $h : \Nmod^\prime \ra \Nmod$, $k : \Pmod \ra \Pmod^\prime$,
      we have
      \begin{equation}
        \btfi{k \circ f \circ (g \otimes h)}{z} = G(k) \circ \btfi{f}{z} \circ (G(g) \otimes_\CC G(h)),
      \end{equation}
      where \(\otimes_\CC\) denotes the tensor product of complex vector
      spaces and linear maps. 
    \item Unitality: For any $\Nmod \in \Dcat$,
      \begin{align} \label{eq:units} \btfi{l_{\Nmod}}{\tstr_0(v),z}n &=
        Y_{G(\Nmod)}(v,z)n,
      \end{align}
      where \(Y_{G(\Nmod)}\) is the \voa{} field map on the module \(G(\Nmod)\).
    \item Skew symmetry: For any $\Mmod, \ \Nmod \in \Dcat$ and $m \in G(\Mmod)$,
      $n \in G(\Nmod)$,
      \begin{equation}\label{eq:skewsym}
        \btfi{c_{\Nmod, \Mmod}}{n,z} m = \ee^{zL_{-1}} \btfi{\id_{\Mmod \otimes \Nmod}}{m, \ee^{\ii\pi} z}n.
      \end{equation}
    \item Associativity: For any $\Mmod, \ \Nmod, \ \Pmod \in \Dcat$,
      $m \in G(\Mmod)$, $n \in G(\Nmod)$, $p \in G(\Pmod)$ and \(x_1,x_2\in\CC\) such
      that $|x_1|>|x_2|>0$ and
      $|x_2| > |x_1 - x_2|>0$,
      \begin{align} \label{eq:assoc}
        \btfi{\alpha_{\Mmod, \Nmod, \Pmod}}{m,x_1} \btfi{\id_{\Nmod \otimes \Pmod}}{n,x_2} p &=
        \btfi{\id_{(\Mmod \otimes \Nmod) \otimes \Pmod}}{\btfi{\id_{\Mmod\otimes \Nmod}}{m, x_1-x_2} n, x_2} p,\nonumber\\
      \end{align}
      where both sides of the equality are to be seen as elements of the
      algebraic completion of \(G((\Mmod \otimes \Nmod) \otimes \Pmod)\) and
      the associativity map $\alpha_{\Mmod, \Nmod, \Pmod}$ is to be seen as an element of
       $ \Homgrp{}{\Mmod \otimes (\Nmod \otimes
        \Pmod)}{(\Mmod \otimes \Nmod) \otimes \Pmod}$ so that
      $\btfi{\alpha_{\Mmod, \Nmod, \Pmod}}{z}$ is an intertwiner of type
      $\binom{G\brac*{(\Mmod \otimes \Nmod) \otimes \Pmod}}{G(\Mmod), G(\Nmod
        \otimes \Pmod)}$.
    \end{itemize}
    \label{itm:voptransf}
    \end{enumerate}
    The linear maps \(\btf\) and the natural transformation $\varphi_2$ uniquely
    characterise each other through the equality \(\btfi{\id_{\Mmod\otimes\Nmod}}{z}=\tstr_2(\Mmod,\Nmod)\circ\ifld{G(\Mmod),
      G(\Nmod)}{-}{z}\), where \(\Mmod,\Nmod\in \Dcat\) and where \(\iop{G(\Mmod), G(\Nmod)}\) is the
    intertwining operator of the universal property \eqref{eq:fusuprop} characterising
    \(G(\Mmod)\fuse G(\Nmod)\).
    \label{thm:btcvoa}
\end{lem}

\begin{rmk}
  By the functoriality condition above, the linear maps \(\btf\) are
  completely determined by their values on \(\id_{\Mmod \otimes \Nmod}\in
  \mathrm{Hom}_{\Dcat}(\Mmod \otimes \Nmod, \Mmod\otimes \Nmod)\). If
  \(\btf_{\id_{\Mmod \otimes \Nmod}}\in
  \grispc{G(\Mmod)}{G(\Nmod)}{G(\Mmod\otimes\Nmod)}\) for all \(\Mmod,\Nmod\in
  \Dcat\), then all \(\btf_{f}\) will be graded intertwining operators, since
  all morphisms in \(\Ccat\) preserve the grading.
  Further,
  note that for each of the equations \eqref{eq:units}, \eqref{eq:skewsym},
    \eqref{eq:assoc}, the \lhs{s} and \rhs{s} are respectively in the same
    space of intertwining operators. If these spaces of intertwining operators
    are finite dimensional, then it is sufficient to verify the equation for
    only a finite number of coefficients. In particular, if the intertwining
    operator space is one dimensional then it is 
    sufficient to compare the leading coefficients. 
\end{rmk}

\begin{proof}
  Note that if the family of maps \(\btf\) in \eqref{eq:Gmap} exists, then since
  \(\id_{\Mmod\otimes \Nmod}\in \Homgrp{\Dcat}{\Mmod\otimes
    \Nmod}{\Mmod\otimes \Nmod}\), it follows that \(\btf_{\id_{\Mmod\otimes
      \Nmod}}\) is an intertwining operator of type
  \(\ityp{G(\Mmod)}{G(\Nmod)}{G(\Mmod\otimes \Nmod)}\). Let
  \(\iop{G(\Mmod),G(\Nmod)}\) be the universal intertwining operator of type
  \(\ityp{G(\Mmod)}{G(\Nmod)}{G(\Mmod)\fuse G(\Nmod)}\) coming from
  universal property \eqref{eq:fusuprop} characterising
  \(G(\Mmod)\fuse G(\Nmod)\). This universal property further implies the
  existence and uniqueness of a family of morphisms
  \(\tstr_2(\Mmod,\Nmod)\in \Homgrp{\Ccat}{G(\Mmod)\fuse
    G(\Nmod)}{G(\Mmod\otimes \Nmod)}\) satisfying
  \(\btfi{\id_{\Mmod\otimes\Nmod}}{z}=\tstr_2(\Mmod,\Nmod)\circ\ifld{G(\Mmod),G(\Nmod)}{-}{z}\).
  Conversely given a family of morphisms \(\tstr_2(\Mmod,\Nmod):G(\Mmod)\fuse
    G(\Nmod) \to G(\Mmod\otimes \Nmod)\), we can define
    \(\btf_{\id_{\Mmod\otimes\Nmod}}\) via
    \(\btf_{\id_{\Mmod\otimes\Nmod}}\brac*{z}=\tstr_2(\Mmod,\Nmod)\circ\iop{G(\Mmod),G(\Nmod)}\).

  We show the logical equivalence of Assertions \ref{itm:ntransf} and
  \ref{itm:voptransf} by respectively showing the equivalence of naturality of
  \(\tstr_2\) and functoriality of \(\btf\); the left unit
  square constraint for \(\tstr_2\) commuting and the unitality of \(\btf\); the
  braiding square constraint for \(\tstr_2\) commuting and the skew symmetry of
  \(\btf\); and the associativity hexagon constraint for \(\tstr_2\) commuting and the associativity of
  \(\btf\). Note that the right unit square constraint does not need to be
  verified, since it is implied by the left unit and braiding.

  Assume \(\btf\) is functorial, and $g : \Mmod^\prime \ra \Mmod$, $h : \Nmod^\prime \ra
  \Nmod$, then
  \begin{multline}
    \brac*{G(g\otimes h)}\circ \tstr_2(\Mmod^\prime,\Nmod^\prime)\circ \iop{G(\Mmod),G(\Nmod)}=
    G(g\otimes h)\circ \btf_{\id_{\Mmod^\prime\otimes\Nmod^\prime}}=\btf_{g\otimes
      h}=\btf_{\id_{\Mmod\otimes\Nmod}}\circ
    G(g)\otimes_\CC G(h)\\=\tstr_2(\Mmod,\Nmod)\circ \iop{G(\Mmod),G(\Nmod)}\circ
    G(g)\otimes_\CC G(h)=
    \tstr_2(\Mmod,\Nmod)\circ\brac*{G(g)\fuse G(h)} \iop{G(\Mmod),G(\Nmod)}.
  \end{multline}
  Thus \(\brac*{G(g\otimes h)}\circ \tstr_2(\Mmod^\prime,\Nmod^\prime)=
  \tstr_2(\Mmod,\Nmod)\circ\brac*{G(g)\fuse G(h)}\) and hence \(\tstr_2\) is natural.

  Conversely, assume \(\tstr_2\) is natural.
  As noted above, We
  first define \(\btf\) on identity morphisms \(\id_{\Mmod\otimes\Nmod}\)
  by \(\btf_{\id_{\Mmod\otimes\Nmod}}=\tstr_2(\Mmod,\Nmod)\circ
  \iop{G(\Mmod),G(\Nmod)}\) and extend functorially, that is for
  \(f\in\Homgrp{\Dcat}{\Mmod\otimes\Nmod}{\Pmod}\),
  $g : \Mmod^\prime \ra \Mmod$, $h : \Nmod^\prime \ra \Nmod$, $k : \Pmod \ra \Pmod^\prime$,
  \begin{equation}
    \btfi{k\circ f\circ\brac*{g\otimes h}}{z}=G(k)\circ
    G(f)\circ\btfi{\id_{\Mmod\otimes\Nmod}}{z}\circ\brac*{G(g)\otimes_\CC
      G(h)} .
    \label{eq:GTformula}
  \end{equation}
  This is well defined if and only if
  \(G(g\otimes h)\circ \btfi{\id_{\Mmod^\prime\otimes \Nmod^\prime}}{z}
  =\btfi{\id_{\Mmod\otimes \Nmod}}{z}\circ G(g)\otimes_\CC G(h)\).
  Consider
  \begin{align}
    &G(g\otimes h)\circ\btfi{\id_{\Mmod^\prime\otimes\Nmod^\prime}}{z}=
        G(g\otimes
      h)\circ\tstr_2(\Mmod^\prime,\Nmod^\prime)\circ\ifld{G(\Mmod^\prime),G(\Nmod^\prime)}{- }{z}
      =
      \tstr_2(\Mmod,\Nmod)\circ\brac*{G(g)\fuse G(h)}\circ
      \ifld{G(\Mmod^\prime),G(\Nmod^\prime)}{- }{z}\nonumber\\
    &\qquad=
      \tstr_2(\Mmod,\Nmod)\circ\ifld{G(\Mmod),G(\Nmod)}{-}{z}\circ G(g)\otimes_\CC G(h) ,
  \end{align}
  where the second equality uses that
  \(\tstr_2\) is natural and the third uses the definition of the tensor product of
  morphisms in \(\Ccat\).
  Hence the formula \eqref{eq:GTformula} is well defined and \(\btf\) is
  functorial.
  For the remainder of the proof we will assume that \(\tstr_2\) natural and
  hence also that \(\btf\) is functorial.
      
  We next show the logical equivalence of the left unit constraint for
  \((G,\tstr_0,\tstr_2)\) and the unitality of \(\btf\).
   Consider the following squares.
    \begin{equation}
      \begin{tikzcd}[row sep=10mm, column sep=20mm]
        V \otimes_\CC G(\Mmod)\arrow[d,"\tstr_0\otimes_\CC\id_{G(M)}"]\arrow[r,"\iop{V,G(M)}"]
        &V \fuse G(\Mmod) \arrow[r, "l_{G(\Mmod)}"] \arrow[d, "\tstr_0 \fuse \id_{G(M)}"] 
        &  G(\Mmod)  \\
        G(1_\Dcat) \otimes_\CC G(\Mmod)\arrow[r,"\iop{G(1_\Dcat),G(M)}"]
        &G(1_\Dcat) \fuse G(\Mmod) \arrow[r, "\tstr_{2}(1_{\Dcat}{,}\Mmod)"]
        & G(1_\Dcat \otimes \Mmod) \arrow[u, "G(l_\Mmod)"]
      \end{tikzcd}
      \label{eq:unitsquare}
    \end{equation}
    Note that we have suppressed formal variables in the 
    images of intertwining operators for visual clarity.
    The left square commutes by the definition (see \cref{thm:gvstr}) of how the functor \(\fuse\)
    is evaluated on pairs of morphisms in \(\Ccat\). Consider the two sequences
    of equalities
    \begin{subequations}
      \begin{align}
        &G(l_\Mmod)\circ\tstr_2(1_\Dcat,\Mmod)\circ\brac*{\tstr_0\fuse\id_{G(\Mmod)}}\brac*{
        \ifld{V,G(\Mmod)}{v}{z}m}=
        G(l_\Mmod)\circ\tstr_2(1_\Dcat,\Mmod)\brac*{
                               \ifld{G(1_\Dcat),G(\Mmod)}{\tstr_0(v)}{z}m}                \nonumber\\
        &\qquad=G(l_\Mmod)\circ \btfi{\id_{1_\Dcat\otimes
          \Mmod}}{\tstr_0(v),z}m
          =\btfi{l_\Mmod}{\tstr_0(v),z}m,\label{eq:unita}\\
        &l_{G(M)}\brac*{\ifld{V,G(M)}{v}{z}m}=Y_{G(\Mmod)}(v,z)m.\label{eq:unitb}
      \end{align}
    \end{subequations}
    The first equality of \eqref{eq:unita} follows from the definition of \(\fuse\)
    evaluated on pair of morphisms, the second from the identity
    relating \(\btf\) and \(\tstr_2\), the third from the functoriality of
    \(\btf\), while \eqref{eq:unitb} is the defining property of left unit morphisms in \(\Ccat\).
      If we assume that \(\btf\) is unital, then the last expressions of
      \eqref{eq:unita} and \eqref{eq:unitb} are equal and hence the first
      terms must be too, implying the commutativity of the right square in
      \eqref{eq:unitsquare}.
      Conversely, if we assume that the right square in
      \eqref{eq:unitsquare} commutes, then the first expressions of \eqref{eq:unita} and \eqref{eq:unitb} are equal and hence the last
      terms must be too. Thus \(\btf\) is unital.

      We next show the logical equivalence of the braiding constraint for
      \((G,\tstr_0,\tstr_2)\) and the skew symmetry of \(\btf\).
      Consider the following squares.
      \begin{equation}
        \begin{tikzcd}[row sep=10mm, column sep=25mm]
          G(\Mmod)\otimes_\CC G(\Nmod)\arrow[r,"\iop{G(\Mmod),G(\Nmod)}"]\arrow[d,"P"]
          &G(\Mmod) \fuse G(\Nmod) \arrow[d, "c_{G(\Mmod),G(\Nmod)}"] \arrow[r, "\tstr_2(\Mmod{,}\Nmod)"] 
          & {G(\Mmod \otimes \Nmod)} \arrow[d, "G(c_{\Mmod,\Nmod})"]\\
          G(\Nmod)\otimes_\CC G(\Mmod)\arrow[r,"\iop{G(\Nmod),G(\Mmod)}"]
          &G(\Nmod) \fuse G(\Mmod) \arrow[r, "\tstr_2(\Nmod{,}\Mmod)"] 
          & {G(\Nmod \otimes \Mmod)} \label{eq:braid}
        \end{tikzcd}
      \end{equation}
      where \(P\) is the tensor flip.
      We have again suppressed formal variables. The left square
      commutes by the definition (see \cref{thm:gvstr}) of braiding for intertwining
      operators.
      Consider the two sequences
      of equalities
      \begin{subequations}
        \begin{align}
          &\tstr_2(\Nmod,\Mmod)\circ c_{G(\Mmod),G(\Nmod)}\brac*{
           \ifld{G(\Mmod),G(\Nmod)}{m,z}n}=
           \tstr_2(\Nmod,\Mmod)\brac*{\ee^{zL_{-1}}\ifld{G(\Nmod),G(\Mmod)}{n,\ee^{\ii\pi}z}m}\nonumber\\
        &\qquad=\ee^{zL_{-1}}\tstr_2(\Nmod,\Mmod)\brac*{\ifld{G(\Nmod),G(\Mmod)}{n,\ee^{\ii\pi}z}m}
          \ee^{zL_{-1}}\btfi{\id_{\Nmod\otimes
          \Mmod}}{n,\ee^{\ii\pi}z}m , \label{eq:skewa}\\
          &G(c_{\Mmod,\Nmod})\circ\tstr_2(\Mmod,\Nmod)\circ\ifld{G(\Mmod),G(\Nmod)}{m}{z}n=
          G(c_{\Mmod,\Nmod})\circ\btfi{\id_{\Mmod\otimes \Nmod}}{m,z}n=
          \btfi{c_{\Mmod,\Nmod}}{m,z}n . \label{eq:skewb}
        \end{align}
      \end{subequations}
      As for the unitality argument above, the equalities follow from the
      defining properties of the tensor structures in \(\Ccat\) and the
      functoriality of \(\btf\) or naturality of \(\tstr_2\). Note that \(\tstr_2(\Nmod,\Mmod)\) is a module map and
      hence commutes with \(L_{-1}\).
      If we assume that \(\btf\) is skew symmetric, then the last expressions
      of \eqref{eq:skewa} and \eqref{eq:skewb} are equal and hence the first
      are too. Thus the right square in \eqref{eq:braid} commutes. Conversely,
      if the right square in \eqref{eq:braid} commutes, then the first
      expressions of \eqref{eq:skewa} and \eqref{eq:skewb} are equal and hence the last
      are too, implying the skew symmetry of \(\btf\).
      
      Finally we show the equivalence of the associativity hexagon condition for $\tstr_2$, and
      the associativity of \(\btf\). Consider the following
      triangle and hexagon.
      \begin{equation}
        \begin{tikzcd}[row sep=7mm, column sep=11mm]
          G(\Mmod) \otimes_\CC G(\Nmod) \otimes_\CC G(\Pmod)
          \arrow[r,"\iop{}\brac*{\ }\iop{}\brac*{\ }"]
          \arrow[dr,"\iop{}\brac*{\iop{}\brac*{\ }}"]
          &G(\Mmod) \fuse \brac*{G(\Nmod) \fuse G(\Pmod)}
          \arrow[d,"\alpha_{\Ccat}"] \arrow[r, "\id \fuse \tstr_2"] &
          G(\Mmod) \fuse G(\Nmod \otimes \Pmod)
          \arrow[r, "\tstr_2"] &
          G\brac*{\Mmod \otimes \brac{\Nmod \otimes \Pmod}}
          \arrow[d, "G\brac*{\alpha_\Dcat}"]\\
          &\brac*{G(\Mmod) \fuse G(\Nmod)} \fuse G(\Pmod)
          \arrow[r, "\tstr_2 \fuse \id"]  &
          G\brac*{\Mmod \otimes \Nmod} \fuse G(\Pmod)
          \arrow[r, "\tstr_2"] &
          G\brac*{(\Mmod \otimes \Nmod) \otimes \Pmod} 
        \end{tikzcd}
        \label{eq:assocdiag}
      \end{equation}
      Here \(\iop{}\brac*{\ }\iop{}\brac*{\ }\) and
      \(\iop{}\brac*{\iop{}\brac*{\ }}\) denote the obvious product and
      iterate of intertwining operators and we have suppressed the objects
      labelling the natural transformations \(\tstr_2, \alpha_\Ccat,
      \alpha_\Dcat\).
      The left triangle
      commutes by the definition (see \cref{thm:gvstr}) of associativity for intertwining
      operators.
      Let \(m\in G(\Mmod)\), \(n\in G(\Nmod)\), \(p\in G(\Pmod)\), \(x_1,x_2\in \CC\),
      \(|x_1|>|x_2|>0\), \(|x_2|>|x_1-x_2|>0\) and
      consider the two sequences
      of equalities
      \begin{subequations}
        \begin{align}
          &G(\alpha_{\Mmod,\Nmod,\Pmod})\circ\tstr_2(\Mmod,\Nmod\otimes\Pmod)\circ\brac*{\id_{G(M)}\fuse
          \tstr_2(\Nmod,\Pmod)}\circ\ifld{G(\Mmod),G(\Nmod)\fuse
          G(\Pmod)}{m}{x_1}\ifld{G(\Nmod),G(\Pmod)}{n}{x_2}p \nonumber\\
        &\qquad=
        G(\alpha_{\Mmod,\Nmod,\Pmod})\circ\tstr_2(\Mmod,\Nmod\otimes\Pmod)\circ
        \ifld{G(\Mmod),G(\Nmod\otimes
          \Pmod)}{m}{x_1}\tstr_2(\Nmod,\Pmod)\circ\ifld{G(\Nmod),G(\Pmod)}{n}{x_2}p
          \nonumber\\
        &\qquad=
          G(\alpha_{\Mmod,\Nmod,\Pmod})\circ\btfi{\id_{\Mmod\otimes(\Nmod\otimes\Pmod)}}{m,x_1}
          \btfi{\id_{\Nmod\otimes\Pmod}}{n,x_2}p=
          \btfi{\alpha_{\Mmod,\Nmod,\Pmod}}{m,x_1}
          \btfi{\id_{\Nmod\otimes\Pmod}}{n,x_2}p,\label{eq:assoca}\\
          &\tstr_2(\Mmod\otimes\Nmod,\Pmod)\circ\brac*{\tstr_2(\Mmod,\Nmod)\fuse\id_{G(\Pmod)}}\circ
        \alpha_{G(\Mmod),G(\Nmod),G(\Pmod)}\circ \ifld{G(\Mmod),G(\Nmod)\fuse
            G(\Pmod)}{m}{x_1}\ifld{G(\Nmod),G(\Pmod)}{n}{x_2}p\nonumber\\
          &\qquad=
          \tstr_2(\Mmod\otimes\Nmod,\Pmod)\circ\brac*{\tstr_2(\Mmod,\Nmod)\fuse\id_{G(\Pmod)}}\circ
          \ifld{G(\Mmod\otimes\Nmod),G(\Pmod)}
            {\ifld{G(\Mmod),G(\Nmod)}{m}{x_1-x_2}n}{x_2}p\nonumber\\
          &\qquad=
          \tstr_2(\Mmod\otimes\Nmod,\Pmod)\circ
            \ifld{G(\Mmod\otimes\Nmod),G(\Pmod)}
            {\tstr_2(\Mmod,\Nmod)\circ\ifld{G(\Mmod),G(\Nmod)}{m}{x_1-x_2}n}{x_2}p\nonumber\\
          &\qquad=\btfi{\id_{(\Mmod\otimes\Nmod)\otimes\Pmod}}{\btfi{\id_{\Mmod\otimes\Nmod}}{m,x_1-x_2}n,x_2}p.
            \label{eq:assocb}
        \end{align}
      \end{subequations}
      As with the arguments for the previous commutative diagrams, the
      equivalence of \(\btf\) being associative and the hexagon in
      \eqref{eq:assocdiag} commuting follows by recognising the equality of either
      the first or last terms of \eqref{eq:assoca} and \eqref{eq:assocb}.

      Assertions \ref{itm:ntransf} and \ref{itm:voptransf} are therefore equivalent.
  \end{proof}
  
  \begin{cor}
    Let \(\Ccat,\ \Dcat,\ G\) and \(\tstr_0\) be as in \cref{thm:btcvoa}. Further,
    assume \(\tstr_0\) is an isomorphism and that
    there exists a natural transformation \(\tstr_2:G(\Mmod)\fuse
      G(\Nmod)\to G(\Mmod \otimes\Nmod)\)
    such that \((G,\tstr_0,\tstr_2)\) is a  braided monoidal functor.
    Then \(\tstr_2\) is a natural isomorphism (equivalently
    \((G,\tstr_0,\tstr_2)\) is a strong  braided monoidal functor) if either of
    the following sets of sufficient conditions are satisfied.
    \begin{enumerate}
    \item The unique morphism \(f_{\Mmod,\Nmod}\in  \Homgrp{\Ccat}{G(\Mmod)\fuse
        G(\Nmod)}{G(\Mmod\otimes \Nmod)}\) satisfying
  \(\btfi{\id_{\Mmod\otimes\Nmod}}{z}=f_{\Mmod,\Nmod}\circ\ifld{G(\Mmod),G(\Nmod)}{-}{z}\)
  is an isomorphism for all \(\Mmod,\Nmod\in \Dcat\).
    \item For all \(\Mmod,\Nmod\in \Dcat\), the objects \(G(\Mmod)\fuse
      G(\Nmod)\) and \(G(\Mmod \otimes\Nmod)\) are isomorphic, and
      \(\btfi{\id_{\Mmod\otimes\Nmod}}{z}\) is a surjective intertwining
      operator.
      \label{itm:intertweq}
    \end{enumerate}
    If in addition \(\Dcat\) is \rgv{} with dualising object \(K_\Dcat\) and
    twist \(\theta_\Dcat\), then a braided monoidal equivalence
    \((G,\tstr_0,\tstr_2)\) is a \rgv{} equivalence, if and only if
    \begin{equation}
      G(K_\Dcat)\cong V^\prime \qquad \text{and}\qquad
      G(\theta_\Dcat)=\ee^{2\pi\ii L_0}\vert_{G(-)}.
    \end{equation}
    \label{thm:tstrcor}
  \end{cor}

Due to the categories and functors above being abelian, the functor \(G\) and its
monoidal structure morphisms, or equivalently the family of linear maps
\(\btf\), distribute over direct sums. It is therefore sufficient to only
consider indecomposable modules \(\Mmod, \Nmod\) and \(\Pmod\) when verifying
the properties of \(\btf\). In practice general indecomposable \voa{} modules
can still be untractably complicated and so it would be convenient to not have
to consider all. For sufficiently well behaved functors and categories one can,
for example, restrict one's attention to projective modules only (and if
projective modules are sufficient, then indecomposable projectives are too), as we show in the
next proposition. This result will not be needed for the categories we
consider later in \cref{sec:freebosons}, 
because they will all be semisimple, however, the authors believe the
result to be of sufficient independent interest to warrant inclusion.
\begin{prop} \label{thm:projsufficient}
  Let $\Ccat$ and $\Dcat$ be abelian braided tensor categories
  both satisfying that there are  sufficiently many projectives, 
  that the tensor
  products are biexact and that projectives form a tensor ideal. 
  Let \(\Ccat^p\) and \(\Dcat^p\), respectively, be the full subcategories of
  projective objects of \(\Ccat\) and \(\Dcat\).
  Let $F : \Ccat \to \Dcat$ be an exact functor satisfying that the image of
  any projective object is projective and admitting an isomorphism $\tstr_0 :
  1_\Dcat  \to F(1_\Ccat)$. 
  The functors \(F\) and \(\otimes_\Ccat\) can therefore be restricted to \(\Ccat^p\)
  to obtain functors to \(\Dcat^p\).
  If there exists a natural isomorphism \(\tstr_2^p :
  F(-|_p) \otimes_\Dcat F(-|_p) \to F \circ\brac*{-|_p \otimes_{\Ccat} -|_p}\),
  where \(-|_p\) denotes the restriction to \(\Ccat^p\),
  such that for all $\Mmod, \ \Nmod, \ \Pmod \in
  \Ccat^p$, and a projective cover \(\Qmod \stackrel{\pi}{\rightarrow}  1_{\Ccat}\) of the unit object, the four diagrams below commute,
  then \(\tstr_2^p\) admits a unique extension \(\tstr_2\) to
  \(\Ccat\) and \((F,\tstr_0,\tstr_2)\) is a braided lax tensor functor. If
  \(\tstr_2^p\) is, in addition, a natural isomorphism, then \(\tstr_2\) is
  too and \((F,\tstr_0,\tstr_2)\) is a strong braided tensor functor.\\
  \vspace{1mm}
  \begin{equation}
    \centering
    \begin{tikzcd}[row sep=10mm, column sep=20mm]
      1_\Dcat \otimes_\Dcat F(\Mmod) \arrow[d, "l_{F(\Mmod)}"] 
      &  F(1_\Ccat) \otimes_\Dcat F(\Mmod) \arrow[l, "\tstr_0^{-1} \otimes_\Dcat \id_{F(\Mmod)}"] 
      & F(\Qmod) \otimes_\Dcat F(\Mmod) \arrow[l,"F(\pi_1) \otimes_\Dcat \id_{F(\Mmod)}"] \arrow[d, "\tstr_2^p\brac*{\Qmod, \Mmod}"] \\
      F(\Mmod) 
      & F(1_\Ccat \otimes_\Ccat \Mmod)  \arrow[l, "F(l_{\Mmod})"] 
      & F(\Qmod \otimes_\Ccat \Mmod)  \arrow[l,"F\brac*{\pi_1 \otimes_\Ccat \id_{F(\Mmod)}}"]
    \end{tikzcd}
    \label{eq:lunitrect}
  \end{equation}
  \vspace{1mm}
  \begin{equation}
    \centering
    \begin{tikzcd}[row sep=10mm, column sep=20mm]
      F(\Mmod) \otimes_\Dcat 1_\Dcat  \arrow[d, "r_{F(\Mmod)}"] 
      &  F(\Mmod) \otimes_\Dcat  F(1_\Ccat)\arrow[l, "\id_{F(\Mmod)} \otimes_\Dcat  \tstr_0^{-1}"] 
      &  F(\Mmod) \otimes_\Dcat F(\Qmod) \arrow[l,"\id_{F(\Mmod)} \otimes_\Dcat F(\pi_1) "] \arrow[d, "\tstr_2^p\brac*{\Mmod , \Qmod}"] \\
      F(\Mmod) 
      & F(\Mmod \otimes_\Ccat 1_\Ccat )  \arrow[l, "F(r_{\Mmod})"] 
      & F(\Mmod \otimes_\Ccat  \Qmod )  \arrow[l,"F\brac*{\id_{F(\Mmod)} \otimes_\Ccat \pi_1} "]
    \end{tikzcd}
    \label{eq:runitrect}
  \end{equation}
	\vspace{1mm}
	\begin{equation}
          \centering
          \begin{tikzcd}[row sep=10mm, column sep=11mm]
            F(\Mmod) \otimes_\Dcat F(\Nmod) \arrow[r, "c_{F(\Mmod), F(\Nmod)}"] \arrow[d, "\tstr_2^p\brac*{\Mmod, \Nmod}"] 
            & F(\Nmod) \otimes_\Dcat F(\Mmod) \arrow[d, "\tstr_2^p\brac*{\Nmod, \Mmod}"] \\
            {F(\Mmod \otimes_\Ccat \Nmod)} \arrow[r, "F\brac*{c_{\Mmod, \Nmod}}"] 
            & {F(\Nmod \otimes_\Ccat \Mmod)} 
          \end{tikzcd}%
          \hspace{-4mm}
          \begin{tikzcd}[row sep=7mm, column sep=17mm]
            F(\Mmod) \otimes_\Dcat \brac*{F(\Nmod) \otimes_\Dcat F(\Pmod)} \arrow[r, "\alpha_{F(\Mmod), F(\Nmod), F(\Pmod)}"] \arrow[d, "\id_{F(\Mmod)} \otimes_\Dcat \tstr_2^p\brac*{\Nmod, \Pmod}"] &
            \brac*{F(\Mmod) \otimes_\Dcat F(\Nmod)} \otimes_\Dcat F(\Pmod) \arrow[d, "\tstr_2^p\brac*{\Mmod, \Nmod} \otimes_\Dcat \id_{F(\Pmod)}"]  \\
            F(\Mmod) \otimes_\Dcat F(\Nmod \otimes_\Ccat \Pmod)  \arrow[d, "\tstr_2^p\brac*{\Mmod, \Nmod \otimes_\Ccat \Pmod}"] &
            F\brac*{\Mmod \otimes_\Ccat \Nmod} \otimes_\Dcat F(\Pmod) \arrow[d, "\tstr_2^p\brac*{\Mmod \otimes_\Ccat \Nmod, \Pmod}"] \\
            F\brac*{\Mmod \otimes_\Ccat \brac{\Nmod \otimes_\Ccat \Pmod}} \arrow[r, "F\brac*{\alpha_{\Mmod, \Nmod, \Pmod}}"] &
            F\brac*{(\Mmod \otimes_\Ccat \Nmod) \otimes_\Ccat \Pmod} 
          \end{tikzcd}
          \label{eq:brassocassum}
        \end{equation}
\end{prop}
\begin{proof}
  For $\Mmod,\Nmod \in \Ccat$ we define the family of morphisms
  $\tstr_2\brac*{\Mmod, \Nmod}:F(\Mmod)\otimes_\Dcat F(\Nmod)\to
  F(\Mmod\otimes_\Ccat\Nmod)$ from the definition of \(\tstr_2^p\) on projective objects. Choose projective resolutions $0 \leftarrow \Mmod\leftarrow \Mmod_0 \leftarrow \Mmod_1 \leftarrow\cdots $ and $0 \leftarrow \Nmod\leftarrow \Nmod_0\leftarrow \Nmod_1 \leftarrow\cdots $. 
  By assumption, $F$ is exact and maps projectives to projectives. This
  ensures that the images of the projective resolutions of \(\Mmod\) and
  \(\Nmod\) are projective resolutions of \(F(\Mmod)\) and
  \(F(\Nmod)\), respectively.
  Similarly, the biexactness of
  the tensor product and projectives forming a tensor ideal implies that the
  $\Mmod_i \otimes_\Ccat \Nmod_j$ are projective, and that the total complex
  \(\mathrm{Tot}^\oplus \brac*{\Mmod_\bullet\otimes_\Ccat \Nmod_\bullet}\) is
  a projective resolution of $\Mmod \otimes_\Ccat \Nmod$.
  Combining these two insights, we can form the two total complexes
  $\mathrm{Tot}^\oplus \brac*{F(\Mmod_\bullet)\otimes_\Dcat F(\Nmod_\bullet)}$ and
  $\mathrm{Tot}^\oplus \brac*{F(\Mmod_\bullet \otimes_\Ccat \Nmod_\bullet)}$ in
  \(\Dcat\), which
  give projective resolutions for $F(\Mmod) \otimes_\Dcat F(\Nmod)$ and $F(\Mmod \otimes_\Ccat \Nmod)$, respectively.
  We can therefore consider the following diagram, where the dashed arrow is
  the morphism we are seeking.
  \begin{equation}
    \begin{tikzcd}[row sep=10mm, column sep=6mm]
      0 & F(\Mmod) \otimes_\Dcat F(\Nmod) \arrow[d, dashed, "\tstr_2\brac*{\Mmod, \Nmod}"] \arrow[l]
      &&  F(\Mmod_0) \otimes_\Dcat F(\Nmod_0)  \arrow[ll, "d_0"] \arrow[d, "\tstr_2^p\brac*{\Mmod_0, \Nmod_0}"]
      &  \brac*{F(\Mmod_1) \otimes_\Dcat F(\Nmod_0)}  \oplus  \brac*{F(\Mmod_0) \otimes_\Dcat F(\Nmod_1)} \arrow[l, "d_1"] \arrow[d, "\tstr_2^p\brac*{\Mmod_1, \Nmod_0}\, \oplus \, \tstr_2^p\brac*{\Mmod_0, \Nmod_1}"] 
      & \cdots \arrow[l]  \\
      0 & F(\Mmod \otimes_\Ccat \Nmod) \arrow[l]
      && F(\Mmod_0 \otimes_\Ccat \Nmod_0)  \arrow[ll, "\delta_0"]	
      & \brac*{F(\Mmod_1 \otimes_\Ccat \Nmod_0)} \oplus \brac*{F(\Mmod_0 \otimes_\Ccat \Nmod_1)} \arrow[l, "\delta_1"]
      & \cdots \arrow[l]
    \end{tikzcd}
    \label{eq:restotcomps}
  \end{equation}
  By the assumed naturality of $\tstr_2^p$, the squares with solid edges in
  the diagram commute. Recall the following well known fact about abelian categories: 
  for any exact sequence \(0 \rightarrow \Amod\stackrel{\alpha}{\longrightarrow} \Bmod \stackrel{\beta}{\longrightarrow} \Cmod
  \rightarrow 0\), every morphism
  \(\Bmod\stackrel{\widetilde{\beta}}{\longrightarrow} \widetilde{\Cmod}\)
  satisfying \(\tilde{\beta}\circ \alpha=0\) factorises uniquely over
  \(\beta\), that is, there exists a unique morphism
  \(\Cmod\stackrel{\gamma}{\longrightarrow}\widetilde{\Cmod}\) such that
  \(\tilde{\beta}=\gamma\circ \beta\). If we set \(\alpha=d_1\), \(\beta=d_0\)
  and \(\tilde{\beta}=\delta_0\circ \tstr_2^p\brac*{\Mmod_0, \Nmod_0}\), then
  we see that
  \begin{equation}
    \tilde{\beta}\circ \alpha =\delta_0\circ \tstr_2^p\brac*{\Mmod_0,
      \Nmod_0}\circ d_1=\delta_0\circ \delta_1\circ \brac*{\tstr_2^p\brac*{\Mmod_1, \Nmod_0} \oplus  \tstr_2^p\brac*{\Mmod_0, \Nmod_1}}=0,
  \end{equation}
  where the second equality uses the commutativity of the right square in
  \eqref{eq:restotcomps} and the third uses the exactness of the lower row of
  \eqref{eq:restotcomps}. Thus \(\tstr_2(\Mmod,\Nmod)\) is the unique
  morphism factorising \(\delta_0\circ \tstr_2^p\brac*{\Mmod_0, \Nmod_0}\) over \(d_0\).

  Next we show  that this characterisation of \(\tstr_2\) is 
  independent of the projective resolutions:  suppose $0 \leftarrow
  \Mmod\leftarrow \Mmod_0^\prime\leftarrow \Mmod_1^\prime\leftarrow \cdots  $ is
  also a projective resolution of \(\Mmod\), then by the lifting lemma there
  exist chain maps $f_\bullet:\Mmod_\bullet\to \Mmod_\bullet^\prime$ and
  $f^\prime_\bullet:\Mmod_\bullet^\prime\to \Mmod_\bullet$ lifting the
  identity on $\Mmod$, and which are
  unique up to chain homotopy and inverse to each other up to chain
  homotopy.
  To show that this resolution of \(\Mmod\) gives the same morphism  $\tstr_2\brac*{\Mmod, \Nmod}$,
  we need to verify that
  \begin{equation}
    \tstr_2^p\brac*{\Mmod^\prime_0, \Nmod_0}\circ \brac*{F(f_0) \otimes_\Dcat F(\id_{\Nmod_0})}=
    F(f_0 \otimes_\Ccat \id_{\Nmod_0}) \circ \tstr_2^p\brac*{\Mmod_0, \Nmod_0}.
  \end{equation}
  This equality holds due to
  the assumed naturality of  $\tstr_2^p\brac*{\Mmod_0, \Nmod_0}$ on projectives. 
  The independence of \(\tstr_2\brac*{\Mmod,\Nmod}\) from choice of projective
  resolution of \(\Nmod\) follows in the same way.
  
  We now  show that \((F,\tstr_0,\tstr_2)\) satisfies the constraints of a
  braided lax tensor functor.  First, we show that the first two commutative diagrams lead to the unital constraints for projective objects. 
  Consider the commuting square \eqref{eq:lunitrect}, where we add the
  morphism \(\tstr_2\brac*{1_\Ccat, \Mmod}\) to the middle column.
  \begin{equation} \label{eq:unitproj}
    \centering
    \begin{tikzcd}[row sep=10mm, column sep=20mm]
      1_\Dcat \otimes_\Dcat F(\Mmod) \arrow[d, "l_{F(\Mmod)}"] 
      &  F(1_\Ccat) \otimes_\Dcat F(\Mmod) \arrow[l, "\tstr_0^{-1} \otimes_\Dcat \id_{F(\Mmod)}"] \arrow[d, dashed, "\tstr_2\brac*{1_\Ccat, \Mmod}"]
      & F(\Qmod) \otimes_\Dcat F(\Mmod) \arrow[l,"F(\pi) \otimes_\Dcat \id_{F(\Mmod)}"] \arrow[d, "\tstr_2^p\brac*{\Qmod, \Mmod}"] \\
      F(\Mmod) 
      & F(1_\Ccat \otimes_\Ccat \Mmod)  \arrow[l, "F(l_{\Mmod})"] 
      & F(\Qmod \otimes_\Ccat \Mmod)  \arrow[l,"F\brac*{\pi_1 \otimes_\Ccat \id_{F(\Mmod)}}"]
    \end{tikzcd}
  \end{equation}
  The right square commutes by the naturality of \(\tstr_2\)
  and the outer rectangle commutes by assumption. Therefore
  \begin{align}
    F(l_\Mmod)\circ \tstr_2(1_\Ccat,\Mmod)\circ \brac*{F(\pi)\otimes_\Dcat\id_{F(\Mmod)}}&=
    F(l_\Mmod)\circ F(\pi\otimes_\Ccat \id_{F(\Mmod)})\circ \tstr_2(Q,\Mmod)\nonumber\\
    &= l_{F(\Mmod)}\circ \tstr_0^{-1}\otimes_\Dcat\id_{F(\Mmod)}\circ \brac*{F(\pi)\otimes_\Dcat\id_{F(M)}}.
  \end{align}
  Since \(F(\pi)\otimes_\Dcat\id_{F(M)}\) is an epimorphism \(F(l_\Mmod)\circ
  \tstr_2(1_\Ccat,\Mmod)=l_{F(\Mmod)}\circ \tstr_0^{-1}\otimes_\Dcat\id_{F(\Mmod)}\)
  and the hence the left square commutes. A
  similar argument holds for the right unit. 
  
  Next we show that \(\tstr_2\) satisfies the constraints of the braiding square
    and the associativity hexagon. For any \(\Mmod,\Nmod,\Pmod\in\Ccat\) and
    choices of projective resolutions for these objects consider the following diagrams.
    \begin{equation}\label{eq:projsuf2}
      \centering
      \adjustbox{scale=0.7}{
        \begin{tikzcd}[row sep=20, column sep=1]
          & F(\Nmod) \otimes_\Dcat F(\Mmod)  \ar{dd} & &  F(\Nmod_0) \otimes_\Dcat F(\Mmod_0) \ar{ll} \ar{dd}  \\
          F(\Mmod) \otimes_\Dcat F(\Nmod) \ar{ur} \ar{dd} & & F(\Mmod_0) \otimes_\Dcat F(\Nmod_0) \ar{ur} \ar[crossing over]{ll} \\
          & F(\Nmod \otimes_\Ccat \Mmod)  & &  F(\Nmod_0 \otimes_\Ccat \Mmod_0)  \ar{ll}  \\
          F(\Mmod \otimes_\Ccat \Nmod) \ar{ur} & & F(\Mmod_0 \otimes_\Ccat \Nmod_0) \ar{ur} \ar{ll} \ar[from=uu,crossing over]
        \end{tikzcd}}%
      \adjustbox{scale=0.7}{
        \begin{tikzcd}[row sep=5, column sep=-20]
          & F(\Mmod) \otimes_\Dcat (F(\Nmod) \otimes_\Dcat F(\Pmod)) \ar{dd} & &  F(\Mmod_0) \otimes_\Dcat (F(\Nmod_0) \otimes_\Dcat F(\Pmod_0)) \ar{ll} \ar{dd}  \\
          (F(\Mmod) \otimes_\Dcat F(\Nmod)) \otimes_\Dcat F(\Pmod) \ar{ur} \ar{dd} & & (F(\Mmod_0) \otimes_\Dcat F(\Nmod_0)) \otimes_\Dcat F(\Pmod_0) \ar{ur} \ar[crossing over]{ll} \\
          &  F(\Mmod) \otimes_\Dcat F(\Nmod \otimes_\Ccat \Pmod) \ar{dd} & &  F(\Mmod_0) \otimes_\Dcat F(\Nmod_0 \otimes_\Ccat \Pmod_0) \ar{dd} \ar{ll}  \\
          F(\Mmod \otimes_\Ccat \Nmod) \otimes_\Dcat F(\Pmod) \ar{dd}  & & F(\Mmod_0 \otimes_\Ccat \Nmod_0) \otimes_\Dcat F(\Pmod_0) \ar{dd} \ar[crossing over]{ll} \ar[from=uu,crossing over] \\
          & F(\Mmod \otimes_\Ccat (\Nmod \otimes_\Ccat \Pmod))  & &  F(\Mmod_0 \otimes_\Ccat (\Nmod_0 \otimes_\Ccat \Pmod_0)) \ar{ll}  \\
          F((\Mmod \otimes_\Ccat \Nmod) \otimes_\Ccat \Pmod)  \ar{ur} & & F((\Mmod_0 \otimes_\Ccat \Nmod_0) \otimes_\Ccat \Pmod_0) \ar{ur} \ar{ll} \ar[from=uu,crossing over]
        \end{tikzcd}}
    \end{equation}
    The left faces are, respectively, the braiding and associativity
    constraints whose commutativity we need to show. The
    right faces are the same diagrams evaluated on the first projective
    coefficients of the appropriate projective resolutions formed by taking
    total complexes, while the horizontal arrows are those of the
    these resolutions. Note that the right faces
    commute by assumption \eqref{eq:brassocassum}, as they are evaluated on
    projective objects, while the front and back commute by the naturality of \(\tstr_2\).
    We present the detailed argument for the braiding square. The argument for the associativity hexagon is similar.
  Consider the following paths through the braiding diagram.
  \begin{equation}
    \centering
    \adjustbox{scale=0.7}{
      \begin{tikzcd}[row sep=20, column sep=1]
        & F(\Nmod) \otimes_\Dcat F(\Mmod) \ar{dd} & &  F(\Nmod_0) \otimes_\Dcat F(\Mmod_0) \ar[dashed]{dd}  \\
        F(\Mmod) \otimes_\Dcat F(\Nmod) \ar{ur}  & & F(\Mmod_0) \otimes_\Dcat F(\Nmod_0)
        \ar[dashed]{ur} \ar[crossing over, "d", near start]{ll} \\
        & F(\Nmod \otimes_\Ccat \Mmod)  & &  F(\Nmod_0 \otimes_\Ccat \Mmod_0)
        \ar[dashed,"\delta"]{ll}  \\
        F(\Mmod \otimes_\Ccat \Nmod) & & F(\Mmod_0 \otimes_\Ccat \Nmod_0) 
      \end{tikzcd}}%
    \hspace{10mm}
    \adjustbox{scale=0.7}{
      \begin{tikzcd}[row sep=20, column sep=1]
        & F(\Nmod) \otimes_\Dcat F(\Mmod)   & &  F(\Nmod_0) \otimes_\Dcat F(\Mmod_0)  \\
        F(\Mmod) \otimes_\Dcat F(\Nmod) \ar{dd} & & F(\Mmod_0) \otimes_\Dcat F(\Nmod_0)
        \ar[crossing over, "d"]{ll} \\
        &  F(\Nmod \otimes_\Ccat \Mmod)  & &  F(\Nmod_0 \otimes_\Ccat \Mmod_0) \ar[dotted,"\delta", near start]{ll}  \\
        F(\Mmod \otimes_\Ccat \Nmod) \ar{ur} & & F(\Mmod_0 \otimes_\Ccat \Nmod_0)
        \ar[dashed]{ur}  \ar[from=uu,crossing over,dotted]
      \end{tikzcd}}
  \end{equation}
  The solid paths in the left and right diagrams denote the compositions of
  maps $\tstr_2\brac*{\Nmod,\Mmod} \circ c_{F(\Mmod), F(\Nmod)}\circ d$ and
  $ F(c_{\Mmod, \Nmod}) \circ \tstr_2\brac*{\Mmod, \Nmod}\circ d$, respectively. The
  dashed paths denote analogous morphisms on the projective modules, that is, $\tstr_2\brac*{\Nmod_0,\Mmod_0} \circ
  c_{F(\Mmod_0), F(\Nmod_0)}$ and $F(c_{\Mmod_0, \Nmod_0})  \circ
  \tstr_2^p\brac*{\Mmod_0, \Nmod_0} $ in the left and right diagrams,
  respectively. By construction, all four morphisms \(F(\Mmod_0)\otimes_\Dcat
  F(\Nmod_0)\to F(\Nmod\otimes_\Ccat \Mmod)\) in the left and right diagrams are
  equal, in particular the morphisms consisting of compositions of solid
  arrows. Thus, since \(d\) is an epimorphism, the left face commutes.
\end{proof}

\begin{rmk}
  When the category \(\Dcat\) in \cref{thm:projsufficient} is a concrete
  category where the objects are at the very least abelian groups and the
  morphisms group homomorphisms, for example, a module category for a \voa{},
  it may desirable to have an actual formula for the natural transformation
  \(\tstr_2\) on non-projective modules. This can be done by considering the
  diagram \eqref{eq:restotcomps} and picking a set theoretic right inverse \(d_0^{-1}\) to
  the morphism \(d_0\) in \eqref{eq:restotcomps}, that is, \(d_0 \circ
  d_0^{-1}=\id_{F(\Mmod)\otimes_\Dcat F(\Nmod)}\). Note that in general
    \(d_0^{-1}\) cannot be chosen to be a morphism of \(\Dcat\). One can then
    define \(\tstr_2(\Mmod,\Nmod) = \delta_0\circ \tstr_2^p(\Mmod,\Nmod)\circ
    d_0^{-1}\). It is not hard to show that this formula gives a morphism in
    \(\Dcat\) and that it does not depend on the choice of right inverse \(d_0^{-1}\).
\end{rmk}

\section{The (chiral) free boson in three guises}
\label{sec:freebosons}

While the previous section was very general, we now change gears and consider a
specific family of monoidal categories, typically called 
free bosons in the context
of \voa{s}. We will, however, consider these free bosons in slightly greater
generality than is usually done in the literature (by allowing for different
choices of conformal structures) and we will show that they
admit \rgv{} structures.

\subsection{Lattice data for free bosons}
\label{sec:data}
Throughout this section, we will make frequent use of certain linear
algebraic and lattice data, whose structure we record here.
\begin{defn} \label{def:boslatt}
  A set of \emph{\data{}} is a quadruple \((\hvec,\pair{-}{-},\lat,\ffv)\), where
  \begin{itemize}
  \item $\hvec$ is a finite dimensional real vector space,
  \item $\pair{-}{-}$ is a non-degenerate symmetric real-valued bilinear form on $\hvec$,
  \item $\lat \subset \hvec$ is a lattice  (that is, a discreet
    subgroup of \(\hvec\)), which is even with respect to
    $\pair{-}{-}$,
  \item \(\ffv \in \qlat\) is a distinguished element called the
    \emph{Feigin-Fuchs boson}, where \(\dlat = \set{\mu \in \hvec \
      \vert \ \pair{\mu}{ \lat} \subset \ZZ}\).
  \end{itemize}
\end{defn}
Note also that we are not assuming that, \(\pair{-}{-}\) is positive definite, that \(\lat\) is non-trivial nor that
\(\pair{-}{-}\) restricted to \(\lat\) is non-degenerate. Further, if \(\lat\)
is not full rank, then \(\dlat\) will not be discreet.
For any set of \data{} 
we can always choose a section $s: \qlat \ra \dlat$, that is
\(\forall \rho\in \qlat\), \(s(\rho)\in\rho\) or in other words a map which
chooses a representative for each coset. Note that \(s\) will generally only
be a set theoretic section and not a group homomorphism. Additionally, we will always assume that
\(s(\lat)=0\in \dlat\).
From the section \(s\) we construct the associated $2$-cocycle
 \(k:\qlat\times \qlat\to \lat\), 
\(k(\mu,\nu)=s(\mu+\nu) - s(\mu) - s(\nu)\), \(\mu,\nu\in\dlat\), which
encodes the failure of \(s\) to be a group homomorphism.
Finally, let $\tcycsymb: \lat \times \lat \ra \CC^\times$ be a normalised
2-cocycle with commutator function
$C(\alpha,\beta)=(-1)^{\pair{\alpha}{\beta}}$, 
that is, $\tcycsymb$ 
satisfies the following conditions, for $\alpha, \beta, \gamma \in
\lat$. 
\begin{align} 
  \tcyc{\alpha}{0} &= \tcyc{0}{ \alpha} = 1,\quad
  \tcyc{\beta}{\gamma} \tcyc{\alpha + \beta}{\gamma}^{-1} \tcyc{\alpha}{\beta+
  \gamma} \tcyc{\alpha}{\beta}^{-1} = 1 ,\quad
  \tcyc{\alpha}{\beta} \tcyc{\beta}{\alpha}^{-1} = (-1)^{\pair{\alpha}{\beta}}.
  \label{eq:voacocycle}
\end{align}
An example of such a 2-cocycle can be constructed from any ordered
  \(\ZZ\)-basis \(\set{\alpha_i}\) of \(\lat\) by defining \(\tcycsymb\) to
  be the group homomorphism uniquely characterised by
  \begin{equation}
    \tcyc{\alpha_i}{\alpha_j}=
    \begin{cases}
      (-1)^{\pair{\alpha_i}{\alpha_j}}&\text{if}\ i<j\\
      1&\text{if}\ i\ge j
    \end{cases}\ .
  \end{equation}
  Note that in general \(\tcycsymb\) need not be a homomorphism, however,
since \(\lat\) is an abelian group, all choices of 2-cocycle are cohomologous.
The section and 2-cocycles will always be denoted \(s,\) \(k\) and
\(\tcycsymb\), respectively, for any
set of \data{}. They will be required for giving explicit formulae for certain
structures such as braiding and associativity isomorphisms.

Each set of \data{} will allow us to define a category of graded vectors
spaces, a vertex operator algebra module category and a quasi-Hopf algebra
module category all with a natural choice of 
ribbon \gv{} structure determined by these \data{}. 
Any such triple of
categories will be shown to be ribbon \gv{} equivalent provided their \data{}
are equal. Different choices of section \(s\) and 2-cocycle \(\tcycsymb\) will
yield equivalent categories, hence \(s\) and \(\tcycsymb\) are not data.

\begin{lem} \leavevmode
  Let \((\hvec,\pair{-}{-},\lat,\ffv)\)  be a set of \data{}.
  \begin{enumerate}
  \item Let \(\plat=\set{\mu \in \hvec \st \pair{\mu}{ \lat}= 0}\).
    There exists a finitely generated free abelian subgroup $\Gamma
    \subset \dlat$ such that $\dlat = \plat \oplus \Gamma$,
    where \(\oplus\) is the (internal) direct sum of \(\ZZ\)-modules.
    \label{itm:dlatdecomp}
  \item Let \(\nlat=\set{v\in \lat\st \pair{v}{w}=0, \forall w\in\lat}\).
    The subgroup \(\Gamma\) from above can be chosen in such a way that
    there exists a vector subspace \(V\in \plat\) and free finitely generated
    groups \(F,D\subset\Gamma\) such that all of the following hold.
    \begin{itemize}
    \item As an abelian group, $\dlat$ admits a direct sum decomposition
      \begin{equation}
        \dlat = V \oplus \rspn\nlat \oplus F \oplus D.
      \end{equation}
    \item The three subgroups $V$, $\rspn\nlat \oplus F$ and $D$ are mutually orthogonal.
    \item The restriction of \(\pair{-}{-}\) to each of the three subgroups
      $V$, $\rspn\nlat \oplus F$ and $D$ individually is non-degenerate.
    \item The restriction of \(\pair{-}{-}\) to $\rspn\nlat$ and  $F$ individually is trivial.
    \end{itemize}
    \label{itm:dlatdecomprefined}
  \end{enumerate}
  \label{thm:hbasis}
\end{lem}
\begin{proof}
  \begin{enumerate}
  \item If \(\Gamma\) exists, then it must be isomorphic to the quotient
    \(\dlat/\plat\), we therefore first need to show that \(\dlat/\plat\) is
    freely finitely generated. By definition \(\plat\) is the kernel of the surjective
    group homomorphism
    \begin{align}
      \psi:\dlat &\to \Homgrp{}{\lat}{\ZZ}, \nonumber\\
      \kappa &\mapsto  \pair{\kappa}{-}|_\lat.
    \end{align}
    Hence \(\dlat/\plat\cong \Homgrp{}{\lat}{\ZZ}\cong \lat\) which is freely
    finitely generated. Thus \(\dlat \to \dlat/\plat\) is a surjective
    homomorphism onto a free finitely generated \(\ZZ\) module with \(\plat\) as its kernel,
    hence \(\plat\) admits a free finitely generated direct sum complement in \(\dlat\).
  \item  Note that \(\lat/\nlat\) is torsion free. This can be seen by
  contradiction. If there was an element \(t\in \lat\setminus \nlat\) such that \(k t\in \nlat\)
  for some non-zero \(k\in \ZZ\), then \(0=\pair{kt}{w}=k\pair{t}{w}\) for
  all \(w\in\lat\), but
  this would imply \(t\in\nlat\). Since \(\lat/\nlat\) is torsion free,
  \(\nlat\) admits a direct sum complement \(\lat^c\) in \(\lat\).
  The directness of the sum \(\lat=\nlat\oplus\lat^c\) implies that \(\pair{-}{-}\)
  restricted to \(\lat^c\) is non-degenerate. Note that the subgroup
  \(\Gamma\) from above can be chosen such that \(\lat^c\subset
  \Gamma\). Define \(D=\Gamma\cap\rspn{\lat^c}\), then the restriction of
  \(\pair{-}{-}\) to \(D\) is non-degenerate, because it is non-degenerate on
  \(\lat^c\). Next define \(F=\set{t\in \Gamma\st \pair{t}{f}=0, \forall f\in
    D}\), \(V=\set{v\in \plat\st \pair{v}{t}=0, \forall t\in F}\) and
  \(W=\set{w\in \plat\st \pair{w}{v}=0, \forall v\in V}\).

  We show that \(\Gamma=F\oplus D\). Let \(\set{f_i}_{i=1}^{\rk D}\) be  a
  \(\ZZ\)-basis of \(D\). Since \(\pair{-}{-}\) is non-degenerate 
  on \(D\), there exists an \(\RR\)-basis \(\set{f^i}_{i=1}^{\rk D}\) of
  \(\rspn D\), which is dual to \(\set{f_i}_{i=1}^{\rk D}\), that is
  \(\pair{f_i}{f^j}=\delta_{i,j}\). Note that this implies that the \(f^i\)
  basis elements pair integrally with any element in \(D\).
  Consider \(v\in \Gamma\), then
  \begin{equation}
    \tilde{v}=\sum_{i=1}^{\rk D}\pair{v}{f^i}f_i \in D
  \end{equation}
  and for any \(f^j\) we have
  \begin{equation}
    \pair{v-\tilde{v}}{f^j}=\pair{v}{f^j}-\pair{v}{f^j}=0.
  \end{equation}
  Since all elements of \(D\) are \(\RR\)-linear combinations of the
  \(\RR\)-basis elements \(f^i\), this implies \(v-\tilde{v}\in F\)
  and hence \(v\in F + D\). Next consider \(v\in F\cap D\), then
  \(\pair{v}{f}=0\) for all \(f\in D\), but \(\pair{-}{-}\) is non-degenerate
  on \(D\), hence \(v=0\) and \(\Gamma=F\oplus D\).

  Note that \(D\) is orthogonal to \(F\) by construction and to \(\plat\)
  since \(D\subset \rspn\lat\). Thus \(V\) is orthogonal to \(D\), \(F\) and
  \(W\) and so \(\pair{-}{-}\) must be non-degenerate on \(V\) in order to be
  non-degenerate on \(\hvec\). By a similar argument as for \(D\) and \(F\),
  we therefore have that \(\plat=V\oplus W\).

  By construction \(\nlat\) is orthogonal to \(V\)
  (because it is orthogonal to \(\plat\)) and also \(\nlat\subset
  \plat\). Hence \(\rspn \nlat\subset W\). A brief counting of
  dimensions and ranks reveals \(\rk\nlat=\rk\lat 
  -\rk D=\dim W=\rk F\), implying that  \(\rspn \nlat = W\).
  Finally, by construction \(F\) is orthogonal to \(V\) and
  \(\Gamma\) hence, by the non-degeneracy of \(\pair{-}{-}\) on \(\hvec\), \(F\) must pair non-trivially with \(W\).
  \end{enumerate}
\end{proof}

\begin{rmk}
  The quotient group \(\qlat\) will feature prominently below. The
  decomposition in Part \ref{itm:dlatdecomprefined}, after observing
  that \(\lat=\nlat\oplus D\cap \lat\), implies the decomposition
  \begin{equation}
    \qlat = V\oplus \frac{\rspn\nlat}{\nlat}\oplus F\oplus\frac{D}{D\cap \lat}.
  \end{equation}
  Thus \(\qlat\) decomposes into an abelian Lie group with a vector space part
  \(V\) and a compact part
  \(\rspn\nlat/\nlat\), and a finitely generated group with a free part \(F\)
  and a finite part \(D/D\cap \lat\).
\end{rmk}

\begin{ex}
  We have the following natural examples to consider.
  \begin{enumerate}
  \item An empty lattice: \(\hvec=\RR^n\) and \(\lat=\set{0}\). In the
    decomposition of \cref{thm:hbasis}, we have \(\dlat=\hvec=V\),
    \(\Gamma=\nlat=F=D=\set{0}\). Hence \(\qlat\cong\hvec\) and \(\ffv\) can
    be any element in \(\hvec\). In this case there is only one choice of
    section \(s\), the canonical identification of \(\hvec/\set{0}\) with
    \(\hvec\), and \(k=0\).
  \item A full rank lattice:  \(\hvec=\RR^n\) and \(\lat\) a rank \(n\) even
    integral lattice. In the
    decomposition of \cref{thm:hbasis}, \(\plat =\set{0}\) and so
    \(\dlat=\Gamma=D\) is finitely generated. Further, \(\qlat\) is a finite
    group whose order is equal to the determinant (up to a sign) of the Gram
    matrix of the pairing in any choice of \(\ZZ\)-basis of \(\lat\). We can
    construct a section \(s\) by fixing a \(\ZZ\)-basis \(\set{e_i}\) of
    \(\dlat\). The image of this basis in \(\dlat\) will be a set of
    generators \(\set{\overline{e_i}}\) and each \(\mu \in\qlat\) has a unique
    expansion
    \(\mu=\sum_i a_i \overline{e_i}\) such that the coefficients \(a_i\) are
    minimal non-negative integers. Then \(s(\mu)=\sum_i a_i e_i\) is a choice
    of section.
  \item Half rank indefinite lattice: \(\hvec=\RR^2\) with pairing
    \(\pair{(x_1,x_2)}{(y_1,y_2)}=x_1y_2+x_2y_1\) and lattice
    \(\lat=\set{(0,m)\st m\in\ZZ}\). Then, in the
    decomposition of \cref{thm:hbasis}, \(\nlat=\lat\), \(\plat=\rspn\nlat=\set{(0,x)\st x\in\RR}\),
    \(\dlat=\set{(m,x)\st m\in \ZZ,\ x\in \RR}\cong \ZZ\times \RR\),
    \(\qlat\cong \ZZ\times \RR/\ZZ\), 
    \(V=D=\set{0}\) and \(F=\set{(0,m)\st m\in\ZZ}\). Since the pairing is
    trivial when restricted to \(\lat\), we can choose the 2-cocycle to be
    trivial, that is, \(\tcycsymb=1\). We choose
    \(\ffv=(1,0+\ZZ)\) for the Feigin-Fuchs boson, as this a convenient choice for the free field
    realisations of bosonic ghost systems. See \cite{WoOsp19} for an example.
    Finally, we can define a choice
    of section \(s\) by \(s(x,y+\ZZ)=(x,\tilde{y})\), where \(\tilde{y}\) is
    the unique representative of the coset \(y+\ZZ\) in the interval \([0,1)\).
  \end{enumerate}

\end{ex}

\subsection{Categories of vector spaces graded by abelian groups}
\label{sec:grvspaces}

\begin{defn}
  Let $\Vect_G$ denote the category of finite dimensional 
  complex vector spaces graded by an abelian group $G$, whose morphisms are all grade
  preserving linear maps.
  This category is semisimple
  with the isomorphism classes of simple objects represented by the one
  dimensional vector spaces $\CC_g$ which are $\CC$ at grade $g \in G$
  and trivial at other grades. 
\end{defn}
Note that if \(G\) is not finite, then objects in $\Vect_G$ will have only
finitely many non-trivial homogenous spaces.
We define a tensor product bifunctor on $\Vect_G$ by
asserting
\begin{equation}
  (\Mmod \otimes \Nmod)_g = \bigoplus_{h\in G}
  \Mmod_{g-h}\otimes_\CC \Nmod_{h},\qquad g\in G,\  \Mmod,\Nmod\in \Vect_G,
\end{equation}
where \(\otimes_\CC\) is the tensor product of complex vectors spaces and
having the tensor product of morphisms be that of linear maps. Further the
unit morphisms of vector spaces then also define unit morphisms for the
tensor functor \(\otimes\) on $\Vect_G$. The associativity and braiding
isomorphisms can then be defined on tensor products of the simple objects \(\CC_g\) to be scalar multiplies
of the vector space associator and tensor flip respectively. We shall denote
these scalar multiples as \(F\) and \(\Omega\) below.

\begin{thm}[Eilenberg and MacLane\cite{Mac52}, Joyal and Street \cite{Joy93}]
  Let \(G\) be an abelian group and $\Vect_G$ the category of finite
  dimensional \(G\) graded complex vector spaces with tensor functor and unit
  isomorphisms defined above.
  Then the braiding and associativity morphisms on $\Vect_G$ are in bijection
  with normalised abelian 3-cocycles $(F,\Omega)$, that is pairs of maps
  \(F:G\times G\times G\to \CC^\times\) and \(\Omega:G\times G\to \CC^\times\) characterised by the relations
  \begin{align}
    F(g+h,k,l)F(g,h,k+l)&=F(g,h,k)F(g,h+k,l)F(h,k,l),\qquad g,h,k,l\in G\nonumber\\
    F(h, k, g)^{-1} \Omega(g, h + k) F(g, h, k)^{-1} &= \Omega(g, k) F(h, g, k)^{-1}
                                             \Omega(g, h), \nonumber \\
    F(k, g, h) \Omega(g+h,k) F(g, h, k) &= \Omega(g, k) F(g, k, h) \Omega(h, k),
  \end{align}
  and additionally requiring that both maps evaluate to \(1\in\CC^\times\) if any argument is \(0\in
  G\).
  Inequivalent associativity and braiding structures are parametrised by the
  cohomology classes of the third abelian group cohomology $ H^3_{\text{ab}}(G,\CC^\times)$.
  The cohomology classes \(\omega=[(F,\Omega)]\in
  H^3_{\text{ab}}(G,\CC^\times)\) are uniquely characterised by their trace
  $\brac*{\tr{\omega}}(g) = \Omega(g, g) = q(g)$, which yields a quadratic
  form $q: G \ra \CC^{\times}$.
  \label{thm:vecasbr}
\end{thm}

Due to the above theorem, we denote by $\Vect^q_G$ the equivalence class of
braided tensor categories with structure characterised by $q$, and by
$\Vect_G^{(F, \Omega)}$ the specific representative category whose
associativity and braiding structures corresponds to the abelian 3-cocycle $(F,\Omega)$. 

\begin{prop}
  Let \((F,\Omega)\) be an abelian 3-cocycle and consider the braided tensor category $\Vect_{G}^{(F, \Omega)}$.
  \begin{enumerate}
  \item For any \(h\in G\), \(K=\CC_{h}\) is a dualising object and hence endows
    $\Vect_{G}^{(F, \Omega)}$ with the structure of a \gv{} category.
    \label{itm:dobj}
  \item The dualising functor corresponding to the choice of dualising object
    \(K=\CC_{h}\), \(h\in G\) is characterised by
    \(D(\Mmod)_g\simeq (\Mmod_{h-g})^\ast\), \(g\in G\), \(\Mmod\in
    \Vect_{G}^{(F, \Omega)}\), where \(\ast\) denotes the ordinary vector
    space dual.
    \label{itm:dfun}
  \item Every dualising object of $\Vect_{G}^{(F, \Omega)}$ is isomorphic to
    one of the simple objects
    \(\CC_{h}\) for some \(h\in G\).
    \label{itm:alldobj}
   \end{enumerate}
  Denote by $\Vect_{G}^{(F, \Omega,h)}$ the \gv{} category constructed
  from $\Vect_{G}^{(F, \Omega)}$  with dualising object \(K=\CC_{2h}\).
  \begin{enumerate}[resume]
  \item The \gv{} category $\Vect_{G}^{(F, \Omega,h)}$ admits a twist
    \(\theta\) by defining
      \begin{equation}
    \theta_{Q}|_{\Mmod_g} = Q(g) \id_{\Mmod_g},\qquad
    Q(g) = 
    \frac{\Omega(g-h,g-h)}{\Omega(-h,-h)},\qquad
    \Mmod\in \Vect_{G}^{(F, \Omega,h)},\ g\in G.
    \label{eq:weakqform}
  \end{equation}
  \label{itm:gvtwist}
  \end{enumerate}
  \label{thm:vecgv}
\end{prop}
\begin{proof}
We first show Part \ref{itm:dobj}. This can be computed directly by
  comparing the dimensions of morphism spaces or one can note the following. The category $\Vect_{G}^{(F, \Omega,h)}$
  is known to be rigid (the rigid dual of any simple object \(\CC_h, h\in G\)
  is \(C_h^\vee\cong \CC_{-h}\) and the evaluation and coevaluation maps are
  those of vector spaces) and hence the unit object \(\CC_0\) is dualising. The
  simple modules \(\CC_h, h\in G\) are all invertible. Thus, by
  \cref{thm:invdual}.\ref{itm:invtwist}, \(\CC_0\otimes (\CC_h)^{-1}\cong
  \CC_{-h}\) is also dualising. \cref{thm:invdual}.\ref{itm:invtwist} also immediately implies Part \ref{itm:alldobj}.
  Part \ref{itm:dfun} follows by noting that the proposed formula for \(D\)
  satisfies the defining relation \eqref{eq:GVprop} for the dualising object
  \(K=\CC_h\).
  Part \ref{itm:gvtwist} follows from the fact that the given formula for \(Q(g)\) satisfies the
  relations implied by \eqref{eq:twistbalancing} and
  \eqref{eq:dualisedtwist}. 
\end{proof}

\begin{rmk}
  Observe that the choice of \(\CC_{2h}\) as dualising object in
  $\Vect_{G}^{(F, \Omega,h)}$ excludes those simple objects not labelled by
  the double of a group element in \(G\). This is not an oversight; while every
  simple object in $\Vect_{G}^{(F, \Omega)}$ is a valid choice of dualising
  object (this follows from the category being rigid, hence the tensor unit
  is dualising, and 
  by \cref{thm:invdual}, we can shift by invertible objects), 
  simple objects not labelled by the double of a group element need
  not admit a twist which satisfies \(D(\theta)=\theta_{D(-)}\). Fortunately,
  the vertex operator algebraic constructions to be discussed below will
  always yield dualising objects that admit twists and make a preferred choice
  of twist.
  
  Functions of the form \(Q\) in Equation \eqref{eq:weakqform} are called
  weak quadratic forms centred at \(h\). It is interesting to note that (at
  least in the special case of \(G\) being a finite group) the
  \gv{} ribbon twists on $\Vect_{G}^{(F, \Omega,h)}$ are in bijection with
  such weak quadratic forms centred at \(h\), as was shown in Zetzsche's Masters thesis
  \cite[Theorem 4.2.2]{Zet18}. This classification of \rgv{} structures by weak quadratic forms is a
  generalisation of the classification of rigid braided tensor structures by
  quadratic forms, which corresponds to the special case $h = 0$
  for the dualising object. Note also that for \(h=0\) the category is ribbon.
\end{rmk}

Let \(\ldat=(\hvec,\pair{-}{-},\lat,\ffv)\) be a set of \data{} and recall the
decomposition \(\dlat=\plat\oplus \Gamma\) of \cref{thm:hbasis}.\ref{itm:dlatdecomp}.
We specialise the results of \cref{thm:vecasbr,thm:vecgv} using \(\ldat\). 
We choose the abelian group to be
\(G=\qlat\) and the quadratic form to be
\begin{equation}
  q(\cst{\alpha})=
  \ee^{\ii \pi \pair{\sct{\alpha}}{\sct{\alpha}}}, \quad \alpha \in \qlat,
\end{equation}
which defines the equivalence class of braided monoidal categories
\(\Vect^q_{\qlat}\). Note that this choice of quadratic form is independent of
the choice of section \(s\) due to \(\lat\) being even. Note further that
  \(\pair{\sct{\alpha}}{\sct{\alpha}}\) need not be integral and so we have
  chosen \(\ee^{\ii\pi}\) as a specific branch of logarithm for \(-1\). The section \(s\) then allows us to realise a
representative \(\Vect_{\qlat}^{(F, \Omega)}\) of \(\Vect^q_{\qlat}\), by
defining the abelian 3-cocycle to be
\begin{align}
  \Omega(\cst{\alpha}, \cst{\beta})&=\ee^{\ii \pi \pair{\sct{\alpha}}{\sct{\beta}}}, \qquad F(\cst{\alpha}, \cst{\beta}, \cst{\gamma}) = 
  (-1)^{\pair{\sct{\alpha}} {k(\beta, \gamma)}}\frac{\tcyc{k(\alpha,\beta)}{k(\alpha+\beta,\gamma)}}{\tcyc{k(\beta,\gamma)}{k(\alpha,\beta+\gamma)}} \quad {\alpha},
  {\beta}, \gamma \in \qlat. 
  \label{eq:gvsbrassoc}
\end{align}
Note that the abelian 3-cocycle does depend on the choice of section \(s\),
however, all choices of \(s\) yield the same trace and hence yield equivalent
braided monoidal structures. Similarly, different choices of the 2-cocycle
  \(\tcycsymb\) will yield equivalent associators.
Finally, every \(\ffv\in \qlat\) yields a \rgv{} category 
\({\Vect}_{\qlat}^{(F, \Omega, \ffv)}\) with dualising object \(\CC_{2\ffv}\) and with ribbon twist
\(\theta|_{\Mmod_\alpha} = Q(\alpha)\id_{\Mmod_\alpha}\),  \(\Mmod\in
{\Vect}_{\qlat}^{(F, \Omega, \ffv)}\), \(\alpha\in\qlat\), given by
\begin{equation}
  Q(\alpha) =  \ee^{\ii \pi\sqbrac*{
      \pair{\sct{\alpha - \ffv}}{\sct{\alpha - \ffv}} -
      \pair{\sct{- \ffv}}{\sct{-\ffv}}}} = \ee^{\ii \pi\sqbrac*{
      \pair{\sct{\alpha} - \sct{\ffv}}{\sct{\alpha} - \sct{\ffv}} -
      \pair{\sct{- \ffv}}{\sct{-\ffv}}}} =
  \ee^{\ii \pi
    \pair{\sct{\alpha}}{\sct{\alpha} +2 \sct{- \ffv}}}
  =
  \ee^{\ii \pi
    \pair{\sct{\alpha}}{\sct{\alpha} - 2 \sct{\ffv}}},
\end{equation}
where we have used that the lattice $\Lambda$ is even.
As with the quadratic form, the weak quadratic form \(Q\), which characterises the twist,
is independent of the choice of section due to \(\lat\) being even. We denote
the \rgv{} category constructed above by \(\Vect(\ldat)\).

\begin{ex}
  Recall the half rank lattice example at the end of \cref{sec:data}. In the
  notation and conventions introduced there, we have the \rgv{}
  structure defined by the abelian 3-cocycle, trace and twist
  \begin{align}
    \Omega((x_1,x_2+\ZZ), (y_1,y_2+\ZZ)) &= \ee^{\ii \pi \brac*{x_1 \tilde{y}_2
        + \tilde{x}_2y_1}},
    & F((x_1,x_2+\ZZ), (y_1,y_2+\ZZ),(z_1,z_2+\ZZ)) &=  (-1)^{x_1 \brac*{\widetilde{y_2 + z_2} -
          \tilde{y}_2 - \tilde{z}_2}},\nonumber\\
    q(x_1,x_2+\ZZ) &= \ee^{\ii 2\pi x_1\tilde{x}_2} ,&
    Q(x_1,x_2+\ZZ) &= \ee^{\ii 2\pi \brac*{x_1 - 1}\tilde{x}_2}.
  \end{align}
\end{ex}

\subsection{Categories of Heisenberg  and lattice \voa{} modules}
\label{sec:voas}

Let \(\ldat=(\hvec,\pair{-}{-},\lat,\ffv)\) be a set of \data{}.
Treating $\hvec$ as a real abelian Lie algebra, let
$\hlie=\hvec_{\CC}\otimes \CC\sqbrac*{t,t^{-1}}\oplus\CC\wun$ be the
affinisation of $\hvec_{\CC}$ (the complexification of $\hvec$ with the
bilinear form extended in the obvious way) with respect to the bilinear form $\pair{-}{-}$. This is called the Heisenberg Lie algebra (at level 1). 
For $\alpha \in \mathfrak{h}_\CC$ and \(n\in\ZZ\) denote
$\alpha_n=\alpha\otimes t^n$, then we have
\begin{equation}
  \comm{\alpha_n}{\beta_m} = n \pair{\alpha}{\beta} \delta_{n,-m} \wun, \quad \alpha_n, \beta_m \in \hlie,
\end{equation}
with \(\wun\) central and always taken to act as scalar multiplication by
\(1\) in modules.
We choose the triangular decomposition $\hlie = \hlie_{-} \oplus \hlie_{0}
\oplus \hlie_{+}$ with $\hlie_0 = \mathfrak{h}_\CC\otimes 1\oplus \CC\wun $ and $
\hlie_{\pm}=\spn{\alpha_n\st \alpha\in\hvec_\CC,\ \pm n>0}$. The highest weight modules with respect to
this decomposition (\(\hlie_{-}\) acting freely,
\(\hlie_{+}\) nilpotently and \(\hlie_{0}\) semisimply) are called Fock
spaces
\begin{equation}
  \Fock{\lambda} =
  \Ind{{\hlie}}{{{\hlie_{+}\oplus\hlie_0}}}{\CC\ket{\lambda}},\quad
  \lambda\in\hvec_\CC ,
  \label{eq:Fockdef}
\end{equation}
where
\begin{equation}
  \hlie_{+}\ket{\lambda}=0,\qquad
  \wun\ket{\lambda}=\ket{\lambda},\qquad
  \alpha_0\ket{\lambda}=\pair{\alpha}{\lambda}\ket{\lambda},\quad \alpha\in\hvec_\CC,
\end{equation}
and \(\hlie_{-}\) acts freely. In sequel, any reference to a Fock space \(\Fock{\lambda}\) will
assume the explicit choice of highest weight vector \(\ket{\lambda}\) given in
\eqref{eq:Fockdef}. This explicit choice of highest weight vector will be
required for giving explicit normalisations of intertwining operators.
For the lattice \voa{s} and modules to be considered in this section,
we shall mostly focus on real
weights, that is, \(\lambda\in \hvec\). For any coset \(\mu\in
\Lambda^\ast/\Lambda\) we define the lattice Fock space
\begin{equation}
  \LFock{\mu}=\bigoplus_{\nu \in \mu}\Fock{\nu}.
\end{equation}

\begin{prop}
  The Fock space \(\Fock{0}\) admits the structure of a \voa{} uniquely
  characterised by the choice of field map
  \begin{equation}
    Y(\alpha_{-1}\ket{0}, z) = \alpha(z) = \sum_{n=0}^{\infty} \alpha_{n}
    z^{-n-1},\quad \alpha\in\hvec_\CC,
    \label{eq:fldmap}
  \end{equation}
  and choice of conformal vector
  \begin{equation}
    \omega_{\gamma} = \frac{1}{2} \sum_{i} \alpha^i_{-1} \alpha^{i*}_{-1}\ket{0} +
    \gamma_{-2}\ket{0},\quad \gamma\in \hvec_\CC ,
  \end{equation}
  where \(\set{\alpha^{i}}_{i=1}^{\dim \hvec}\) and
  \(\set{\alpha^{j*}}_{j=1}^{\dim \hvec}\) are any dual choices of basis of \(\hvec_\CC\).
  We denote this \voa{} by \(\hvoa{\gamma}\). For any \(\alpha,\beta\in
  \hvec_\CC\), the \ope{s} of the corresponding fields \(\alpha(z), \beta(z)\) amongst themselves and
  with the conformal field \(T_\gamma(z)=Y(\omega_\gamma,z)\) are
  \begin{equation}
    \alpha(z)\beta(w)\sim \frac{\pair{\alpha}{\beta}}{(z-w)^2},\qquad
    T_\gamma(z)\alpha(w)\sim \frac{-2\pair{\gamma}{\alpha}}{(z-w)^3}+\frac{\alpha(w)}{(z-w)^2}+\frac{\partial \alpha(w)}{z-w},
  \end{equation}
  and the central charge determined by \(\omega_{\gamma}\) is
  \begin{equation}
    c_{\gamma}=\dim \hvec - 12\pair{\gamma}{\gamma}.
  \end{equation}
  Any choice of basis of \(\hvec_\CC\) is a set of strong generators of
  \(\hvoa{\gamma}\). For any \(\alpha\in\hvec_\CC\), the Fock space
  \(\Fock{\alpha}\) is a module over \(\hvoa{\gamma}\) with field map
  \(Y_\alpha\) characterised by the same formula \eqref{eq:fldmap} as the
  field map of \(\hvoa{\gamma}\) acting on itself.
\end{prop}

Let \(\CC\sqbrac*{\hvec_\CC}\) be the group algebra of \(\hvec_\CC\) seen as
an abelian group under addition and denote the basis element corresponding to
any group element \(\alpha\in\hvec_\CC\) by \(\ee^\alpha\). 
To each such basis vector we assign a linear map \(\ee^\alpha\), called a shift operator,
\begin{align}
  \ee^\alpha:\Fock{\gamma}&\to\Fock{\alpha+\gamma},\nonumber\\
  p\ket{\gamma}&\mapsto p\ket{\alpha+\gamma} ,
  \label{eq:shiftop}
\end{align} 
where \(p\in \UEA{\hlie_{-}}\). 
Further let
\begin{equation}
  E^{\pm} (\alpha,x) = \exp{\brac*{\mp \sum_{n=1}^{\infty} \frac{\alpha_{\pm
          n}}{n} x^{\mp n}}},\quad
  U(p,\alpha,x)=E^-(\alpha,x)Y(p,x)E^+(\alpha,x),\qquad
  \alpha\in \hvec_\CC,\ p\in \UEA{\hlie_-}.
  \label{eq:heisseries}
\end{equation}
Then we define linear maps $I_{\mu,\nu}: \Fock{\mu} \otimes \Fock{\nu} \ra
\Fock{\mu+\nu} \powser{z,z^{-1}}z^{\pair{\mu}{\nu}}$, for $\mu, \nu \in \hvec_\CC$
by 
\begin{align} 
  I_{\mu,\nu}\brac*{ p\ket{\mu}, z} q\ket{\nu} &= z^{\pair{\mu}{\nu}}\ee^{\mu}U(p,\alpha,z)q\ket{\nu}\nonumber\\
  &=z^{\pair{\mu}{\nu}}\ee^{\mu} E^- (\mu,z) Y\brac*{p,z} E^+ (\mu, z) q\ket{\nu},\qquad p,q\in \UEA{\hlie_-},
  \label{eq:utwvop}
\end{align}
where \(Y\brac*{p,z}\) is the series of Heisenberg algebra valued coefficients
obtain by expanding the field map \(Y(p\ket{0},z)\) in the \voa{} \(\hvoa{\gamma}\).
The linear maps \(I_{\mu,\nu}\) are generally known as (chiral)
vertex operators in theoretical physics
literature and are called untwisted vertex operators in \cite{DoLGVA93}.

\begin{prop}
  Let \(\mu,\nu,\rho,\in \hvec_\CC\), then
  \begin{equation}
    \dim \ispc{\Fock{\mu}}{\Fock{\nu}}{\Fock{\rho}}=
    \begin{cases}
      1& \rho=\mu+\nu\\
      0& \rho\neq\mu+\nu
    \end{cases}
  \end{equation}
  and \(I_{\mu,\nu}\) is an intertwining operator of type \(\binom{\Fock{\mu+\nu}}{\Fock{\mu},\Fock{\nu}}\).
\end{prop}

Lattice \voa{s} are constructed from Heisenberg \voa{s} by taking the
underlying vector space to be a sum over Fock spaces whose weights lie in a
lattice. The field maps for vectors lying in Fock spaces with non-zero weight
are then constructed from the untwisted intertwining operators \(I_{\mu,\nu}\) above.
As can be seen from the definitions of modules and intertwining operators, and
the unit isomorphism conditions \eqref{eq:unitmor}, the field maps encoding
the action of a \voa{} on its modules are a special case of an intertwining
operator with a canonical choice of normalisation. General intertwining
operators, however, have no obvious choice of normalisation. So in order to
extend a Heisenberg \voa{} to a lattice \voa{}, one needs to specify
normalisations. These normalisations need to be compatible with the vacuum,
skew-symmetry and associativity properties of \voa{s}, which implies that they
satisfy the defining properties of the 2-cocycles \(\tcycsymb\) in
\eqref{eq:voacocycle}. As previously noted all choices of 2-cocycle are
cohomologous and hence give rise to isomorphic lattice \voa{s} \cite[Chapter 5]{Fre88}.

\begin{prop} Let \(\rffv\) be a choice of representative of \(\ffv\).
  \begin{enumerate}
  \item The lattice Fock space
    \(\LFock{\Lambda}=\bigoplus_{\alpha\in \Lambda}\Fock{\alpha}\) admits the
    structure of a \voa{}, uniquely characterised by the choice of field map
    \begin{equation} 
      \left.Y\right|_{\Fock{\mu}\otimes\Fock{\nu}}=\tcyc{\mu}{\nu}
      I_{\mu,\nu},\quad \mu,\nu\in \lat,
      \label{eq:laction}
    \end{equation}
    (note that on \(\Fock{0}\) this specialises to the field map of the
    Heisenberg \voa{}) and choice of conformal vector
    \begin{equation}
      \omega_{\rffv} = \frac{1}{2} \sum_{i} \alpha^i_{-1} \alpha^{i*}_{-1}\ket{0} +
      \rffv_{-2}\ket{0},\quad \rffv\in \dlat,
    \end{equation}
    where \(\set{\alpha^{i}}_{i=1}^{\dim \hvec}\) and
    \(\set{\alpha^{j*}}_{j=1}^{\dim \hvec}\) are any dual choices of basis of
    \(\hvec_\CC\).  We denote this \voa{} by \(\lvoa{\rffv}{\Lambda}\). The
    central charge determined by \(\omega_{\rffv}\) is
    \begin{equation}
      c_{\rffv}=\dim \hvec - 12\pair{\rffv}{\rffv}.
    \end{equation}
  \item 
    The zero modes of \(Y(\alpha_{-1}\ket{0},z)\), \(\alpha\in \lat\) furnish
    \(\lvoa{\rffv}{\Lambda}\) with a \(\Lambda\)-grading. 
  \item For any \(\rho\in \Lambda^\ast/\Lambda\), the lattice Fock space
    \(\LFock{\rho}\) equipped with the field map
    \begin{equation}
      \left.Y_{\LFock{\rho}}\right|_{\Fock{\mu}\otimes\Fock{s(\rho)+\nu}}=\tcyc{\mu}{\nu}
      I_{\mu,s(\rho)+\nu},\quad \mu,\nu \in \lat,
      \label{eq:modaction}
    \end{equation}
    is a simple discreetly strongly \(\Lambda^\ast\)-graded generalised
    \(\lvoa{\rffv}{\Lambda}\) module. The conformal weight of the highest
    weight vector \(\ket{\mu}\) of a Fock space direct summand
    \(\Fock{\mu}\), \(\mu\in\rho\) is
    \begin{equation}
      h_\mu=\frac{1}{2}\pair{\mu}{\mu-2\rffv}.
      \label{eq:confwt}
    \end{equation}
  \item Every lattice Fock space \(\LFock{\rho},\ \rho\in
    \Lambda^\ast/\Lambda\), is graded \cfin{1} as a module over the
    Heisenberg \voa{} \(\hvoa{\beta}\). 
  \end{enumerate}
\end{prop}
\begin{proof}
  \begin{enumerate}
  \item The existence of the \va{} structure on
    \(\LFock{\lat}\) was shown in \cite[Theorem 3.6, Remark 3.7]{Li09}. Note
    that this \va{} structure is also unique in the sense that all
    choices of normalised 2-cocycles are cohomologous and yield isomorphic
    \va{s}. The restriction of \(\rffv\) to \(\dlat\) is equivalent to
    requiring that the grading of \(\LFock{\lat}\) be integral.
  \item This follows by construction.
  \item That the lattice Fock space \(\LFock{\rho}\) is a module follows from \cite[Theorem
    3.6]{Li09}. Each doubly homogeneous space of \(\LFock{\rho}\) is just an
    \(L_0\) eigenspace of one of the underlying Fock spaces \(\Fock{\mu},\
    \mu\in\rho\). Since these eigenspaces are all finite dimensional, the
    doubly homogeneous spaces are too. 
    Formula \eqref{eq:confwt} follows by direct computation and implies that
    all conformal weights are real
    and that the Fock spaces \(\Fock{\mu}\) are discretely strongly graded. 
    Hence the \(\LFock{\rho}\) are also discretely strongly graded.
  \item The \(\dlat\) homogeneous spaces of lattice Fock spaces are just the
    ordinary Fock spaces. These are all \cfin{1} over \(\hvoa{\rffv}\)
    because the \(C_1\) subspace has codimension 1.
  \end{enumerate}
\end{proof}
\begin{rmk}
  Note that the conformal structure of \(\lvoa{\rffv}{\Lambda}\)
  genuinely depends on the choice of vector \(\rffv\in \dlat\) rather than its
  coset \(\ffv=\rffv+\lat\in \qlat\). For example, shifting \(\rffv\) by some
  \(\alpha\in\lat\) will generally give a different central charge. It will
  also shift the conformal weight of any lattice module by some integer. However,
  the \rgv{} structure of the module category to be defined below
  will only depend on the coset \(\ffv\) (specifically, the dualising object
  and the twist depend on \(\ffv\), the associativity and braiding isomorphisms do
  not), rather than a choice of
  representative of this coset.
\end{rmk}

\begin{defn}
  For any set of \data{} \(\ldat=(\hvec,\pair{-}{-},\lat,\ffv)\) and a
  representative \(\rffv\in\ffv\),
  let $\lvmd{\ldat}$ be the full subcategory of generalised \(\Lambda^\ast\)-graded
  $\lvoa{\rffv}{\Lambda}$-modules whose objects are finitely generated, with
  $\hlie_{+}$ acting locally nilpotently and \(\hvec\) acting semisimply with real eigenvalues. 
\end{defn}

\begin{prop}
  The category $\lvmd{\ldat}$ is linear, abelian and semisimple.
  The lattice Fock spaces \(\LFock{\mu}\), \(\mu\in \Lambda^\ast/\Lambda\) form a complete set of
  mutually inequivalent representatives of isomorphism classes of simple objects.
  Further, the category $\lvmd{\ldat}$ satisfies all of the conditions of \cref{thm:simpsufconds},
  and therefore admits the braided monoidal structure of \cref{thm:tenstr} and the
  \rgv{} structure of \cref{thm:gvstr}. 
  \label{thm:gvboson}
\end{prop}
\begin{proof}
  The category $\lvmd{\ldat}$ is clearly linear and abelian by
  construction. We first show semisimplicity. Let \(M\in
  \lvmd{\ldat}\) be indecomposable. Since \(\hvec\) is required to
  act semisimply and real, \(M\) must be \(\hvec\) graded. Further, in order
  for \(M\) to be a $\lvoa{\rffv}{\Lambda}$-module all fields in
  $\lvoa{\rffv}{\Lambda}$ must have integral exponents when expanded on
  \(M\). Hence \(M\) is \(\Lambda^\ast\) graded and
  its \(\Lambda^\ast\) homogeneous spaces are modules over
  \(\hvoa{\rffv}\) by restriction. Since \(\lvoa{\rffv}{\Lambda}\) is
  \(\Lambda\)-graded, homogenous spaces of \(M\) corresponding to elements in
  \(\Lambda^\ast\) which are in different cosets of \(\Lambda\)
  cannot mix under the action of
  \(\lvoa{\rffv}{\Lambda}\). Since \(M\) is indecomposable the weights of non-zero \(\Lambda^\ast\)
  homogeneous spaces of \(M\) must all lie in the same \(\Lambda\) coset.
  Local nilpotence of $\hlie_{+}$ and semisimple action of \(\hvec\) then implies
  by an algebraic version of the Stone-von Neumann theorem \cite[Prop
  3.6]{LepRam82} that each \(\Lambda^\ast\) homogeneous space of \(M\) is a
  semisimple \(\hvoa{\rffv}\) module and a possibly infinite direct sum of
  Fock spaces. So assume there exists a direct
  sum decomposition \(M^{(\mu)}=A\oplus B\) of the homogeneous space of weight \(\mu\in \dlat\) into
  non-zero but not necessarily simple \(\hvoa{\rffv}\) modules \(A,B\).
  Then the \(\lvoa{\rffv}{\Lambda}\) submodules of \(M\) generated by \(A\)
  and \(B\) would intersect trivially and hence provide a direct sum
  decomposition of \(M\), contradicting indecomposability.
  Thus every non-trivial homogeneous space of \(M\) is isomorphic to a single Fock
  space of the same weight.
  The module \(M\) is therefore isomorphic to a lattice Fock
  space and hence simple. Further, lattice Fock spaces form a complete set of
  mutually inequivalent simple objects. Here we implicitly use the uniqueness
  of module structures on lattice Fock spaces which was shown in \cite[Proposition 4.2]{Li09}.

  The first three conditions of \cref{thm:simpsufconds} clearly hold and so we
  only need to verify the fourth. Consider two lattice Fock spaces
  \(\LFock{\mu},\LFock{\nu}\), \(\mu,\nu\in \qlat\) and let \(M\) be a
  finitely generated lower bounded submodule of
  \(\comp{\LFock{\mu}}{\LFock{\nu}}\). We need to verify that \(M\) is an
  object in \(\lvmd{\ldat}\). Since \(\lvmd{\ldat}\) is closed under
  contragredients, this is equivalent to \(M^\prime\) being in
  \(\lvmd{\ldat}\). By \cite[Part IV, Proposition 5.24]{HuaLog},
  \(M\subset \comp{\LFock{\mu}}{\LFock{\nu}}\)
  implies the existence of a surjective intertwining operator \(\mathcal{Y}\) of type
  \(\binom{M^\prime}{\LFock{\mu},\LFock{\nu}}\), we show that the image of any such
  intertwining operator must be an object in \(\lvmd{\ldat}\). By
  assumption \(M^\prime\) is finitely generated and hence we need only verify that
  \(\hvec\) acts semisimply and \(\hlie_{+}\) locally nilpotently. Assume
  \(m_\mu\in \LFock{\mu},m_\nu\in \LFock{\nu}\), the Jacobi identity for
  intertwining operators implies for any \(\alpha\in \hvec_\CC\) and \(n\ge1\)
  \begin{align*}
    \alpha_0\mathcal{Y}\brac*{m_\mu,x} m_\nu
    &= \mathcal{Y}\brac*{m_\mu,x} \alpha_0 m_\nu + \mathcal{Y}\brac*{\alpha_0 m_\mu, x}m\nu,\nonumber\\
    \brac{\alpha_{n} - x \alpha_{n-1}} \mathcal{Y}\brac*{m_\mu,x} m_\nu
    &= \mathcal{Y}\brac*{m_\mu,x} \brac{\alpha_{n}- x\alpha_{n-1}}m_\nu
      + \sum_{t= 0}^{n}\binom{t-n}{t} (-1)^t
      x^{n-t-1}\mathcal{Y}\brac*{\alpha_{t+1}m_\mu,x}m_\nu,
      \quad n\ge 1.
  \end{align*}
  The first equality shows that the semisimplicity of \(\alpha_0\) on
  \(m_\mu\) and \(m_\nu\) implies the  semisimplicity of \(\alpha_0\) on the
  image of \(\mathcal{Y}\). The second equality shows that the nilpotency of
  \(\hlie_{+}\) on \(m_\mu\) and \(m_\nu\) implies the local  nilpotency of
  \(\hlie_{+}\) on the image of \(\mathcal{Y}\). Thus all conditions of
  \cref{thm:simpsufconds} are satisfied, hence intertwining operators equip \(\lvmd{\ldat}\) with
  the braided monoidal structures of \cref{thm:tenstr}.

  Finally, the contragredient of a lattice Fock space is again a lattice Fock
  space (though generally of different weight). Hence \(\lvmd{\ldat}\) is
  closed under taking contragredients and thus admits a \rgv{} structure.
\end{proof}

Recall again that we are not assuming the lattice \(\lat\) to be non-zero and
so the above considerations capture the ordinary free boson without a lattice
by setting \(\lat=\{0\}\). Henceforth all references to $\lvmd{\ldat}$ are to be understood as
including the braided monoidal and \rgv{} structures provided in \cref{thm:gvboson}.

\begin{prop}
  Let \(\ldat\) be a set of \data{} and let \((\Omega,F)\) be the abelian
  3-cocycle constructed from \(\ldat\) by the formulae (\ref{eq:gvsbrassoc}).
  Since \(\lvmd{\ldat}\) is semisimple its structure
  isomorphisms are uniquely determined by their values on simple
  modules. Consider the lattice Fock spaces
  \(\LFock{\mu}, \LFock{\nu}, \LFock{\rho}\), \(\mu,\nu,\rho\in \qlat\).
  \begin{enumerate}
  \item For any two lattice Fock spaces \(\LFock{\mu}, \LFock{\nu}\) a choice
    of fusion product is given by
    \begin{equation}
      \LFock{\mu}\fuse \LFock{\nu}= \LFock{\mu+\nu},
    \end{equation}
    with corresponding universal intertwining operator
    \begin{equation}
      \left.\iop{\LFock{\mu},\LFock{\nu}}\right|_{\Fock{s(\mu)+\alpha_1}\otimes\Fock{s(\nu)+\alpha_2}}=(-1)^{\pair{s(\mu)}{\alpha_2}}\tcyc{\alpha_1}{\alpha_2}\tcyc{\alpha_1+\alpha_2}{k(\mu,\nu)}I_{s(\mu)+\alpha_1,s(\nu)+\alpha_2},\qquad \alpha_1,\alpha_2\in\lat.
      \label{eq:latintops}
    \end{equation}
    \label{itm:intformula}
  \item The braiding isomorphism \(c_{\mu,\nu}:\LFock{\mu}\fuse \LFock{\nu}\to
    \LFock{\nu}\fuse \LFock{\mu}\) is given by 
    \begin{equation}
      c_{\mu,\nu}=\ee^{\ii\pi\pair{s(\mu)}{s(\nu)}} \id_{\LFock{\mu+\nu}}=\Omega\brac*{\mu,\nu}\id_{\LFock{\mu+\nu}}.
      \label{eq:voabr}
    \end{equation}
    \label{itm:braidformula}
  \item The associativity isomorphism
    \(A_{\mu,\nu,\rho}: \LFock{\mu}\fuse
    \brac*{\LFock{\nu}\fuse\LFock{\rho}}\to\brac*{\LFock{\mu}\fuse
    \LFock{\nu}}\fuse\LFock{\rho}\) is given by
  \begin{equation}
    A_{\mu,\nu,\rho}=
    (-1)^{\pair{s(\mu)}{k(\nu,\rho)}}\frac{\tcyc{k(\mu,\nu)}{k(\mu+\nu,\rho)}}{\tcyc{k(\nu,\rho)}{k(\mu,\nu+\rho)}}\id_{\LFock{\mu+\nu+\rho}}=F\brac*{\mu,\nu,\rho}\id_{\LFock{\mu+\nu+\rho}}.
    \label{eq:voaassoc}
  \end{equation}
 \item The contragredient of a lattice Fock space is
  \begin{equation}
    \LFock{\rho}^\prime=\LFock{2\ffv-\rho},\qquad \rho\in \qlat,
  \end{equation}
  and hence the dualising object is \(\LFock{2\ffv+\lat}\).
\item The twist isomorphism is given by
  \begin{equation}
    \theta_{\LFock{\rho}}=\ee^{\pi\ii\pair{s(\rho)}{s(\rho)-2\rffv}}\id_{\LFock{\rho}},
    \qquad\rho\in\qlat.
  \end{equation}
  Note that \(\lat\) being even guarantees that the above twist formula is
  independent of the choice of section \(s\).
  \label{itm:twistformula}
\end{enumerate}
\label{thm:voabosstrmaps}
\end{prop}

\begin{proof}
  Parts \ref{itm:braidformula} -- \ref{itm:twistformula} follow by simple
  computations from the explicit formulae for intertwining operators in Part \ref{itm:intformula}.
  \begin{enumerate}
  \item The lattice intertwining operator formulae \eqref{eq:latintops} were
    proved in \cite{DoLGVA93} in the context of full rank even lattices,
    however, the arguments showing that these formulae satisfy the
    intertwining operator axioms, such as the Jacobi identity, do not depend
    on the lattice being full rank. See also, \cite{Tui12} for detailed
    descriptions on how to compute with Heisenberg intertwining operators.
  \item Since the lattice Fock spaces are simple modules, the braiding
    isomorphism is determined by comparing the leading terms of
    \(\ifld{\LFock{\mu},\LFock{\nu}}{\ket{s(\mu)+\alpha_1}}{z}\ket{s(\nu)+\alpha_2}\) and
    \(\ee^{zL_{-1}}\ifld{\LFock{\mu},\LFock{\nu}}{\ket{s(\nu)+\alpha_2}}{\ee^{\ii\pi}z}\ket{s(\mu)+\alpha_1}\),
    where \(\mu,\nu\in\qlat\) and \(\alpha_1,\alpha_2\in \lat\). These are
    \begin{align}
      \ee^{zL_{-1}}\ifld{ \LFock{\nu},\LFock{\mu}}{\ket{s(\nu)+\alpha_2}}{\ee^{\ii\pi}z}\ket{s(\mu)+\alpha_1}
      &= (\ee^{\ii\pi}z)^{\pair{s(\mu)+\alpha_1}{s(\nu)+\alpha_2}}\tcyc{\alpha_2}{\alpha_1}(-1)^{\pair{s(\nu)}{\alpha_1}}\brac*{\ket{s(\mu)+\alpha_1+s(\nu)+\alpha_2}+\mathcal{O}(z)},\nonumber\\
      \ifld{\LFock{\mu},\LFock{\nu}}{\ket{s(\mu)+\alpha_1}}{z}\ket{s(\nu)+\alpha_2}
      &= z^{\pair{s(\mu)+\alpha_1}{s(\nu)+\alpha_2}}\tcyc{\alpha_1}{\alpha_2}(-1)^{\pair{s(\mu)}{\alpha_2}}\brac*{\ket{s(\mu)+\alpha_1+s(\nu)+\alpha_2}+\mathcal{O}(z)}.
    \end{align}
    Comparing the leading terms we obtain
    \begin{equation}
      \frac{\tcyc{\alpha_2}{\alpha_1}}{\tcyc{\alpha_1}{\alpha_2}}\ee^{\ii \pi\pair{s(\mu)+\alpha_1}{s(\nu)+\alpha_2}}(-1)^{\pair{s(\nu)}{\alpha_1}}(-1)^{-\pair{s(\mu)}{\alpha_2}}=\ee^{\ii\pi\pair{s(\mu)}{s(\nu)}},
    \end{equation}
    and hence \(c_{\mu,\nu}=\ee^{\ii\pi\pair{s(\mu)}{s(\nu)}} \id_{\LFock{\mu+\nu}}\).
  \item As with the braiding isomorphisms, since the lattice Fock spaces are
    simple modules, thus the associativity
    isomorphisms are determined by comparing the leading
    terms. Let \(\mu,\nu,\rho\in\qlat\) and \(\alpha_1,\alpha_2,\alpha_3\in
    \lat\) and consider
    \begin{align}
      &\ifld{\LFock{\mu},\LFock{\nu+\rho}}{\ket{s(\mu)+\alpha_1}}{x_1}\ifld{\LFock{\nu},\LFock{\rho}}{\ket{s(\nu)+\alpha_2}}{x_2}\ket{s(\rho)+\alpha_3}\nonumber\\
      &\quad=
        (-1)^{\pair{s(\mu)}{k(\nu,\rho)+\alpha_2+\alpha_3}}\tcyc{\alpha_1}{
        \alpha_2+\alpha_3+k(\nu,\rho)}\tcyc{\alpha_1+\alpha_2+\alpha_3+k(\nu,\rho)}{k(\mu,\nu+\rho)}\nonumber\\
      &\qquad
        (-1)^{\pair{s(\nu)}{\alpha_3}}\tcyc{\alpha_2}{\alpha_3}\tcyc{\alpha_2+\alpha_3}{k(\nu,\rho)}\nonumber\\
      &\qquad(x_1-x_2)^{\pair{s(\mu)+\alpha_1}{s(\nu)+\alpha_2}}x_1^{\pair{s(\mu)+\alpha_1}{s(\rho)+\alpha_3}}x_2^{\pair{s(\nu)+\alpha_2}{s(\rho)+\alpha_3}}
        \brac[\big]{\ket{s(\mu)+s(\nu)+s(\rho)+\alpha_1+\alpha_2+\alpha_3}+\mathcal{O}(z)}\nonumber\\
      &\ifld{\LFock{\mu+\nu},\LFock{\rho}}{\ifld{\LFock{\mu},\LFock{\nu}}{\ket{s(\mu)+\alpha_1}}{x_1-x_2}\ket{s(\nu)+\alpha_2}}{x_2}\ket{s(\rho)+\alpha_3} \nonumber\\
      &\quad=
        (-1)^{\pair{s(\mu)}{\alpha_2}}\tcyc{\alpha_1}{\alpha_2}\tcyc{\alpha_1+\alpha_2}{k(\mu,\nu)}\nonumber\\
      &\qquad  (-1)^{\pair{s(\mu+\nu)}{\alpha_3}}\tcyc{\alpha_1+\alpha_2+k(\mu,\nu)}{\alpha_3}\tcyc{\alpha_1+\alpha_2+\alpha_3+k(\mu,\nu)}{k(\mu+\nu,\rho)}\nonumber\\
      &\qquad
        (x_1-x_2)^{\pair{s(\mu)}{s(\nu)}}x_2^{\pair{s(\mu)+s(\nu)}{s(\rho)}}\brac*{\frac{x_1}{x_2}}^{\pair{s(\mu)}{s(\rho)}}
        \brac[\big]{\ket{s(\mu)+s(\nu)+s(\rho)+\alpha_1+\alpha_2+\alpha_3}+\mathcal{O}(z)}.
    \end{align}
    The limit of the ratio of the \(x_1\) dependent factors is
  \begin{equation}
    \lim_{x_2\to 1}\lim_{x_1-x_2\to 1}
    \frac{(x_1-x_2)^{\pair{s(\mu)}{s(\nu)}}x_2^{\pair{s(\mu)+s(\nu)}{s(\rho)}}\brac*{\frac{x_1}{x_2}}^{\pair{s(\mu)}{s(\rho)}}}{(x_1-x_2)^{\pair{s(\mu)}{s(\nu)}}x_1^{\pair{s(\mu)}{s(\rho)}}x_2^{\pair{s(\nu)}{s(\rho)}}}
    =1
  \end{equation}
  and the associativity isomorphism is scalar multiplication by
  \begin{align}
    &\frac{(-1)^{\pair{s(\mu)}{\alpha_2}}\tcyc{\alpha_1}{\alpha_2}\tcyc{\alpha_1+\alpha_2}{k(\mu,\nu)}(-1)^{\pair{s(\mu+\nu)}{\alpha_3}}\tcyc{\alpha_1+\alpha_2+k(\mu,\nu)}{\alpha_3}\tcyc{\alpha_1+\alpha_2+\alpha_3+k(\mu,\nu)}{k(\mu+\nu,\rho)}}
    {(-1)^{\pair{s(\mu)}{k(\nu,\rho)+\alpha_2+\alpha_3}}\tcyc{\alpha_1}{
      \alpha_2+\alpha_3+k(\nu,\rho)}\tcyc{\alpha_1+\alpha_2+\alpha_3+k(\nu,\rho)}{k(\mu,\nu+\rho)}(-1)^{\pair{s(\nu)}{\alpha_3}}\tcyc{\alpha_2}{\alpha_3}\tcyc{\alpha_2+\alpha_3}{k(\nu,\rho)}}\nonumber\\
    &= (-1)^{\pair{s(\mu)}{k(\nu,\rho)}}\frac{\tcyc{k(\mu,\nu)}{k(\mu+\nu,\rho)}}{\tcyc{k(\nu,\rho)}{k(\mu,\nu+\rho)}}
  \end{align}
 \item The Heisenberg weight of \(\LFock{\rho}^\prime\) is determined by computing the
  opposed field map of \(\alpha_{-1}\ket{0}\), \(\alpha\in \hvec_\CC\). This
  is given by
  \begin{align}
    Y(\alpha_{-1}\ket{0},z)^{\opp}&=Y(e^{zL_1}(-z^{-2})^{L_0}\alpha_{-1}\ket{0},z^{-1})=
    -z^{-2}Y(\alpha_{-1}\ket{0},z^{-1})+z^{-1}2\pair{\rffv}{\alpha}Y(\ket{0},z^{-1})\nonumber\\
    &= -z^{-2}Y(\alpha_{-1}\ket{0},z^{-1})+z^{-1}2\pair{\rffv}{\alpha}\id,
  \end{align}
  for any \(\alpha\in\hvec_\CC\). This implies that the Heisenberg weight of
  \(\LFock{\rho}^\prime\) is \(2\rffv-\rho\).
\item The formula for the twist isomorphism follows immediately from the
  conformal weight of Fock space highest weight vectors.
  \end{enumerate}
\end{proof}

We prepare some notation in order to use \cref{thm:btcvoa} to show that
  \(\Vect\brac*{\ldat}\) and \(\lvmd{\ldat}\) are equivalent as \rgv{} categories.
To any object in \(\Vect\brac*{\ldat}\) we can associate an
object in \(\lvmd{\ldat}\) by the following induction construction. Let
\(M=\bigoplus_{\alpha\in \qlat}M_\alpha\) be a decomposition into homogeneous
spaces and consider the
vector space \(M\otimes\CC[\lat]\) and endow it with the structure of an
\(\hlie_{\ge}=\hlie_0\oplus\hlie_+\) module by defining
\begin{equation}
  \hlie_+\cdot M\otimes\CC[\lat]=0,\qquad
  \alpha\cdot m_\rho\otimes \beta=\pair{\alpha}{s(\rho)+\beta}\id,\quad
  \alpha\in \hvec_\CC,\ \rho\in \qlat,\ \beta\in \lat.
\end{equation}
Further induce \(M\otimes\CC[\lat]\) to a module over \(\hlie\) by defining
\begin{equation}
  \iFock{M}=\Ind{\hlie}{\hlie_\ge}{M\otimes\CC[\lat]}.
  \label{eq:indFock}
\end{equation}
Next, define the action of the shift operators \(\ee^\gamma,\
\gamma\in\lat\) on \(M\otimes\CC[\lat]\) to be
\begin{equation}
  \ee^\gamma m\otimes\ee^\delta=m\otimes\ee^{\gamma+\delta},\quad m\in M,\
  \gamma,\delta\in \lat
\end{equation}
and extend to all of \(\iFock{M}\) to obtain a well defined action of the
obvious analogue of
untwisted vertex operators \eqref{eq:utwvop} (with the first of the two indices
parametrising weights in \(\lat\)) and hence also the field map
\eqref{eq:modaction} by defining
\begin{align}
  Y_{\iFock{M}}(u\ket{\alpha_1},z)v\cdot m\otimes
  \ee^{\alpha_2}=\tcyc{\alpha_1}{\alpha_2} (-1)^{\pair{s(\lat)}{\alpha_2}}z^{\pair{\alpha_1}{s(\mu)+\alpha_2}}\ee^{\alpha_1}U(\alpha_1,u,z)\,
  v\cdot
  m\otimes \ee^{\alpha_2},
\end{align}
for \(\alpha_1,\alpha_2\in \lat\), \(\mu\in\qlat\), \(m\in M_\mu\), \(u,v\in \UEA{\hlie_{-}}\) and 
where \(U(\alpha_1,u,z)\) is the Heisenberg algebra valued series \eqref{eq:heisseries}.
Thus \(\iFock{M}\) has the structure of a
\(\lvoa{\rffv}{\lat}\) module, with decomposition into lattice Fock spaces
given by
\begin{equation}
  \iFock{M} \cong \bigoplus_{\rho\in \qlat} \dim\brac*{M_\rho}\LFock{\rho}.
  \label{eq:indFockdecomp}
\end{equation}
To define intertwining operators for the modules constructed above, we shall
need the following auxiliary linear maps which for any \(M,N,P\in
{\Vect}\brac*{\ldat}\) and \(f\in \Homgrp{}{M\otimes N}{P}\) are defined to be
\begin{align}
  f_m: \iFock{N}&\to \iFock{P}, \qquad\qquad  m\in M,\ n\in N,\
                         u\in\UEA{\hlie_{-}},\ \gamma\in\lat.\nonumber\\
  u\cdot n\otimes \ee^{\gamma}&\mapsto u\cdot f(m\otimes n)\otimes \ee^\gamma
  \label{eq:tensormaps}
\end{align}

\begin{thm}
  Let $\ldat$ be a set of \data{}, $\Vect\brac*{\ldat}$ be the associated
  \rgv{} category from the previous section and  $\lvmd{\ldat}$ the module category of 
  the lattice \voa{} described above. Further,
  let \(G: {\Vect}\brac*{\ldat} \to\lvmd{\ldat}\) be the functor
  which assigns to any \(M\in {\Vect}\brac*{\ldat}\) the object
  \(G(M)=\iFock{M}\) from \eqref{eq:indFock}
  with the obvious extension to morphisms. Consider the
  following maps.
  \begin{itemize}
  \item Let $\tstr_0 : \lvoa{\rffv}{\lat} \mapsto G(\CC_{0})$ be the
    module map uniquely characterised by
    \(\tstr_0(\ket{0})=1_0\otimes \ee^0\), where \(1_0\in\CC_0\).
  \item For \(M,N,P\in {\Vect}\brac*{\ldat}\), \(f\in \Homgrp{}{M\otimes N}{P}\), 
    \(\mu,\nu\in\qlat\), \(m\in M_\mu\), \(n\in N_\nu\), \(\alpha_1,\alpha_2\in\lat\) and
    \(u,v\in \UEA{\hlie_{-}}\) define \(\btf\) by
    \begin{multline}
      \btfi{f}{u\cdot m\otimes \ee^{\alpha_1},z}v\cdot
      n\otimes \ee^{\alpha_2} \\
      =
      (-1)^{\pair{s(\mu)}{\alpha_2}}\tcyc{\alpha_1}{\alpha_2}\tcyc{\alpha_1+\alpha_2}{k(\mu,\nu)}
      z^{\pair{s(\mu)+\alpha_1}{s(\nu)+\alpha_2}}f_m\ee^{\alpha_1}
      U(u, s(\mu) + \alpha_1,z)
      v\cdot
      n\otimes \ee^{\alpha_2}.
      \label{eq:bosonTstr}
    \end{multline}
  \end{itemize}
  Then \(\tstr_0\) and \(\btf\) 
  satisfy the conditions of \cref{thm:btcvoa} and hence
  endow \(G\) with the structure of a braided monoidal functor. The functor \(G\) with
  this choice of monoidal structure is an equivalence of \rgv{} categories.
   In particular, for $\ffv = 0$, the functor \(G\) is a ribbon equivalence.
   \label{thm:voaequiv}
\end{thm}
  The equivalence of \(\lvmd{\ldat}\) and ${\Vect}\brac*{\ldat}$
  as braided tensor categories is well known \cite{DoLGVA93} in the special
  case of positive definite even full rank
  lattices. 
  Here we use the opportunity to illustrate the application of \cref{thm:btcvoa} and
  to show the equivalence of the \rgv{} structures as well.
  \begin{proof}
    We prove the theorem by showing that \(\tstr_0\) and the family of linear maps \(\btf\) of
    (\ref{eq:bosonTstr}) satisfy the conditions of
    \cref{thm:btcvoa}.\ref{itm:voptransf} and
    \cref{thm:tstrcor}.\ref{itm:intertweq}.
    We first show the functoriality of \(\btf\).  For any $\Mmod,\Mmod^\prime,
    \Nmod,\Nmod^\prime, \Pmod,\Pmod^\prime\in {\Vect}\brac*{\ldat}$,
    \(\mu,\nu\in\qlat\), \(m\in \Mmod^\prime_\mu\), \(n\in \Nmod^\prime_\nu\), \(\alpha_1,\alpha_2\in\lat\) and
    \(u,v\in \UEA{\hlie_{-}}\)
    consider
    \begin{align}
      \btfi{k\circ f\circ\brac*{g\otimes h}}{u\cdot m\otimes
        \ee^{\alpha_1};z}v\cdot n\otimes \ee^{\alpha_2}&=
      z^{\pair{s(\mu)+\alpha_1}{s(\nu)+\alpha_2}}\brac*{k\circ f\circ\brac*{g\otimes h}}_m\ee^{\alpha_1}
      U(s(\mu)+\alpha_1,u,z))\, 
      v\cdot
      n\otimes \ee^{\alpha_2}\nonumber\\
      &=
      G(k)z^{\pair{s(\mu)+\alpha_1}{s(\nu)+\alpha_2}}f_{g(m)}\ee^{\alpha_1}
      U(s(\mu)+\alpha_1,u,z)  v\cdot
        h(n)\otimes \ee^{\alpha_2}\nonumber\\
      &=
        G(k)\circ \btfi{f}{G(g)u\cdot m\otimes
        \ee^{\alpha_1}, z}G(h)v\cdot n\otimes \ee^{\alpha_2},
    \end{align}
    where the second and third equalities follow from the definition of the \(f_m\)
    notation in \eqref{eq:tensormaps}. Thus \(\btf\) is functorial.

    Next we show the unitality of \(\btf\). For any \(\Nmod\in {\Vect}\brac*{\ldat}\), \(\nu\in\qlat\),  \(n\in \Nmod_\nu\), \(\alpha_1,\alpha_2\in\lat\) and
    \(u,v\in \UEA{\hlie_{-}}\)  consider
    \begin{align}
      \btfi{l_\Nmod}{\tstr_0(u\ket{\alpha_1}),z}v\cdot n\otimes
      \ee^{\alpha_2}&=
      \btfi{l_\Nmod}{u\cdot 1_0\otimes\ee^{\alpha_1},z}v\cdot n\otimes
                      \ee^{\alpha_2}\nonumber\\
      &=
         z^{\pair{s(\mu)+\alpha_1}{s(\nu)+\alpha_2}}\brac*{l_\Nmod}_{1_0}\ee^{\alpha_1}
      U(s(\mu)+\alpha_1,u,z)\, v\cdot
        n\otimes \ee^{\alpha_2}\nonumber\\
      &= 
        z^{\pair{s(\mu)+\alpha_1}{s(\nu)+\alpha_2}}\ee^{\alpha_1}
        U(s(\mu)+\alpha_1,u,z)\, v\cdot
        n\otimes \ee^{\alpha_2}\nonumber\\
      &=
        Y_{G(\Nmod)}\brac*{u\ket{\alpha_1},z}v\cdot
        n\otimes \ee^{\alpha_2},
    \end{align}
    where in the third identity we have used that \(l_\Nmod(1_0\otimes n)=n\).
    Thus \(\btf\) is unital.

    Next we show the skew symmetry of \(\btf\). For any \(\Mmod,\Nmod\in
    {\Vect}\brac*{\ldat}\), \(\mu,\nu\in\qlat\), \(m\in \Mmod_\mu\), \(n\in \Nmod_\nu\), \(\alpha_1,\alpha_2\in\lat\) and
    \(u,v\in \UEA{\hlie_{-}}\) consider
    \begin{align}
      \btfi{c_{\Nmod,\Mmod}}{v\cdot n\otimes \ee^{\alpha_2},z}u\cdot
      m\otimes\ee^{\alpha_1}
      &=
      z^{\pair{s(\nu)+\alpha_2}{s(\mu)+\alpha_1}}\brac*{c_{\Nmod,\Mmod}}_{n}\ee^{\alpha_2}
      U(s(\nu)+\alpha_2,u,z)\,  v\cdot
        n\otimes \ee^{\alpha_1}\nonumber\\
      &=\ee^{\ii\pi\pair{s(\nu)}{s(\mu)}}z^{\pair{s(\nu)+\alpha_2}{s(\mu)+\alpha_1}}\brac*{P_{\Nmod\otimes
        \Mmod}}_{n}\ee^{\alpha_2}
        U(s(\nu)+\alpha_2,u,z)\, v\cdot
        n\otimes \ee^{\alpha_1}\nonumber\\
      &=\ee^{\ii\pi\pair{s(\nu)}{s(\mu)}}\btfi{P_{\Nmod\otimes\Mmod}}{v\cdot n\otimes \ee^{\alpha_2},z}u\cdot
        m\otimes\ee^{\alpha_1}\nonumber\\
      &=\ee^{z L_{-1}}\btfi{\id_{\Mmod\otimes\Nmod}}{u\cdot
        m\otimes\ee^{\alpha_1},\ee^{\ii\pi}z}v\cdot n\otimes \ee^{\alpha_2},
    \end{align}
    where \(P_{\Nmod\otimes\Mmod}:\Nmod\otimes \Mmod\to \Mmod\otimes\Nmod\) is
    the standard tensor flip of (graded) vector spaces and where in the fourth
    identity we have used the well known behaviour of untwisted vertex
    operators \eqref{eq:utwvop} with respect to \(L_{-1}\). Thus \(\btf\) is skew symmetric.

    Next we show \(\btf\) is associative.
    For any $\Mmod, \ \Nmod, \ \Pmod \in {\Vect}\brac*{\ldat}$,
    $m \in \Mmod$, $n \in \Nmod$, $p \in \Pmod$,
    \(\mu,\nu,\rho\in\qlat\), \(\alpha_1,\alpha_2,\alpha_3\in\lat\),
    \(u,v,w\in \UEA{\hlie_{-}}\)
    and \(x_1,x_2\in\CC\) such
      that $|x_1|>|x_2|>0$ and
      $|x_2| > |x_1 - x_2|>0$, consider
      \begin{align}
        &\btfi{\alpha_{\Mmod, \Nmod, \Pmod}}{u\cdot m\otimes
        \ee^{\alpha_1},x_1} \btfi{\id_{\Nmod \otimes \Pmod}}{v\cdot n\otimes
          \ee^{\alpha_2},x_2} w\cdot p\otimes \ee^{\alpha_3}\nonumber\\
        &\qquad=
          x_1^{\pair{s(\mu)+\alpha_1}{s(\nu+\rho)-k(\nu,\rho)+\alpha_2+\alpha_3}}x_2^{\pair{s(\nu)+\alpha_2}{s(\rho)+\alpha_3}}
          \brac*{\alpha_{\Mmod,\Nmod,\Pmod}}_m\ee^{\alpha_1}
          U(s(\nu)+\alpha_1,u,x_1)
          \brac*{\id_{\Nmod\otimes\Pmod}}_n
          \ee^{\alpha_2}
          U(s(\nu)+\alpha_2,v,x_2)\,
          w\cdot p\otimes \ee^{\alpha_3}\nonumber\\
        &\qquad=
          (-1)^{\pair{s(\mu)}{k(\nu,\rho)}}\frac{\tcyc{k(\mu,\nu)}{k(\mu+\nu,\rho)}}{\tcyc{k(\nu,\rho)}{k(\mu,\nu+\rho)}}
          x_1^{\pair{s(\mu)+\alpha_1}{s(\nu+\rho)-k(\nu,\rho)+\alpha_2+\alpha_3}}x_2^{\pair{s(\nu)+\alpha_2}{s(\rho)+\alpha_3}}
          \brac*{\id_{\Mmod\otimes\brac*{\Nmod\otimes\Pmod}}}_m\ee^{\alpha_1}
          U(s(\mu)+\alpha_1,u,x_1)
          \nonumber\\
        &\qquad\qquad\cdot
          \brac*{\id_{\Nmod\otimes\Pmod}}_n
          \ee^{\alpha_2}
          U(s(\nu)+\alpha_2,v,x_2)\,w\cdot p\otimes \ee^{\alpha_3}
          \nonumber\\
        &\qquad=\btfi{\id_{(\Mmod \otimes \Nmod) \otimes
        \Pmod}}{\btfi{\id_{\Mmod\otimes \Nmod}}{u\cdot m\otimes
        \ee^{\alpha_1}, x_1-x_2} v\cdot n\otimes\ee^{\alpha_2}, x_2} w\cdot
        p\otimes \ee^{\alpha_3},
      \end{align}
      where in the third equality we have used the well known behaviour of untwisted vertex
      operators \eqref{eq:utwvop}, see for example \cite[Section 12]{DoLGVA93}
      or \cite{Tui12}.
      Thus \(\btf\) is associative.

     The intertwining operators \(\btfi{\id_{\Mmod\otimes \Nmod}}{z}\)
     are surjective by construction for any \(\Mmod,\Nmod\in {\Vect}\brac*{\ldat}\). 
      Hence, by \cref{thm:tstrcor}, the functor \(G\) with the monoidal structure constructed from
      \(\btf\) is a braided monoidal equivalence.

      The equivalence of the \rgv{} structures then follows from noting that
      the dualising objects are isomorphic, that is,
      \begin{align}
        G(\CC_{2\ffv})\cong \dim (\CC_{2\ffv})\;\LFock{2\ffv}=\LFock{2\ffv},
      \end{align}
      and that the twists are equivalent, that is, for any \(\mu\in\qlat\)
      \begin{align}
        G(\theta_{\CC_\mu})=\ee^{\pi\ii\pair{s(\mu)}{s(\mu)-2s(\ffv)}}\id_{G(\CC_\mu)}=
        \ee^{\pi\ii\pair{s(\mu)}{s(\mu)-2\rffv}}\id_{G(\CC_\mu)}=\ee^{2\pi\ii L_0}\vert_{G(\CC_\mu)}=\theta_{G(\CC_\mu)},
      \end{align}
      where in the second equality we have used that \(s(\ffv)\) and \(\rffv\)
      differ at most by an element in \(\lat\).
    \end{proof}

\begin{ex}
  Recall the half rank lattice example at the end of \cref{sec:data}. In the
  notation and conventions introduced there, we choose \(\rffv=(1,0)\) as a
  representative of \(\ffv\). Further let \(\alpha=(1,0)\in\RR^2\) and
  \(\beta=(0,1)\in \RR^2\), then \voa{} structure on \(\Fock{0}\) is strongly
  generated by the fields corresponding to \(\alpha,\beta\), whose defining
  \ope{s} are
  \begin{align}
    \alpha(z)\alpha(w)\sim 0\sim \beta(z)\beta(w),\qquad \alpha(z)\beta(w)\sim \frac{1}{(z-w)^2}.
  \end{align}
  The choice of element \(\rffv\) defines the conformal vector and central charge
  \begin{align}
    \omega_{\rffv}=\alpha_{-1}\beta_{-1}\ket{0}+\rffv_{-2}\ket{0},\qquad c_{\rffv}=2.
  \end{align}
  Further, \(\beta\) generates the lattice \(\lat\) and the Fock spaces with
  weights in \(\lat\) have generating highest weight vectors of conformal
  weights
  \begin{equation}
    h_{n\beta}=\frac{1}{2}\pair{(0,n)}{(-2,n)}=-n.
  \end{equation}
\end{ex}

\subsection{Categories of Hopf algebra modules}
\label{sec:hopfalgcat}

For any set of \data{} \(\ldat=(\hvec,\pair{-}{-},\lat,\ffv)\),
let $\UEA{\plat}$ denote the universal enveloping algebra (or
symmetric algebra) of the complexification \(\plat_\CC\) of the vector space
$\plat$ seen as an abelian Lie algebra
and $\CC[\qlat^\perp]$ the group algebra of the
abelian group $\dlat / \plat$. These associative algebras both admit well
known Hopf algebra structures by defining the elements of \(\plat_\CC\) to be
primitive and those of \(\dlat/\plat\) to be group like, that is
\begin{align}
  \Delta(\mu)&=\mu\otimes 1+1\otimes\mu,&\epsilon(\mu)&=0,& s(\mu)&=-\mu,
  &\mu\in \plat_\CC,\nonumber\\
  \Delta(K_\nu)&=K_\nu\otimes K_\nu,&\epsilon(K_\nu)&=1,& s(K_\nu)&=K_{-\nu} = K_{\nu}^{-1},
  &\nu\in \dlat/\plat,
\end{align}
where \(K_\nu\) is the basis element of \(\CC[\dlat/\plat]\) corresponding to
\(\nu\in \dlat/\plat\). We call
\begin{equation}
  H_\lat = \UEA{\lat^{\perp}} \otimes \CC[\qlat^\perp]
\end{equation}
the \emph{lattice Hopf algebra} of \(\lat\), where the Hopf algebra structures
are those inherited from the two tensor factors.

Every object \(\Mmod\) in \(\Vect\brac*{\ldat}\) can be given the structure of an
\(H_\lat\) module by defining the representation \(\rho_{\Mmod}:H_\lat
\to \End{\Mmod}\) on homogeneous spaces by
\begin{equation} 
  \rho_\Mmod (\mu) \vert_{\Mmod_{\alpha}} = \pair{\mu}{s(\alpha)} \id_{\Mmod_{\alpha}}, \quad 
  \rho_\Mmod (K_\nu) \vert_{\Mmod_{\alpha}} = \ee^{2 \pi i
    \pair{\nu}{s(\alpha)}} \id_{\Mmod_{\alpha}},\qquad
  \alpha\in \qlat,\ \mu\in \plat_\CC,\ \nu\in \dlat/\plat.
  \label{eq:haction}
\end{equation}
Note that the above formulae do not depend on the choice of section \(s\).
We can therefore interpret \(\Vect\brac*{\ldat}\) as a category of representations
of the group \(\dlat\). Further, for
\( \mu\in \plat_\CC\cap \lat\), \(\rho_\Mmod (\mu)
\vert_{\Mmod_{\alpha}} =0\) and for \(\nu\in \dlat/\plat\),
\(\nu\cap\lambda\neq 0\), \(\rho_\Mmod (K_\nu) \vert_{\Mmod_{\alpha}}
=\id_{\Mmod_{\alpha}}\). Hence the objects of \(\Vect\brac*{\ldat}\) also can be
interpreted as representations of the quotient group \(\qlat\).
Since \(H_\lat\) is a Hopf algebra, there is of course a natural representation
on tensor products of objects \(\Mmod,\Nmod\in\Vect\brac*{\ldat}\) given by
\(\rho_{\Mmod\otimes
  \Nmod}=\brac*{\rho_\Mmod\otimes\rho_{\Nmod}}\circ\Delta\). 
Now that we have recast \(\Vect\brac*{\ldat}\), as an abelian category, as a category of
modules over \(H_\lat\), it is interesting to see if we can capture the
braided monoidal, \gv{} and ribbon structures of \({\Vect}\brac*{\ldat}\)
in Hopf algebraic terms by specifying an \(R\)-matrix, coassociator and
ribbon element. To do so, we recall the decomposition \(\dlat=\plat\oplus \Gamma\) of \(\dlat\) in \cref{thm:hbasis}.\ref{itm:dlatdecomp}.
We define formal operators in terms of their action
on the objects of \(\Vect\brac*{\ldat}\) (though they could also be thought of as lying in suitable completions of tensor powers of
\(H_\lat\)). 
Let \(\set{\mu_i}_{i=1}^{\dim\plat}\) be an $\RR$-basis of $\plat$ and let
$\set{\nu_j}_{j=1}^{\rk \lat}$ be a $\ZZ$-basis of $\Gamma$. Since the real
span of \(\dlat\) is \(\hvec\), \(\set{\mu_i, \nu_j}\) is an \(\RR\) basis of \(\hvec\).
Hence there exists a dual basis \(\set{\mu^i,\nu^j}\). Let \(\log_s K_\nu\),
\(\nu\in \dlat/\plat\) be the formal operator, depending on the section \(s\), defined on the homogeneous
spaces of an object \(\Mmod\in \Vect\brac*{\ldat}\) to act as
\begin{equation}
  \log_s(K_\nu)|_{\Mmod_\alpha} =
  \pair{\nu}{s(\alpha)}\id_{\Mmod_\alpha},\qquad \alpha\in \qlat.
\end{equation}
Further, consider the \(\hvec\) valued operators
\begin{equation}
  X=\sum_{i=1}^{\dim\plat} \mu^i\otimes \mu_i,\qquad \log_s K=
  \sum_{j=1}^{\rk\lat}\nu^j \otimes \log_sK_{\nu_j},
\end{equation}
which define maps \(\Mmod\to \hvec_\CC\otimes \Mmod\) by the action
\begin{equation}
  X|_{\Mmod_\alpha}=\sum_{i=1}^{\dim\plat} \mu^i\pair{\mu_i}{s(\alpha)}\otimes\id_{\Mmod_\alpha},\qquad
  \log_s K|_{\Mmod_{\alpha}} = \sum_{j=1}^{\rk\lat}\nu^j \otimes\pair{\nu_j}{s(\alpha)}\otimes\id_{\Mmod_\alpha}.
\end{equation}
So for any function \(f:\brac*{\dlat}^{n}\to \CC^\times\), \(n\in\NN\)
  and \(\alpha_1,\dots,\alpha_n\in\dlat\), we define
  the linear operator
\begin{equation}
  f(X_1+\log_s K_1,\dots,X_n+\log_s
  K_n)|_{\Mmod_{\alpha_1}\otimes\cdots\otimes \Mmod_{\alpha_n}} = f(\alpha_1,\dots,\alpha_n) \id_{\Mmod_\alpha}.
\end{equation}
Then we define the following ribbon element \(r:\Mmod\to
\Mmod\), \(R\)-matrix \(R:\Mmod\otimes\Nmod\to
\Mmod\otimes \Nmod\)  and coassociator \(\Phi:\Mmod\otimes\brac*{\Nmod\otimes\Pmod}\to
\Mmod\otimes \brac*{\Nmod\otimes\Pmod}\),
whose names will be justified by \cref{thm:hpfgvstr}.
\begin{align}
  r&=\exp\sqbrac*{-\pi\ii\pair{X_1+\log_sK_1}{X_1+\log_sK_1-2s(\ffv)}},
  \qquad   R =  \exp\sqbrac*{\ii \pi \brac*{\pair{X_1+\log_s{K}_1}{X_2+\log_sK_2}}},
  \nonumber\\
  \Phi &=
  \exp\sqbrac*{\ii \pi \pair{X_1 + \log_sK_1}{\log_sK_2+\log_sK_3-\log_sK_{2\otimes3}}}\nonumber\\
  &\qquad\cdot \frac{\tcyc{k(X_1 + \log_sK_1,X_2 + \log_sK_2)}{k(X_1 + \log_sK_1+X_2 + \log_sK_2,X_3 + \log_sK_3)}}{\tcyc{k(X_2 + \log_sK_2,X_3 + \log_sK_3)}{k(X_1 + \log_sK_1,X_2 + \log_sK_2+X_3 + \log_sK_3)}},
  \label{eq:hopfdata}
\end{align}
where \(\log_s(K)_{2\otimes3}\) is to be evaluated after the \(\Nmod\otimes \Pmod\)
tensor product has been evaluated, that is, \(\log_s(K_\nu)|_{\Nmod_\alpha\otimes\Pmod_\gamma}=\pair{\nu}{s(\alpha+\gamma)}\id_{\Nmod_\alpha\otimes\Pmod_\gamma}\).

\begin{thm}
  Let \(\lhmd{\ldat}\) be the category of \(H_\lat\) modules constructed
  from \(\Vect\brac*{\ldat}\), that is, the objects are the pairs
  \((\Mmod,\rho_\Mmod)\), \(\Mmod\in \Vect\brac*{\ldat}\) and the morphism are
  \(H_\lat\) module homomorphisms (these are precisely the morphisms of
  \(\Vect\brac*{\ldat}\)). Define a tensor functor on \(\lhmd{\ldat}\) by
  \begin{align}
    (\Mmod, \rho_{\Mmod}) \otimes (\Nmod, \rho_{\Nmod}) &=
    \brac*{\Mmod \otimes \Nmod, (\rho_{\Mmod} \otimes \rho_{\Nmod}) \circ \Delta},\qquad 
    \Mmod,\Nmod\in \Vect\brac*{\ldat},
  \end{align}
  where the tensor product of morphisms is the standard tensor product of
  linear maps.
  \begin{enumerate}
  \item The ribbon element, \(R\)-matrix and coassociator given in
    \eqref{eq:hopfdata} equip \(\lhmd{\ldat}\) with the structure of a
    \rgv{} category with twist \(\theta\), braiding \(c\) and associator \(\alpha\) respectively given by
    \begin{equation}
      \theta_\Mmod =r^{-1},\qquad
      c_{\Mmod,\Nmod}=P\circ R,\qquad
      \alpha_{\Mmod,\Nmod,\Pmod}=\alpha^{\text{vec}}\circ \Phi,
    \end{equation}
    where \(P\) is the tensor flip of vector spaces and \(\alpha^{\text{vec}}\) is
    the standard associator of vector spaces. All future references to
    \(\lhmd{\ldat}\) will include the \rgv{} structure given here.
  \item Let $F : \Vect\brac*{\ldat} \ra \lhmd{\ldat} $ be the
    functor which equips the vector space $\Mmod\in \Vect\brac*{\ldat}$ with the \(H_\lat\) action defined by the representation \(\rho_\Mmod\), that is,
    \begin{equation}
      F: \Mmod \mapsto (\Mmod, \rho_\Mmod),
    \end{equation}
    and which is the identity on morphisms.
    Let the isomorphism $\varphi_0 : (\CC_0, \rho_{\CC_0}) \ra F(\CC_0)=\CC_0$ be
    the identity map \(\id_{\CC_0}\) on the tensor unit $\CC_0$. Let
    $\tstr_2: F(-) \otimes F(-) \ra F(- \otimes -)$ be the natural
    transformation given by
    \begin{align}
      \tstr_2((\Mmod, \rho_\Mmod), (\Nmod, \rho_\Nmod))=\id_{\Mmod\otimes\Nmod} .
    \end{align}
    Then $(F, \tstr_0, \tstr_2)$ is a \rgv{} equivalence.
  \end{enumerate}
  \label{thm:hpfgvstr}
\end{thm}
\begin{proof}
  The proposed tensor product functor \(\otimes\) is well defined, because
  \(H_\lat\) is a Hopf algebra. We can therefore use the proposed tensor
  functor \(F\) to map the twist, braiding and
  associativity isomorphisms from \(\Vect\brac*{\ldat}\) to
  \(\lhmd{\ldat}\). If the images of these structure morphisms match the evaluations of
  the formal operators \eqref{eq:hopfdata}, it then automatically follows that
  these operators satisfy the defining properties of ribbon elements,
  \(R\)-matrices and coassociators and that \((F,\tstr_0,\tstr_2)\) is an
  equivalence of \rgv{} categories.
  Let \(\eta,\kappa,\tau\in\qlat\) and \(\Mmod,\Nmod,\Pmod\in \lhmd{\ldat}\),
  then
  \begin{align}
    r|_{\Mmod_\eta}&=\exp\sqbrac*{-\pi\ii\pair{\sum_{i=1}^{\dim\plat}\mu^i\pair{\mu_i}{s(\eta)}+\sum_{j=1}^{\rk\lat}\nu^j\pair{\nu_j}{s(\eta)}}{\sum_{k=1}^{\dim\lat}\mu^k\pair{\mu_k}{s(\eta)}+\sum_{l=1}^{\rk\lat}\nu^l\pair{\nu_l}{s(\eta)}-2s(\ffv)}}\nonumber\\
    &=\ee^{-\pi\ii\pair{s(\eta)}{s(\eta)-2s(\ffv)}}\id_{\Mmod_\eta}
  \end{align} 
  and similarly,
  \begin{equation}
  R\vert_{\Mmod_{\eta} \otimes \Nmod_\kappa} = \ee^{\ii \pi \pair{\sct{\eta}}{\sct{\kappa}}} \id_{\Mmod_\eta \otimes \Mmod_\kappa}, \qquad \Phi\vert_{\Mmod_\eta\otimes\Nmod_\kappa\otimes\Pmod_\tau} = (-1)^{\pair{s(\eta)}{k(\kappa,\tau)}}\frac{\tcyc{k(\eta,\kappa)}{k(\eta+\kappa,\tau)}}{\tcyc{k(\kappa,\tau)}{k(\eta,\kappa+\tau)}} \id_{\Mmod_\eta\otimes\Nmod_\kappa\otimes\Pmod_\tau}.
\end{equation}
Therefore the ribbon element, \(R\) matrix and coassociator evaluate exactly as the
twist, braiding isomorphisms and associativity isomorphism in
\(\Vect\brac*{\ldat}\) do
and the theorem follows. The equivalence of the \gv{} structures then follows
by noting that \(\CC_{2\ffv}\) is the dualising object for both categories.
\end{proof}

\begin{ex}
  Recall the example from the end of \cref{sec:data}.
  \begin{enumerate}
 \item  If the lattice \(\lat\) is full rank, then \(\plat\) is trivial and
  \(\qlat\) is a finite group. In this case the lattice Hopf algebra is just
  the group algebra \(\CC\sqbrac*{\dlat}\). 
 \item If \(\lat\) is the trivial
  lattice, then \(\plat=\dlat=\hvec\) and in this case the lattice Hopf algebra
  is the universal enveloping algebra \(\UEA{\plat}\) of the complexification
  of \(\plat\).
 \item  Finally, in the half rank example \(\dlat\cong \ZZ\times \RR\) and so
   the lattice Hopf algebra is a tensor product of the \(\ZZ\)-group algebra and
  the universal enveloping algebra of the abelian one-dimensional Lie
  algebra \(\mathfrak{gl}(1)\). Further, the modules defined by the action \eqref{eq:haction}
  descend to modules over the group \(U(1)\times \ZZ\).
  Explicitly we can give the lattice Hopf algebra as
  \begin{align}
		H_\lat = \CC[X, K, K^{-1}], \qquad &K^{\pm 1} K^{\mp 1} = 1, \nonumber\\
		\Delta(X) = X \otimes 1 + 1 \otimes X, \qquad &S(X) = -X, \\
		\Delta(K^{\pm 1}) = K^{\pm 1} \otimes K^{\pm 1}, \qquad &S(K^{\pm 1}) = K^{\mp 1} \nonumber.
	\end{align}
The action on the module \(\CC_{x_1,x_2+\ZZ}\) is then given by
\begin{equation}
		\rho_{\CC_{x_1, x_2+\ZZ}} \brac*{X} = x_1
		\id_{\CC_{x_1, x_2+\ZZ}}
		,  \qquad \rho_{\CC_{x_1, x_2+\ZZ}} \brac*{K} = \ee^{ 2 \pi \ii \tilde{x}_2} 
		\id_{\CC_{x_1, x_2+\ZZ}}.
\end{equation}
The ribbon element, $R$-matrix and coassociator for this choice of data are then 
\begin{align}
  r&=\exp\brac*{-2\ii\pi\brac*{X_1\log_s K_1-\log_sK_1}}&
  R &=\exp\sqbrac*{\ii\pi\brac*{\log{K_1}\otimes X_2 + X_1\otimes \log{K_2}}},\nonumber\\
  \Phi &= \exp\sqbrac*{\ii\pi \brac*{X_1 \otimes \log{K_2}\otimes\id
         +X_1\otimes\id\otimes \log{K_3} - X_1 \otimes \log{K_{2\otimes 3}}}},
\end{align}
where $\log{K}$ acts as $ \tilde{x}_2$ on $\CC_{x_1, x_2+\ZZ}$.
\end{enumerate}
\end{ex}

\subsection{Simple Current Extensions}

The process of extending a \voa{} by (tensor powers of) modules whose tensor product is
invertible (such extentions are called 
\emph{simple current extentions}) has a long history in the \cft{} and \voa{} literature for both
finite order extensions \cite{SchYank90} and more recently also infinite ones \cite{AugInf17,CreuDLim20}. At a categorical level, extensions (not
necessarily the simple current type) correspond to algebra objects in a
braided monoidal category \cite{KirAlg01,BTC15,Cre17}.  In particular, algebra
objects in categories of graded vectors spaces and their connections to
\voa{s} and \cft{} have been studied in \cite{FucTFT304}.
Let $\ldat_i = (\hvec, \pair{-}{-}, \lat_i, \ffv_i)$ for $i = 1,2$ be two sets of \data{}. Then by \cref{thm:voaequiv,thm:hpfgvstr} we have two triples of ribbon \gv{} equivalent categories 
\begin{equation}
  {\Vect}(\ldat_i) \ \ \cong \ \ \lvmd{\ldat_i} \ \ \cong \ \ \lhmd{\ldat_i}, \quad i = 1,2.
\end{equation}
We will show that if $\lat_1 \subset \lat_2$ and $\ffv_1 \subset \ffv_2$, we
can find an algebra object $\Amod$ in the direct sum completion $\Vect
(\ldat_1)_\oplus$ such that the module category for $\Amod$ is equivalent to
${\Vect}(\ldat_2)$. Transferring the algebra object  $\Amod$ to
\(\lvmd{\ldat_1}_\oplus\) then yields the simple current extension of
$\lvoa{\rffv}{\lat_1}$ to $\lvoa{\rffv}{\lat_2}$, if we choose the same
representative \(\rffv\) for both \(\ffv_1\) and \(\ffv_2\). 
Finally, we will pose the problem of constructing $H_{\lat_2}$ from $H_{\lat_1}$.

\begin{prop}
	Let $\ldat_1$, $\ldat_2$ be two sets of bosonic lattice data, satisfying $\lat_1 \subset \lat_2$ and $\ffv_1 \subset \ffv_2$.
	Let $\sigma : \lat_2/\lat_1 \otimes \lat_2/\lat_1 \ra \CC^\times$ satisfy
	\begin{align} \label{eq:2cocyc}
		&\sigma(\lambda, \lat_1) = \sigma(\lat_1, \lambda) = 1, \qquad \sigma(\lambda_1, \lambda_2) \sigma(\lambda_2, \lambda_1)^{-1} = \Omega(\lambda_1, \lambda_2), \nonumber \\ &\sigma(\lambda_2, \lambda_3) \sigma(\lambda_1 + \lambda_2, \lambda_3)^{-1} \sigma(\lambda_1, \lambda_2+ \lambda_3) \sigma(\lambda_1, \lambda_2)^{-1} = F(\lambda_1, \lambda_2, \lambda_3),
	\end{align}
	where $(F,\Omega)$ is the abelian 3-cocycle associated to $\ldat_1$.
	Taking $\lat_2 / \lat_1 \subset 
	\lat_1^* / \lat_1$ as a subgroup, we define the triple $(\Amod,
        \mu:\Amod \otimes \Amod \ra \Amod, \eta:\CC_{\lat_1}\to \Amod)$ by
	\begin{equation}
		\Amod = \bigoplus
		_{\lambda \in \lat_2 / \lat_1}
		\CC_{\lambda}, \qquad 
		\mu \vert_{\CC_{\lambda_1} \otimes \CC_{\lambda_2}} =
                \sigma(\lambda_1, \lambda_2) J_{\lambda_1,\lambda_2},  \qquad \eta = \id_{\CC_{\lat_1}},
	\end{equation}
        where \(J_{\lambda_1,\lambda_2}\) is the canonical identification
          \(\CC_{\lambda_1} \otimes \CC_{\lambda_2}\cong
          \CC_{\lambda_{1} + \lambda_{2}}\).
          Then
	\begin{enumerate}
	\item $(\Amod, \mu, \eta)$ defines an associative commutative algebra
          with trivial twist and a unique unit (that is, \(\dim \Homgrp{}{\CC_0}{\Amod}=1\), also
          called the haploid condition), in 
	$\Vect (\ldat_1)_\oplus$. 
	\item The category of local \(\Amod\)-modules \(\Amod
            \modules{}^{\mathrm{loc}} \brac*{{\Vect}(\ldat_1)_\oplus}\) (also called
            dyslectic modules) is a \rgv{} category and is equivalent to
          \({\Vect}(\ldat_2)\).
          Thus the images of $\Amod$ under the functors $G$ and \(F\) in
          \cref{thm:voaequiv,thm:hpfgvstr} define equivalent algebras
          $\valg=G(\Amod)$ and $\halg=F(\Amod)$ in \(\lvmd{\ldat_1}_\oplus\)
          and \(\lhmd{\ldat_1}_\oplus\), respectively.  
        Hence we have the sequence 
	\begin{equation}
          {\Vect}(\ldat_2)\cong \valg \modules{}^{\mathrm{loc}} \brac*{\lvmd{\ldat_1}_\oplus} \cong
          \lvmd{\ldat_2}\cong \halg \modules{}^{\mathrm{loc}} \brac*{\lhmd{\ldat_1}_\oplus} \cong \lhmd{\ldat_2} .
        \end{equation}
        of ribbon \gv{} equivalences.
      \item Let \(\rffv\in \ffv_1\subset\ffv_2\) be a choice of representative for both \(\ffv_1\) and
        \(\ffv_2\). The algebra object $\valg =G(\Amod)$ admits the structure of a \voa{} via
        the field map \(Y=\btf_\mu\), with vacuum and conformal vectors given by
      the images of the vacuum and conformal vectors in
      \(\lvoa{\rffv}{\lat_1}\) under the tensor structure map \(\tstr_0\).
      Further, this \voa{} is isomorphic to $\lvoa{\rffv}{\lat_2}$.
	\end{enumerate}
\end{prop}

\begin{proof}
  Denote by \(s_i\) the respective sections of the bosonic lattice data
  \(\ldat_i\). The 2-cocycles \(k\) and \(\tcycsymb\) shall only be needed for
  \(\ldat_1\) and will hence not be given an index, to reduce notational clutter.
  \begin{enumerate}
  \item The conditions \eqref{eq:2cocyc} are equivalent to the
      constraints imposed on \(\mu\) and \(\eta\) by the definition of an
    associative unital commutative algebra
    \cite{EtiTen15}[Definitions 7.8.1 and 8.8.1]. Unitality is implied by the
    first relation, commutativity by the second and associativity by the
    third.

    The haploid or uniqueness of the unit property follows from
    \(\Amod\) containing \(\CC_{\lat_1}\) only once as a direct summand and
    \(\dim \Homgrp{}{\CC_{\lat_1}}{\CC_{\lat_1}}=1\).

    The algebra having trivial twist follows by direct computation. On
      each summand of \(\Amod\), the twist evaluates to
      \(\theta(\lambda)=\ee^{\ii\pi\pair{s_1(\lambda)}{s_1(\lambda)-2s_1(\ffv_1)}}\),
      \(\lambda_2\in \lat_2/\lat_1\). Since \(\lat_2\) is even,
      \(s_1(\lambda)\in\lat_2\) and \(s_1(\ffv_1)\in \ffv_2\subset
      \dlat_2\), we have \(\pair{s_1(\lambda)}{s_1(\lambda)},\ 
      2\pair{s_1(\lambda)}{s_1(\ffv_2)}\in 2\ZZ\) and hence the twist is trivial.
  \item Let $\Amod\modules \brac*{{\Vect}(\ldat_1)_\oplus}$ be the category of all
    $\Amod$-modules in ${\Vect}(\ldat_1)_\oplus$.
    Combining \cite[Theorem 1.6]{KirAlg01}, which asserts that induction and
    restriction are adjoint, exact and injective on morphisms, and that
    induction is a tensor functor with the semisimplicity of
    \({\Vect}(\ldat_1)_\oplus \), we can quickly deduce that
    $\Amod\modules \brac*{{\Vect}(\ldat_1)_\oplus}$ is also semisimple and that every
    simple object in $\Amod\modules \brac*{{\Vect}(\ldat_1)_\oplus}$ is the induction
    of a simple object in \({\Vect}(\ldat_1)\). We denote the simple modules
    induced from the \(\CC_\alpha\), \(\alpha\in \dlat_1/\lat_1\) by
    \begin{equation}
      \Nmod_{\alpha} = \Amod \otimes  \CC_{\alpha} \cong \bigoplus
      _{\lambda \in \lat_2 / \lat_1} 
      \CC_{\lambda + \alpha}. 
    \end{equation}
    Let $\Amod \modules{}^{\mathrm{loc}} \brac*{{\Vect}(\ldat_1)}$ be the
    full subcategory of local modules, that is, all objects
    which have trivial double braiding with the algebra
    $\Amod$. For one of the \(\Nmod_\alpha\) above this means that for all
    $\lambda \in \lat_2 / \lat_1$, we require that
    \begin{equation}
      \Omega(\lambda, \alpha) \Omega(\alpha, \lambda) = \ee^{2 \pi \ii
        \pair{s_1 (\lambda)}{s_1 (\alpha)}} = 1, \ \ \text{or equivalently}  \ \ \pair{s_1 (\lambda)}{s_1 (\alpha)} \in \ZZ.
    \end{equation}
    By assumption \(s_1(\alpha)\in \dlat_1\). If \(s_1(\alpha)\in
      \dlat_2\), then the above condition is satisfied for all \(\lambda \in
      \lat_2/\lat_1\). Conversely, if \(s_1(\alpha)\notin \dlat_2\) then there
      exits a \(\mu \in \lat_2\) such that \(\pair{\mu}{s_1(\alpha)}\notin
      \ZZ\). But then \(s_1(\alpha)\) would pair non-integrally with every
      representative of the \(\lat_1\) coset of \(\mu\) and hence the above
      condition cannot be satisfied.
    Therefore $\alpha \in \lat_2^* / \lat_1$ exhausts all labels for
      simple objects in $\Amod \modules{}^{\mathrm{loc}}
      \brac*{{\Vect}(\ldat_1)}$. Two induced simple modules
    $\Nmod_{\alpha}$, $\Nmod_{\beta}$ are isomorphic if and only if their labels
    differ by a coset in $\lat_2 /\lat_1$. Therefore the isomorphism classes
    of simple modules are labelled by the elements of the quotient group $ \brac*{ \lat_2^* / \lat_1} / \brac*{\lat_2 / \lat_1} \cong \lat_2^* /  \lat_2$. 
    This implies that $\Amod \modules{}^{\mathrm{loc}}
    \brac*{{\Vect}(\ldat_1)_\oplus}$ and ${\Vect}(\ldat_2)$ are equivalent as
    abelian categories.
    By \cite[Theorem 1.10]{KirAlg01} or \cite[Theorem 2.5]{ParAlg95}, $\Amod \modules{}^{\mathrm{loc}}
    \brac*{{\Vect}(\ldat_1)_\oplus}$ is braided monoidal with
    the braiding descending from \({\Vect}(\ldat_1)\). Further,
    from \cite[Theorem 1.6]{KirAlg01} one can deduce that
    \(\Nmod_\alpha\otimes_\Amod \Nmod_\beta\cong \Nmod_{\alpha+\beta}\). Thus
   $\Amod \modules{}^{\mathrm{loc}}
   \brac*{{\Vect}(\ldat_1)_\oplus}$ also has the same tensor product as
    ${\Vect}(\ldat_2)$, hence the braiding and associativity isomorphisms are
    characterised by abelian 3-cocycles for the group \(\dlat_2/\lat_2\). To
    conclude equivalence as braided monoidal categories it is therefore
    sufficient for the trace of the abelian 3-cocycles of $\Amod \modules{}^{\mathrm{loc}}
    \brac*{{\Vect}(\ldat_1)_\oplus}$ and ${\Vect}(\ldat_2)$ to be equal.
    Let \(\Omega_i\), \(i=1,2\) be the respective braidings associated to
    \(\ldat_i\), then for \(\alpha\in \dlat_2/\lat_1\) we need to compare
    \(\Omega_1(\alpha,\alpha)\) and
    \(\Omega_2(\alpha+\lat_2,\alpha+\lat_2)\). Recall that \(s_1(\alpha)\in
    \dlat_2\) and hence \(s_2(\alpha+\lat_2)-s_1(\alpha)=\kappa\in \lat_2\), so
    \begin{equation}
      \Omega_2(\alpha+\lat_2,\alpha+\lat_2)=\ee^{\ii\pi\pair{s_2(\alpha+\lat_2)}{s_2(\alpha+\lat_2)}}=\ee^{\ii\pi\pair{s_1(\alpha)+\kappa}{s_1(\alpha)+\kappa}}=\ee^{\ii\pi\pair{s_1(\alpha)}{s_1(\alpha)}}=\Omega_1(\alpha,\alpha),
    \end{equation}
    where the third equality follows from \(\lat_2\) being even. Thus $\Amod \modules{}^{\mathrm{loc}}
    \brac*{{\Vect}(\ldat_1)_\oplus}$ and ${\Vect}(\ldat_2)$ are equivalent as
    braided monoidal categories.
    \gv{} equivalence follows by noting that the induction of the dualising
    object \(\Nmod_{\ffv_1}\) has \(\ffv_1+\lat_2=\ffv_2\) as its label and is
    hence equivalent to the dualising object of ${\Vect}(\ldat_2)$.
    Finally, ribbon equivalence follows by comparing the twist scalars \(\theta_1,\theta_2\) in both
    categories. We denote \(s_2(\ffv_2)-s_1(\ffv_1)=\tau\in\lat_2\) and
    consider for any \(\alpha\in \dlat_2/\lat_1\)
    \begin{equation}
      \theta_2(\alpha+\lat_2)=\ee^{\ii\pi\pair{s_2(\alpha+\lat_2)}{s_2(\alpha+\lat_2)-2s_2(\ffv_2)}}
      =\ee^{\ii\pi\pair{s_1(\alpha)+\kappa}{s_1(\alpha)+\kappa-2s_1(\ffv_1)-2\tau}}
      =\ee^{\ii\pi\pair{s_1(\alpha)}{s_1(\alpha)-2s_1(\ffv_1)}}=\theta_1(\alpha),
    \end{equation}
    where we have again used the \(\lat_2\) is even.
    Thus $\Amod \modules{}^{\mathrm{loc}}
    \brac*{{\Vect}(\ldat_1)_\oplus}$ and ${\Vect}(\ldat_2)$ are \rgv{} equivalent.
  \item 
    As a module over the Heisenberg algebra $\valg$ decomposes as follows.
    \begin{equation}
      \valg = G(\Amod) = \bigoplus_{\lambda \in \lat_2 / \lat_1} G(\Mmod_{\lambda}) \cong \bigoplus_{\lambda \in \lat_2 / \lat_1} \LFock{\lambda} \cong \bigoplus_{\lambda \in \lat_2/\lat_1 } \bigoplus_{\alpha \in \lat_1 } \Fock{s_1(\lambda) + \alpha} = \bigoplus_{\lambda \in \lat_2 } \Fock{\lambda} ,
    \end{equation}
    which is isomorphic to the vector space on which $\lvoa{\rffv}{\lat_2}$ is defined.
    We need to verify that \(Y = \btf_{\mu}\) is indeed a field map.
    We show this by comparing \(\btf_{\mu}\) to the field map of
    $\lvoa{\rffv}{\lat_2}$. Consider \(\lambda_1,\lambda_2\in \lat_2/\lat_1\),
    \(\alpha_1,\alpha_2\in\lat_1\), then
    \(\btf_{\mu}|_{\brac{G(\CC_{\lambda_1})\otimes \ee^{\alpha_1}} \otimes \brac{G(\CC_{\lambda_2})\otimes\ee^{\alpha_2}}}\) is
    essentially an untwisted vertex operator of the form \eqref{eq:utwvop} up
    to a scaling factor of
    \begin{equation}
      (-1)^{\pair{s_1(\lambda_1)}{\alpha_2}}\tcyc{\alpha_1}{\alpha_2}\tcyc{\alpha_1+\alpha_2}{k(\lambda_1,\lambda_2)}
      \sigma(\lambda_1,\lambda_2).
    \end{equation}
    Therefore \(\btf_{\mu}\) defines a \voa{} structure if and only if
    \begin{align}
      \tau(\gamma,\delta)&=(-1)^{\pair{s_1(\gamma+\lat_1)}{\delta-s_1(\delta+\lat_1)}}\tcyc{\gamma-s_1(\gamma+\lat_1)}{\delta-s_1(\delta+\lat_1)}\nonumber\\
      &\qquad\tcyc{\gamma-s_1(\gamma+\lat_1)+\delta-s_1(\delta+\lat_1)}{k(\gamma+\lat_1,\delta+\lat_1)}
      \sigma(\gamma+\lat_1,\delta+\lat_1),\qquad \gamma,\delta\in \lat_2
  \end{align}
  satisfies the 2-cocycle
    conditions of \eqref{eq:voacocycle} for \(\lat_2\). Since all 2-cocycles
    for \(\lat_2\) are cohomologous, the \voa{} structure defined by
    \(\btf_{\mu}\) is isomorphic to that of $\lvoa{\rffv}{\lat_2}$.
  \end{enumerate}
\end{proof}

\renewcommand{\mapsfrom}{\mathrel{\reflectbox{\ensuremath{\mapsto}}}}
\begin{prob}[Simple Current Extension of Hopf Algebras] Consider a
  quasitriangular (quasi-)Hopf algebra $H$ over a field \(\Bbbk\) with
  $R$-matrix written as $R=\sum_i R^{(i)}_1\otimes
  R_2^{(i)},\ R_1^{(i)},R_2^{(i)}\in H$, and a group $\Gamma$ of $1$-dimensional characters
  $\phi:H\to \Bbbk$ such that $\sum_{i,j}\phi(R_2^{(i)}R_1^{(j)})\psi(R_1^{(i)}R_2^{(j)})=1$ for all
  $\phi,\psi\in \Gamma$. Each such character \(\phi\) defines a
  one-dimensional module \(\Bbbk_\phi\) on which \(h\in H\) acts as \(\phi(h)\id\).
  Then
  then the object $\halg=\bigoplus_{\phi\in\Gamma} \Bbbk_\phi$ can be
  endowed with the structure of a commutative algebra in $H\modules{}$ using
  the multiplication in $\Gamma$ and the $2$-cocycle
  $\sigma(\phi,\psi)=\sum_i\phi(R_1^{(i)})\psi(R_2^{(i)})$ (this is a 2-cocycle because it
  satisfies the pentagon identity
  $\sigma(\phi\ast\rho,\psi)=\sigma(\phi,\psi)\sigma(\rho,\psi),\
  \phi,\rho,\psi\in \Gamma$, where
  \(\ast\) is convolution). 
  Can one construct a quasitriangular (quasi-)Hopf algebra, whose module
  category is \rgv{} equivalent to 
  $\halg\modules{}^{\mathrm{loc}}(H\modules{})$, the category of local \(\halg\)-modules?
\end{prob}

\section{The impact of \gv{} structure on characters and modular transformations}
\label{sec:chars}

We conclude this paper with a final section giving observations on
the modular properties of lattice module characters. Ideally one would want to
extract from \rgv{} categories some analogue of the rich structures enjoyed by
modular tensor categories such as a generalisation of the mapping class group
action and the Verlinde
formula. In particular this would require some notion of categorical trace. The
tools and understanding required for this have, however, not yet been
developed and we hope to return to this in the future.
To support this future research, we record here character formulae and their
modular transformation properties and show that they admit a naive
generalisation of the Verlinde formula in the sense of the standard module
formalism \cite{RidSL208}.

Let \(\ldat=(\hvec, \pair{-}{-},\lat,\ffv)\) be a set of \data\ as
defined at the beginning
of \cref{sec:data} and \(\rffv\) a choice of representative of \(\ffv\).

\begin{prop}
  Let \(\zeta \in \hvec_{\CC}\), \(\tau \in \mathbb{H}_+\), \(q=\ee^{2 \pi i \tau}\) and
  \(\gamma\in \qlat\). Then the character of the lattice Fock space
  \(\LFock{\gamma}\) as a \(\lvoa{\rffv}{\Lambda}\) module is 
  \begin{equation}
    \chi_{\gamma}^{\rffv} ( \zeta, \tau)=
    \mathrm{Tr}_{\LFock{\gamma}} \ee^{2 \pi \ii \pair{\zeta}{-}} q^{L_0 - c/24}
    =  \sum_{\lambda \in \lat} \ee^{2 \pi i
      \pair{\zeta}{s(\gamma) + \lambda}} \frac{q^{\frac{1}{2} \pair{s(\gamma)
          + \lambda - \rffv}{s(\gamma) + \lambda -
          \rffv}}}{\eta(\tau)^{\dim\hvec}} .
    \label{eq:charform}
  \end{equation}
\end{prop}
  Note that since the sum on the \rhs{} of \eqref{eq:charform} ranges over the entire lattice \(\lat\) it does not depend on the choice of representative \(s(\gamma)\) of the coset $\gamma$ and it only depends on the choice of representative \(\tilde{\xi}\) of the coset \(\xi\) by a global factor coming from the first exponential.
  \begin{proof}
    The argument of the sum is the well known character formula for non-lattice Fock
    spaces \(\Fock{\lambda}\). Hence the character for the lattice Fock space
    is just the sum of the characters of the \(\Fock{\lambda}\) summed over
    all \(\lambda\in \lat\).
  \end{proof}
  Recall the decomposition \(\dlat = V \oplus \rspn\nlat \oplus F \oplus D\) 
  in \cref{thm:hbasis}.\ref{itm:dlatdecomprefined}.
\begin{thm}\label{thm:moddat}
  The $T$-transformation of lattice module characters is
  \begin{equation}
    T \set{\chi^{\rffv}_{\gamma} (\zeta, \tau)} =\chi^{\rffv}_{\gamma} (\zeta, \tau+1)
    = \ee^{\pi \ii \pair{\gamma - \ffv}{\gamma - \ffv}} \chi^{\rffv}_{\gamma} (\zeta, \tau) .
  \end{equation}
  If $\pair{-}{-}$ restricted to the groups $V$ and $D$ is positive definite, 
  then the $S$-transformation is
  \begin{equation}
    S \set{\chi^{\rffv}_{\gamma} (\zeta, \tau)}
    = \chi^{\rffv}_{\gamma} \brac*{\frac{\zeta}{\tau}, \frac{-1}{\tau}}
    =\brac*{\frac{\abs{\tau}}{\ii \tau}}^{\rk \nlat} 
    \frac{ \ee^{2\pi i
      \brac*{\frac{\norm{\zeta}^2}{2\tau} +
        \frac{\pair{\rffv}{\zeta}}{\tau} - \pair{\rffv}{\zeta}}}}{\sqrt{|D/ \brac*{D\cap\lat}|}}
    \int_{\qlat} \ee^{- 2\pi \ii \pair{\gamma - \ffv}{\mu - \ffv}} \chi^{\rffv}_{\mu} (\zeta, \tau) d (\mu),
  \end{equation}
  where \(\int_{\qlat}\dd(\mu)=\int_V\dd(v) \cdot \sum_{f\in
    F}\int_{\rspn{\nlat}/\nlat} \dd(\nu) \cdot \sum_{t\in D/D\cap\lat}\), where
  the integrals over \(V\) and \(\rspn\nlat/\nlat\) are to be expanded as follows.
  \begin{itemize}
  \item Pick any $\RR$-basis \(\set{e_i}_{i=1}^{\dim V}\) of $V$ and let $\det \pair{-}{-}\vert_V$ be the
    determinant of the Gram matrix for the pairing \(\pair{-}{-}\) restricted
    to \(V\) in this basis. For any \(v\in V\) denote its expansion in the
    basis \(\set{e_i}_{i=1}^{\dim V}\) by \(v=\sum_i v_i e_i\), then
    \begin{equation}
      \int_V\psi(v)\dd(v)=\sqrt{\det \pair{-}{-}\vert_V} \int_{\RR^{\dim V}} \psi(\sum_{i}v_ie_i)\dd v_1\cdots
      \dd v_{\dim V}.
    \end{equation}
  \item Pick any \(\ZZ\)-basis \(\set{g_i}_{i=1}^{\rk \nlat}\) of
    \(\nlat\) (which will also be an \(\RR\)-basis for \(\rspn{\nlat}\)) and
    any \(\ZZ\)-basis \(\set{h_i}_{i=1}^{\rk \nlat}\) of
    \(F\). For
    any \(\nu\in \rspn{\nlat}\) and \(f\in F\)
    denote their expansion in their respective bases by
    \(\nu=\sum_i \nu_i g_i\) and \(f=\sum_{i} f_i h_i\), then
    \begin{equation}
      \sum_{f\in F}
      \int_{\rspn{\nlat}/\nlat} \psi(f,\nu)\dd(\nu)=
      \sum_{\substack{f_k\in\ZZ\\1\le k\le \rk \nlat}}\int_{[0,1]^{\rk \nlat}}
      \psi(
      \sum_j
      f_j h_j,\sum_i \nu_i g_i)\dd \nu_1\cdots \dd \nu_{\rk\nlat}.
    \end{equation}
  \end{itemize}
\end{thm}
Note that the $S$-transformation does not depend of the choice of the groups $D$  
and $V$ in the decomposition of \(\dlat\), as any change of choice will be compensated for by the determinants.
Note further that, due to \(\lat\) being integral and even, the expressions
\(\ee^{\ii\pi\pair{\gamma-\ffv}{\gamma-\ffv}}\) and
\(\ee^{-2\pi\ii\pair{\gamma-\ffv}{\mu-\ffv}}\) do not depend on the elements
of \(\dlat\) chosen to represent \(\gamma-\ffv,\ \mu-\ffv\) to evaluate the
pairings. Thus (up to the phase factor of \(S\)) the modular transformation formulae depend on \(\ffv\) (which
is also the datum that characterises the \rgv{} structures, that is the
dualising object and twist) but not on the choice of representative \(\rffv\).
The above theorem is most easily proved by decomposing the characters in a
manner compatible with the decomposition of \(\dlat\) in \cref{thm:hbasis}.\ref{itm:dlatdecomprefined}.

\begin{lem}
  Given \(\gamma\in\dlat\) and the decomposition \(\dlat = V \oplus \rspn\nlat
  \oplus F \oplus D\) of \cref{thm:hbasis}.\ref{itm:dlatdecomprefined},
  let \(\gamma_V, \gamma_\circ, \gamma_F,\gamma_D\) be the components of the
  representative \(s(\gamma)=\gamma_V+ \gamma_\circ+ \gamma_F+\gamma_D\) in
  the summands of \(\dlat\) and let
  \(\rffv=\rffv_V+\rffv_\circ+\rffv_F+\rffv_D\) be the analogous decomposition
  for \(\rffv\).
  The character $\chi^{\rffv}_{\gamma}$, $\gamma \in \qlat$ admits the following
  factorisation with contributions from $V$, $\rspn{\nlat} \oplus F$, $D$.
  \begin{equation}
    \chi^{\rffv}_{\gamma}\brac*{\zeta,\tau} =\chi^{\rffv, V}_{\gamma}\brac*{\zeta,\tau}
    \cdot \chi^{\rffv, \circ}_{\gamma}\brac*{\zeta,\tau}\cdot \chi^{\rffv, D}_{\gamma}\brac*{\zeta,\tau}
  \end{equation}
  where
  \begin{align}
    \chi^{\rffv, V}_{\gamma}\brac*{\zeta,\tau}&= \frac{\ee^{2 \pi \ii \pair{\zeta_V}{\gamma_V}}
                                                q^{\frac{1}{2} \norm{\gamma_V -
                                                \rffv_V}^2}}{\eta(\tau)^{\dim V}} , \nonumber\\
    \chi^{\rffv, \circ}_{\gamma}\brac*{\zeta,\tau}&= \frac{\ee^{2 \pi \ii \pair{\zeta_\circ}{\gamma_F}} q^{\pair{\rffv_\circ}{\rffv_F - \gamma_F}}}{\eta(\tau)^{2 \rk \nlat}} \sum_{f \in F} \ee^{2 \pi \ii \pair{\gamma_\circ}{f}} \delta_F \brac*{ \zeta_F + \tau (\gamma_F - \rffv_F) - f} , \nonumber\\
    \chi^{\rffv,D}_{\gamma}\brac*{\zeta,\tau} &=   \sum_{\kappa \in \lat \cap D} \frac{\ee^{2 \pi \ii \pair{\zeta_D}{\kappa + \gamma_D}} q^{\frac{1}{2} \norm{\gamma_D +\kappa- \rffv_D}^2}}{\eta(\tau)^{\rk D}} ,
  \end{align}
  and where $\delta_F(x)$, is the lattice $\delta$-distribution on $F$, that is,
  \begin{equation}
    \delta_F (x) = \prod_{j=1}^{\rk \nlat} \delta \brac*{\pair{a_j}{x}}, \quad
    x\in \rspn{F},\ \set{a_j} \ \text{any } \ZZ\text{-basis of} \ \nlat .
  \end{equation}
  \label{thm:chardecomp}
\end{lem}
\begin{proof}
  Then result follows by direct calculation and using the fact the the three subgroups \(V,\ \rspn\nlat
  \oplus F,\ D\) are mutually orthogonal.
  The only complication is the $\nlat$, $F$ contribution, which we sketch here. 
  Recall that $\nlat$ and $F$ are orthogonal to themselves but pair
  crosswise. Let $\set{a_i}_{i=1}^{\dim \nlat}$ be a $\ZZ$-basis of
  $\nlat$ and let \(\set{a^i}_{i=1}^{\dim \nlat}\) be its dual in \(F\).
  \begin{align}
    \sum_{\kappa \in \nlat} & \frac{\ee^{2 \pi \ii \pair{\zeta_\circ + \zeta_F}{\gamma_\circ + \gamma_F + \kappa}} q^{\frac{1}{2} \norm{\gamma_\circ + \gamma_F + \kappa - \rffv_{\circ} - \rffv_F}^2}}{\eta(\tau)^{2 \rk \dlat}} = \frac{\ee^{2 \pi \ii \pair{\zeta_\circ}{\gamma_F}} \ee^{2 \pi \ii \pair{\zeta_F}{\gamma_\circ}}}{\eta (\tau)^{2 \rk \dlat}} q^{\pair{\gamma_\circ - \rffv_{\circ}}{\gamma_F - \rffv_F}} \sum_{\kappa \in \nlat} {\ee^{2 \pi \ii \pair{\zeta_F+ \tau(\gamma_F - \rffv)}{\kappa}}} \nonumber\\
    &= \frac{\ee^{2 \pi \ii \pair{\zeta_\circ}{\gamma_F}} \ee^{2 \pi i
      \pair{\zeta_F}{\gamma_\circ}}}{\eta (\tau)^{2 \rk \dlat}}
      q^{\pair{\gamma_\circ - \rffv_{\circ}}{\gamma_F - \rffv_F}}
      \sum_{\substack{n_i \in \ZZ\\1\le i\le \rk\nlat}} \ee^{ 2 \pi \ii n_i \pair{\zeta_F
      + \tau(\gamma_F - \rffv)}{a_i}} \nonumber\\
    &= \frac{\ee^{2 \pi i
      \pair{\zeta_\circ}{\gamma_F}} \ee^{2 \pi i
      \pair{\zeta_F}{\gamma_\circ}}}{\eta
      (\tau)^{2 \rk \dlat}} q^{\pair{\gamma_\circ
      - \rffv_{\circ}}{\gamma_F - \rffv_F}}
      \sum_{\substack{n_i \in \ZZ\\1\le i\le \rk\nlat}}
      \delta\brac*{\pair{\zeta_F + \tau(\gamma_F
      - \rffv)}{a_i} -n_i}\nonumber\\
    &= \frac{\ee^{2 \pi i
      \pair{\zeta_\circ}{\gamma_F}} \ee^{2 \pi i
      \pair{\zeta_F}{\gamma_\circ}}}{\eta
      (\tau)^{2 \rk \dlat}} q^{\pair{\gamma_\circ
      - \rffv_{\circ}}{\gamma_F - \rffv_F}} \sum_{\substack{n_i \in \ZZ\\1\le i\le \rk\nlat}}\delta_F \brac{\zeta_F + \tau (\gamma_F- \rffv_F) -  n_i a^i}
      \nonumber\\
    &=\frac{\ee^{2 \pi i
      \pair{\zeta_\circ}{\gamma_F}} \ee^{2 \pi i
      \pair{\zeta_F}{\gamma_\circ}}}{\eta
      (\tau)^{2 \rk \dlat}} q^{\pair{\gamma_\circ
      - \rffv_{\circ}}{\gamma_F - \rffv_F}}
      \sum_{f \in F}\delta_F
      \brac{\zeta_F + \tau (\gamma_F- \rffv_F)- f}\nonumber\\
    &=\frac{\ee^{2 \pi \ii \pair{\zeta_\circ}{\gamma_F}} q^{\pair{\rffv_\circ}{\rffv_F - \gamma_F}}}{\eta(\tau)^{2 \rk \nlat}} \sum_{f \in F} \ee^{2 \pi \ii \pair{\gamma_\circ}{f}} \delta_F \brac*{ \zeta_F + \tau (\gamma_F - \rffv_F) - f},
  \end{align}
  where the third identity follows from the Dirac comb identity
  \begin{equation}
    \sum_{k\in \ZZ}\delta\brac*{x-k} = \sum_{n\in\ZZ}\ee^{2\pi\ii nx},
  \end{equation}
  and the final identity is obtained by using the \(\delta\) distributions to
  substitute
  \(\tau(\gamma_F-\rffv_F)=f-\zeta_F\) in the exponent of \(q\).
\end{proof}
\begin{proof}[Proof of \cref{thm:moddat}]
  The $T$-transformation expression is immediate. The $S$-transformation
  formulae follow from computing the \(S\)-transformations of the three
  factors in \cref{thm:chardecomp}. All three cases boil down to repeated
  evaluation of Gaussian integrals, that is, the well known identity
  \begin{equation}
    \int^{\RR^n}\ee^{-\frac12 \sum_{i,j=1}^n x_i A_{i,j} x_j+\sum_{i=1}^n
      B_ix_i}\dd x_1\cdots \dd x_n=\sqrt{\frac{\brac*{2\pi}^n}{\det
        A}}\ee^{\frac{1}{2}B^T A^{-1}B},
    \label{eq:gint}
  \end{equation}
  where \(A\) is a symmetric positive matrix and \(B\) is a real \(n\)-vector.

 We first determine the \(S\) transformation of the \(V\) part, \(\chi^{\rffv,V}_{\gamma}\). Fix an
 \(\RR\)-basis \(\set{e_i}_{i=1}^{\dim V}\) of \(V\) and consider
  \begin{align}
    \int_{V} \ee^{-2\pi \ii \pair{\gamma_V - \rffv_V}{v - \rffv_V}}\chi^{\rffv,
    V}_{v}\brac*{\zeta,\tau}\dd(v)
    &=\int_{V} \ee^{-2\pi \ii \pair{\gamma_V - \rffv_V}{v - \rffv_V}}
    \frac{\ee^{2\pi \ii \pair{v}{\zeta_V}}q^{\frac{1}{2} \norm{v -
                                                 \rffv_V}^2}}{\eta(\tau)^{\dim V}} \dd(v)\nonumber\\
    &= \frac{\ee^{\pi \ii \tau \norm{\rffv_V}^2} \ee^{2 \pi \ii \pair{\rffv_V}{\gamma_V - \rffv_V}}}{\eta(\tau)^{\dim V}}\sqrt{\det\pair{-}{-}\vert_V} \int_{\RR^{\dim V}} \ee^{\pi \ii \tau \pair{v}{v}} \ee^{2 \pi \ii \pair{v}{\zeta_V - \gamma_V + \rffv_V - \tau \rffv_V}}  \dd v_1\cdots \dd v_n
    \nonumber\\ &= \frac{\ee^{\pi \ii \tau \norm{\rffv_V}^2} \ee^{2 \pi \ii \pair{\rffv_V}{\gamma_V - \rffv_V}}}{\eta(\tau)^{\dim V} {\sqrt{- 2\pi \ii \tau}}^{\dim V}} \sqrt{\det\pair{-}{-}\vert_V}\int_{\RR^{\dim V}} \ee^{- \frac{1}{2} \pair{\nu}{\nu}} \ee^{\frac{2 \pi i}{\sqrt{- 2 \pi \ii \tau}}\pair{\nu}{\zeta_V - \gamma_V + \rffv_V - \tau \rffv_V}} d \nu
                  \nonumber\\
    &=\ee^{2 \pi \ii \pair{\zeta_V/\tau}{\rffv_V \tau -
                  \zeta_V/2 - \rffv_V}}
      \frac{\ee^{2
                  \pi \ii \pair{\zeta_V/\tau}{\gamma_V}} \ee^{-\pi i
                  \norm{\gamma_V - \rffv_V}^2 /\tau}}{\sqrt{\ii \tau}^{\dim V}
                  \eta(\tau)^{\dim V}}\nonumber\\
    &=\ee^{-2\pi \ii \brac*{\frac{\norm{\zeta_v}^2}{2\tau} + \frac{\pair{\rffv_v}{\zeta_V}}{\tau} - \pair{\rffv_V}{\zeta_V}}}\chi^{\rffv, V}_{\gamma}\brac*{\frac{\zeta}{\tau},\frac{-1}{\tau}},
  \end{align}
  where in the third equality we have rescaled all integration variables to
  absorb a factor of  \(\sqrt{-\ii2\pi\tau}\) and in the fourth equality we have used the Gaussian
  integral formula \eqref{eq:gint} with \(A\) equal to the Gram matrix of \(\pair{-}{-}\vert_V\).

  The \(D\) part, \(\chi^{\rffv,D}_{\gamma}\), is a lattice $\theta$-function
  (for the lattice \(D\cap\lat\)). The modular transformation properties of
  lattice \(\theta\)-functions are well known. In particular, the
  $S$-transformation can be determined by combining the Gaussian integral
  formula \eqref{eq:gint} with the Dirac comb, that is, the
  \(\delta\)-distribution identity
  \begin{equation}
    \sum_{k\in \ZZ}\delta\brac*{x-k} = \sum_{n\in\ZZ}\ee^{2\pi\ii nx}.
  \end{equation}
  Note that \(\sqrt{\det\pair{-}{-}\vert_D}=\sqrt{|D/D\cap\lat|^{-1}}\) and consider
  \begin{align}
    \chi^{\rffv,D}_{\gamma}\brac*{\frac{\zeta}{\tau},\frac{-1}{\tau}}&=\sum_{\kappa\in
    \lat\cap
    D}\frac{\ee^{2\pi\ii\pair{\zeta/\tau}{\gamma_D+\kappa}}\ee^{2\pi\ii\frac{-1}{2\tau}\pair{\gamma_D-\rffv_D+\kappa}{\gamma_D-\rffv_D+\kappa}}}{\eta(\tau)^{\rk
                                                                       D}}\nonumber\\
    &=\sum_{\kappa\in\lat\cap
      D}\ee^{2\pi\ii\brac*{\frac{\norm{\zeta_D}^2}{2\tau}+\frac{\pair{\rffv_D}{\zeta_D}}{\tau}-\pair{\rffv_D}{\zeta_D}}}\sqrt{\det\pair{-}{-}\vert_D}
      \int_{\rspn D}\hspace{-9mm}
      \ee^{-2\pi\ii\pair{\gamma_D-\rffv_D+\kappa}{\nu-\rffv_D}}\frac{\ee^{2\pi\ii\pair{\zeta_D}{\nu}}\ee^{2\pi\ii\frac{\tau}{2}\pair{\nu-\rffv_D}{\nu-\rffv_D}}}{\eta(\tau)^{\rk
      D}}\dd(\nu)\nonumber\\
    &=\frac{\ee^{2\pi\ii\brac*{\frac{\norm{\zeta_D}^2}{2\tau}+\frac{\pair{\rffv_D}{\zeta_D}}{\tau}-\pair{\rffv_D}{\zeta_D}}}}{\sqrt{|D/D\cap\lat|}}
      \int_{\rspn D}\hspace{-9mm}
      \ee^{-2\pi\ii\pair{\gamma_D-\rffv_D}{\nu-\rffv_D}}\frac{\ee^{2\pi\ii\pair{\zeta_D}{\nu}}\ee^{2\pi\ii\frac{\tau}{2}\pair{\nu-\rffv_D}{\nu-\rffv_D}}}{\eta(\tau)^{\rk
      D}}\hspace{-3mm}\sum_{\kappa\in\lat\cap
      D}\hspace{-2mm}\ee^{-2\pi\ii\pair{\kappa}{\nu-\rffv_D}}\dd(\nu)\nonumber\\
    &=\frac{\ee^{2\pi\ii\brac*{\frac{\norm{\zeta_D}^2}{2\tau}+\frac{\pair{\rffv_D}{\zeta_D}}{\tau}-\pair{\rffv_D}{\zeta_D}}}}{\sqrt{|D/D\cap\lat|}}
      \int_{\rspn D}\hspace{-9mm}
      \ee^{-2\pi\ii\pair{\gamma_D-\rffv_D}{\nu-\rffv_D}}\frac{\ee^{2\pi\ii\pair{\zeta_D}{\nu}}\ee^{2\pi\ii\frac{\tau}{2}\pair{\nu-\rffv_D}{\nu-\rffv_D}}}{\eta(\tau)^{\rk
      D}}\hspace{-3mm}\sum_{k\in D}\hspace{-2mm}\delta_D(\nu-k)\dd(\nu)\nonumber\\
    &=\frac{\ee^{2\pi\ii\brac*{\frac{\norm{\zeta_D}^2}{2\tau}+\frac{\pair{\rffv_D}{\zeta_D}}{\tau}-\pair{\rffv_D}{\zeta_D}}}}{\sqrt{|D/D\cap\lat|}}
      \sum_{k\in D}\ee^{-2\pi\ii\pair{\gamma_D-\rffv_D}{\nu-\rffv_D}}\frac{\ee^{2\pi\ii\pair{\zeta_D}{k}}\ee^{2\pi\ii\frac{\tau}{2}\pair{k-\rffv_D}{k-\rffv_D}}}{\eta(\tau)^{\rk
      D}}\nonumber\\
    &=\frac{\ee^{2\pi\ii\brac*{\frac{\norm{\zeta_D}^2}{2\tau}+\frac{\pair{\rffv_D}{\zeta_D}}{\tau}-\pair{\rffv_D}{\zeta_D}}}}{\sqrt{|D/D\cap\lat|}}
      \sum_{t\in D/\lat\cap D}\ee^{-2\pi\ii\pair{\gamma_D-\rffv}{t-\rffv}}
      \sum_{\kappa\in t}
      \frac{\ee^{\pair{\zeta_D}{\kappa}}\ee^{2\pi\ii\frac{\tau}{2}\norm{\kappa-\rffv_D}^2}}{\eta(\tau)^{\rk
      D}}\nonumber\\
    &=\frac{\ee^{2\pi\ii\brac*{\frac{\norm{\zeta_D}^2}{2\tau}+\frac{\pair{\rffv_D}{\zeta_D}}{\tau}-\pair{\rffv_D}{\zeta_D}}}}{\sqrt{|D/D\cap\lat|}}
      \sum_{t\in D/\lat\cap D}\ee^{-2\pi\ii\pair{\gamma_D-\rffv}{t-\rffv}}\chi^{\rffv,D}_{t}\brac*{\zeta,\tau}
  \end{align}

  Finally, we consider the $\nlat$, $F$ part, \(\chi^{\rffv,
    \circ}_{\gamma}\). Note first that, for \(\gamma_F\in F\) and \(\gamma_\circ\in \rspn\nlat\)
  \begin{align}
    \chi^{\rffv, \circ}_{\gamma_F+\gamma_\circ}\brac*{\frac{\zeta}{\tau},\frac{-1}{\tau}}&=
    \frac{\ee^{ 2\pi \ii \pair{\zeta_\circ/\tau}{\gamma_F}} \ee^{-2\pi \ii \pair{\rffv_\circ}{\rffv_F - \gamma_F} /\tau}}{\eta(\tau)^{2\rk \nlat} \sqrt{-i \tau}^{2\rk \nlat}} \sum_{f \in F}  \ee^{2 \pi \ii \pair{\gamma_\circ}{f}} \delta_F (\zeta_F / \tau - (\gamma_F - \rffv_F)/\tau - f)\nonumber \\ &= \frac{\abs{\tau}^{\rk \nlat}}{(-i\tau)^{\rk \nlat}} \frac{\ee^{2 \pi \ii \pair{\zeta_\circ / \tau}{\gamma_F}}}{\eta(\tau)^{2 \rk \nlat}} \ee^{-2 \pi \ii \pair{\rffv_\circ}{\rffv_F - \gamma_F}/\tau} \sum_{f \in F} \ee^{2 \pi \ii \pair{\gamma_\circ}{f}} \delta_F (\zeta_F + \rffv_F - \gamma_F - \tau f),
  \end{align}
  where the second equality follows from the scaling behaviour of \(\delta\)-distributions.
  Let $\set{a_i}_{i=1}^{\dim \nlat}$ be a $\ZZ$-basis of
  $\nlat\cap\lat$ and let \(\set{a^i}_{i=1}^{\dim \nlat}\) be its dual in
  \(F\). Then compare the above to
  \begin{align}
    \sum_{k\in F}&
      \int_{\rspn{\nlat}/\nlat} \ee^{- 2\pi \ii \pair{\gamma_F+\gamma_\circ -
                   \rffv_F-\rffv_\circ}{k+\nu -
                   \rffv_F-\rffv_\circ}}\chi^{\rffv,
                   \circ}_{k+\nu}(\zeta,\tau)\dd(\nu)\nonumber\\
     &=
     \sum_{k\in F} \int_{[0,1]^{\rk \nlat}}
    \ee^{- 2\pi \ii \pair{\gamma_F + \gamma_\circ - \rffv_{F} -
    \rffv_\circ}{k+\nu - \rffv_{F} - \rffv_\circ}} \frac{\ee^{2 \pi \ii
    \pair{\zeta_\circ}{k}}q^{\pair{\rffv_{\circ}}{\rffv_F-k}}}{\eta(\tau)^{2 \rk \nlat}} 
      \sum_{f \in F} \ee^{2 \pi \ii \pair{\nu_\circ}{f}} \delta_F(\zeta_F +
        \tau(k - \rffv_F) - f) \dd\nu_1\cdots\dd\nu_{\rk\nlat}\nonumber\\
                 &=
     \sum_{k\in F} 
    \ee^{- 2\pi \ii \pair{\gamma_F + \gamma_\circ - \rffv_{F} -
    \rffv_\circ}{k - \rffv_{F} - \rffv_\circ}} \frac{\ee^{2 \pi \ii
    \pair{\zeta_\circ}{k}}q^{\pair{\rffv_{\circ}}{\rffv_F -
    k}}}{\eta(\tau)^{2 \rk \nlat}} 
      \sum_{f \in F}  \delta_F(\zeta_F +
        \tau(k - \rffv_F) - f) \int_{[0,1]^{\rk \nlat}}\ee^{2 \pi \ii
        \pair{\nu}{\rffv_F-\gamma_F+f}}\dd\nu_1\cdots\dd\nu_{\rk\nlat}\nonumber\\
                 &=
    \sum_{k\in F} 
    \ee^{- 2\pi \ii \pair{\gamma_F + \gamma_\circ - \rffv_{F} -
    \rffv_\circ}{k - \rffv_{F} - \rffv_\circ}} \frac{\ee^{2 \pi \ii
    \pair{\zeta_\circ}{k}}q^{\pair{\rffv_{\circ}}{\rffv_F -
    k}}}{\eta(\tau)^{2 \rk \nlat}}\delta_F(\zeta_F +
                   \tau(k - \rffv_F) + \rffv_F-\gamma_F)\nonumber\\
                 &=
    \sum_{k\in F} 
    \ee^{- 2\pi \ii \pair{\gamma_F + \gamma_\circ - \rffv_{F} -
    \rffv_\circ}{-k - \rffv_\circ}} \frac{\ee^{2 \pi \ii
    \pair{\zeta_\circ}{\rffv_F-k}}q^{\pair{\rffv_{\circ}}{k}}}{\eta(\tau)^{2 \rk \nlat}}\delta_F(\zeta_F +
                   \rffv_F-\gamma_F - \tau k)\nonumber\\
                 &=
     \ee^{2\pi\ii
                   \pair{\zeta_F+\zeta_\circ}{\rffv_f+\rffv_\circ}}\ee^{-2\pi\ii\brac*{\frac{\norm{\zeta_F+\zeta_\circ}^2}{2\tau}+\frac{\pair{\zeta_F+\zeta_\circ}{\rffv_F+\rffv_\circ}}{\tau}}}
                   \frac{\ee^{\frac{2\pi\ii}{\tau}\pair{\gamma_F}{\zeta_\circ}}\ee^{\frac{2\pi\ii}{\tau}\pair{\rffv_\circ}{\gamma_F-\rffv_F}}}{\eta(\tau)^{2\rk
                   \nlat}}\sum_{k\in
                   F}\ee^{2\pi\ii\pair{\gamma_\circ}{k}}\delta_F\brac*{\zeta_F+\rffv_F-\gamma_F-\tau
                   k}\nonumber\\
    &=
     \ee^{2\pi\ii
      \pair{\zeta_F+\zeta_\circ}{\rffv_F+\rffv_\circ}}\ee^{-2\pi\ii\brac*{\frac{\norm{\zeta_F+\zeta_\circ}^2}{2\tau}+\frac{\pair{\zeta_F+\zeta_\circ}{\rffv_F+\rffv_\circ}}{\tau}}}
      \brac*{\frac{-\ii\tau}{|\tau|}}^{\rk\nlat}\chi^{\rffv, \circ}_{\gamma_F+\gamma_\circ}\brac*{\frac{\zeta}{\tau},\frac{-1}{\tau}}
  \end{align}
\end{proof}

\begin{rmk}
  Armed with the above \(S\)-transformations, we can now propose a Verlinde
  formula following the standard module formalism of \cite{RidSL208} by
  setting
  \begin{equation}
    N_{\lambda,\mu}^\rho=\int_{\qlat}
    \frac{S_{\lambda,\kappa}S_{\mu,\kappa}\overline{S_{\rho,\kappa}}}{S_{0,\kappa}}\dd \kappa,\qquad \lambda,\mu,\rho\in\qlat
  \end{equation}
  and asking, if
  \begin{equation}
    \LFock{\lambda}\fuse\LFock{\mu}\cong\int_{\qlat} N_{\lambda,\mu}^\rho \LFock{\rho}\dd \rho
  \end{equation}
  holds. Indeed a quick calculation reveals that
  \begin{align}
    N_{\lambda,\mu}^\rho&=\det\pair{-}{-}\vert_D
                          \int_{\qlat}\ee^{- 2 \pi \ii \pair{\lambda+\mu-\rho}{\kappa-\rffv}}\dd{\kappa}
                          =\int_{\qlat}\ee^{- 2 \pi \ii \pair{\lambda+\mu-\rho}{\kappa}}\dd{\kappa}\nonumber\\
                        &=\delta(\lambda_V+\mu_V-\rho_v)\delta_{\lambda_F+\mu_F,\rho_F}\brac*{\sum_{\phi\in
                          \nlat}\delta_\nlat(\lambda_\circ+\mu_\circ-\rho_\circ-\phi)}\delta_{\lambda_D+\mu_D=\rho_D},
  \end{align}
  which are of course precisely the fusion multiplicities of
  \begin{equation}
    \LFock{\lambda}\fuse \LFock{\mu}= \LFock{\lambda+\mu}.
  \end{equation}
\end{rmk}

\vspace{20mm}
%

\end{document}